\documentclass[12pt,a4paper]{article}

\usepackage[british]{babel}

\usepackage[T1]{fontenc}

\usepackage[utf8]{inputenc}

\usepackage{csquotes}

\usepackage[backend=biber, maxbibnames=10, maxcitenames=4, maxalphanames=4, bibencoding=utf8, style=alphabetic, isbn=false, url=false, doi=true]{biblatex}
\renewbibmacro{in:}{}
\renewbibmacro*{doi+eprint+url}{%
	\printfield{doi}%
	\newunit\newblock%
	\iftoggle{bbx:eprint}{%
		\usebibmacro{eprint}%
	}{}%
	\newunit\newblock%
	\iffieldundef{doi}{%
		\usebibmacro{url+urldate}}%
	{}%
}

\usepackage{setspace}

\usepackage{lmodern}

\usepackage[indentafter]{titlesec}
\titleformat{name=\section}{}{\thetitle.}{0.8em}{\centering\scshape}
\titleformat{name=\subsection}[runin]{}{\thetitle.}{0.5em}{\bfseries}[.]
\titleformat{name=\subsubsection}[runin]{}{\thetitle.}{0.5em}{\itshape}[.]
\titleformat{name=\paragraph,numberless}[runin]{}{}{0em}{}[.]
\titlespacing{\paragraph}{0em}{0em}{0.5em}
\titleformat{name=\subparagraph,numberless}[runin]{}{}{0em}{}[.]
\titlespacing{\subparagraph}{0em}{0em}{0.5em}

\usepackage{url}

\usepackage{mathtools}

\DeclarePairedDelimiter\abs{\lvert}{\rvert}%
\DeclarePairedDelimiter\norm{\lVert}{\rVert}%

\makeatletter
\let\oldabs\abs
\def\abs{\@ifstar{\oldabs}{\oldabs*}}
\let\oldnorm\norm
\def\norm{\@ifstar{\oldnorm}{\oldnorm*}}
\makeatother

\usepackage{amssymb}

\usepackage{amsthm}

\usepackage{mathrsfs}

\usepackage{dsfont}

\usepackage{enumitem}
\setlist[enumerate]{noitemsep, partopsep=0pt, topsep=0pt, parsep=0pt, itemsep=0pt}
\setlist[itemize]{noitemsep, partopsep=0pt, topsep=0pt, parsep=0pt, itemsep=0pt}

\usepackage{interval}

\intervalconfig{soft open fences}

\usepackage[normalem]{ulem}

\usepackage[stretch=10]{microtype}

\usepackage[affil-it, noblocks]{authblk}

\usepackage{xfrac}

\usepackage{xcolor}

\usepackage[hidelinks,draft=false]{hyperref}
\hypersetup{
  colorlinks   = true, 
  urlcolor     = blue, 
  linkcolor    = blue, 
  citecolor   = red 
}

\usepackage[nameinlink,noabbrev,capitalize]{cleveref}
\crefname{equation}{}{}

\usepackage[plain]{fullpage}

\newlist{theoenum}{enumerate}{1} 
\setlist[theoenum]{label=\normalfont(\roman*), ref=\theproposition~\normalfont(\roman*), noitemsep, partopsep=0pt, topsep=0pt, parsep=0pt, itemsep=0pt}
\crefalias{theoenumi}{theorem}

\usepackage{subcaption}

\usepackage{tocloft}

\usepackage{titlefoot}				

\theoremstyle{plain}
\newtheorem{theorem}{Theorem}
\newtheorem*{theorem*}{Theorem}
\newtheorem{proposition}[theorem]{Proposition}

\newtheorem{lemma}[theorem]{Lemma}
\newtheorem*{lemma*}{Lemma}

\theoremstyle{definition}
\newtheorem{definition}[theorem]{Definition}

\theoremstyle{remark}

\newtheorem{remark}[theorem]{Remark}
\newtheorem*{remark*}{Remark}

\newtheoremstyle{break}
{}
{}
{\itshape}
{}
{\bfseries}
{.}
{\newline}
{}

\theoremstyle{break}
\newtheorem{manualtheoreminner}{Theorem}
\newenvironment{manualtheorem}[1]{%
	\renewcommand\themanualtheoreminner{#1}%
	\manualtheoreminner
}{\endmanualtheoreminner}			

\newcommand*\diff{\mathop{}\!\mathrm{d}}
\newcommand{\R}{\mathds{R}}

\newcommand{\Z}{\mathds{Z}}

\renewcommand{\exp}{\mathrm{exp}}
\renewcommand{\phi}{\varphi}
\renewcommand{\epsilon}{\varepsilon}

\DeclareMathOperator{\T}{\mathrm{T}}

\DeclareMathOperator{\sgn}{sgn}

\newcommand{\J}{\mathcal{J}}
\newcommand{\m}{\mathfrak{m}}

\newcommand{\K}{\mathcal{K}}

\usepackage{todonotes}

\date{}
\title{\textbf{\uppercase{\large{
Measure contraction property, curvature exponent and geodesic dimension of sub-Finsler \texorpdfstring{$\ell^{\MakeLowercase{p}}$}{ellp}-Heisenberg groups
}}}}
\author{Samuël Borza\thanks{\href{mailto:sborza@sissa.it}{sborza@sissa.it}} }
\affil[]{Scuola Internazionale Superiore di Studi Avanzati (SISSA), Trieste}

\author{Kenshiro Tashiro\thanks{\href{mailto: kenshiro.k.tashiro@jyu.fi}{kenshiro-tashiro@oist.jp}} }
\affil[]{Okinawa Institute of Science and Technology}					

\begin{document}

\begin{spacing}{0.6}
\maketitle	
\end{spacing}

\providecommand{\keywords}[1]
{
	\unmarkedfntext{\textbf{\textit{Keywords---}} #1}
}

\providecommand{\msc}[1]
{
	\unmarkedfntext{\textbf{\textit{MSC (2020)---}} #1}
}

\vspace{-0.5cm}\begin{abstract}
    We initiate the study of synthetic curvature-dimension bounds in sub-Finsler geometry. More specifically, we investigate the measure contraction property $\mathsf{MCP}(K, N)$, and the geodesic dimension on the Heisenberg group equipped with an $\ell^p$-sub-Finsler norm.
    We show that for $p\in(2,\infty]$,
    the $\ell^p$-Heisenberg group fails to satisfy any of the measure contraction properties.
    On the other hand,
    if $p\in(1,2)$,
    then it satisfies the measure contraction property $\mathsf{MCP}(K, N)$ if and only if $K \leq 0$ and $N \geq N_p$, where the curvature exponent $N_p$ is strictly greater than $2q+1$ ($q$ being the H\"older conjugate of $p$).
    We also prove that the geodesic dimension of the $\ell^p$-Heisenberg group is $\min(2q+2,5)$ for $p\in[1,\infty)$.
    As a consequence,
    we provide the first example of a metric measure space where there is a gap between the curvature exponent and the geodesic dimension.
\end{abstract}
\keywords{Sub-Finsler geometry, Heisenberg group, Optimal control, Generalized trigonometric \hspace*{8.1em}functions}

\msc{53C17, 26A33, 49N60, 49Q22}				

\begin{spacing}{0.9}
\setlength{\cftbeforesecskip}{0pt}
\vspace{-0.4cm}\tableofcontents
\end{spacing}

\section{Introduction}

Sub-Finsler geometry is a broad generalization of Riemannian geometry that encompasses Finsler geometry and sub-Riemannian geometry. The Heisenberg group, the prototype example of sub-Riemannian geometry, can be endowed with a sub-Finsler metric modelled on the $\ell^p$-norm of the Euclidian space. This $\ell^p$-Heisenberg group is the focus of this work. We will study it from the point of view of the analysis of metric measures spaces, investigating the validity of a synthetic curvature-dimension condition and computing a notion of dimension relevant in metric geometry.

Synthetic definitions of lower curvature bounds based on the theory of optimal transport have been introduced to take into account non-smooth metric spaces. Roughly speaking, the optimal way to transport a mass, i.e. a probability measure, to another one in a metric measure space is characterised by a Wasserstein geodesic on the set of probability measures. The seminal contributions by Lott--Villani and Sturm showed in \cite{lott--villani2009} and \cite{sturm2006-1,sturm2006} independently that in Riemannian geometry, a lower bound $K$ on the Ricci curvature and an upper bound $N$ on the topological dimension are equivalent to a type of $K$-convexity for the
$N$-entropy along the Wasserstein geodesics. This alternative characterisation of curvature-dimension bounds can be expressed solely in terms of the Riemannian distance and volume, without explicitly referencing a differential structure. This observation leads to the introduction of the \textit{curvature dimension condition} $\mathsf{CD}(K, N)$ on general metric measure spaces.

The equivalence between a curvature-dimension condition and an entropic convexity along Wasserstein geodesics can be extended to Finsler manifolds with sufficient regularity. In this setting, Ohta showed in \cite{ohta2009} that $\mathsf{CD}(K,N)$ is equivalent to having a lower bound $K$ of the $N$-weighted Ricci curvature. Here, the $N$-weighted Ricci curvature is a tensor that depends on a given smooth volume form since, in Finsler geometry, there is a priori no canonical smooth measure.

In sub-Riemannian geometry, the curvature-dimension condition is known to fail. The first result in this direction was given by Juillet in \cite{juillet2009} which proved that the Heisenberg group equipped with its sub-Riemannian structure and the Lebesgue measure does not satisfy $\mathsf{CD}(K, N)$ for any $K \in \mathds{R}$ and any $N \geq 1$. The intuition behind this is that if we view the Heisenberg group as a Gromov--Hausdorff limit of a sequence of Riemannian manifolds, its Ricci curvature will diverge, indicating that they are not so called \textit{Ricci limit spaces}. This result has then been generalised to all sub-Riemannian manifolds in \cite{juillet2020} (for strict sub-Riemannian structures) and in \cite{rizzi2023failure} (for strictly positive smooth measures).

Many examples of sub-Riemannian manifolds, however, are known to satisfy a relaxation of the curvature-dimension condition, called the \textit{measure contraction property} introduced by Ohta in \cite{ohta2007} (see \cref{MCPnegligible}). Given a metric measure space $(X, \diff, \mathfrak{m})$ and a Borel subset $\Omega \subseteq X$ with $0 < \mathfrak{m}(\Omega) < +\infty$, the $t$-geodesic homothety $\Omega_t$ from a point $x \in X$ is the set of all $t$-intermediate points $\gamma(t)$ of all minimising constant speed geodesics $\gamma : \interval{0}{1} \to X$ joining $x$ to $\Omega$. The measure contraction property $\mathsf{MCP}(K, N)$, where $K$ is still meant to represent a synthetic ``lower bound on the Ricci curvature'' and $N$ an ``upper bound on the dimension'', consists of a type of convexity of the map $\mathfrak{m}(\Omega_{\boldsymbol{\cdot}}) : \interval{0}{1} \to \mathds{R}$. When $K = 0$, and if $(X, \diff, \mathfrak{m})$ has negligible cut locus (see \cref{def:negligiblecutlocus}), then the $\mathsf{MCP}(0, N)$ is equivalent to $\mathfrak{m}(\Omega_t) \geq t^N \mathfrak{m}(\Omega)$. Juillet demonstrated in \cite{juillet2009} that the three-dimensional Heisenberg group satisfies the measure contraction property $\mathsf{MCP}(0,5)$. Moreover, the pair $(K, N) = (0, 5)$ is optimal in the sense that the Heisenberg group does not satisfy $\mathsf{MCP}(K, N)$ whenever $K > 0$ or $N < 5$. The measure contraction property has also been found to hold in ideal Carnot groups \cite{Rifford2013}, corank 1 Carnot groups \cite{rizzi2016}, in generalised H-type Carnot groups \cite{sharpcarnot2018}, or in the Grushin plane \cite{barilari-rizzi2019}. The optimal constant $N$, i.e. the infimum one, such that $\mathsf{MCP}(0, N)$ holds is called the \textit{curvature exponent}. In the present work, we investigate the measure contraction property and the curvature exponent in sub-Finsler geometry for the first time. It is worth noting that in Finsler geometry, the $\mathsf{MCP}(K, N)$ condition is implied by having a lower bound $K$ on the $N$-Ricci curvature, as shown in \cite[Theorem 1.2]{ohta2009}.

The \textit{geodesic dimension}, on the other hand, is determined by the asymptotic rate of growth of the volume of measurable set under geodesic homotheties. When a given metric measure space has sufficient regularity, it is the optimal $N$ such that $\mathfrak{m}(\Omega_t) \sim t^N$ as $t \to 0^+$.
Since the geodesic dimension considers the asymptotic behavior of $\Omega_t$ as $t\to 0^+$,
it explains more local features than the measure contraction property.
It is related with the curvature exponent:
it was shown in \cite{rizzi2016} that the geodesic dimension is always a lower bound for the curvature exponent (see \cref{theo:NcurvNgeoNH}).
Moreover, in many cases, see \cite{Rifford2013,rizzi2016,sharpcarnot2018,barilari-rizzi2019}, the curvature exponent is equal to the geodesic dimension.

In this work, we initiate the study of synthetic curvature-dimension bounds and geodesic dimension in sub-Finsler geometry. We examine the $\ell^p$-Heisenberg group for any $p \in \interval{1}{\infty}$, which consists of the Heisenberg group equipped with an $\ell^p$-sub-Finsler norm and the Lebesgue measure. It is a generalisation of the sub-Riemannian Heisenberg group, which is the $\ell^2$-Heisenberg group in our notation. The curvature exponent and the geodesic dimension of the usual sub-Riemannian Heisenberg group are known to be $5$, see \cite{juillet2009}. The choice to focus on this family of sub-Finsler spaces is motivated by the fact that they exhibit a variety of behaviors that will be seen to be determining factors in relation with the measure contraction property, the curvature exponent, and the geodesic dimension. The $\ell^1$- and $\ell^\infty$-Heisenberg groups have \textit{branching} geodesics, and have large (non-negligible) cut locus. The $\ell^p$-Heisenberg group have $C^2$ geodesics if $p < 2$, while they are only $C^1$ if $p > 2$.

In the strictly convex case, i.e., when $p \in \ointerval{1}{\infty}$, the main techniques are the following. By Pontryagin's Maximum Principle, we are able to write an exponential map $\exp^t_x$, that is, the projection of the Hamiltonian flow at time $t$ from $x$. It is then used to estimate the volume of the $t$-geodesic homothety
\[
\mathfrak{m}(\Omega_t) = \int_{\exp^{-1}(\Omega)} \mathrm{Jac} \circ \exp^t_x \diff \mathcal{L}^3,
\]
where the integrand is the Jacobian determinant of the exponential map $\exp^t_x$ at time $t$. The definition of the $\mathsf{MCP}(K, N)$ is then shown to be equivalent to a differential inequality on $\mathrm{Jac} \circ \exp^t$ (see \cref{proposition:mcp}). In Finsler or sub-Finsler geometry, when it can be well-defined, the exponential map is typically not smooth. In our case, it will be smooth except for a negligible set, and studying what happens near non smooth points will be crucial.

The geometry underlying the structure of the $\ell^p$-Heisenberg group is described in \cref{sec2} using special trigonometric functions known as Shelupsky's $p$-trigonometric functions. They can be defined geometrically, as in \cite{lok}, or seen to be solutions to a kind of $1$-dimensional $q$-Laplace equation, where $q := (1-1/p)^{-1}$ is the Hölder conjugate of $p$, which is studied in \cite{paredes-uchiyama,edmunds2012,girg-kotrla2014} and the references therein. We analyse the geometric and geodesic structure of the $\ell^p$-Heisenberg group in \cref{sec:geometry}. As a byproduct of our work, we establish the cotangent injectivity radius of the $\ell^p$-Heisenberg in \cref{cotinjdom}. The Jacobian determinant is written out explicitly in \cref{regularityF}, and we study it comprehensively in \cref{sec:jacobian}.

Regarding the curvature exponent of the $\ell^p$-Heisenberg group,
we prove following results.

\begin{manualtheorem}{A}[Curvature exponent for $p\in(1,\infty) \setminus \{2\} $, \cref{thm:MCPp>2} and \cref{thm:MCPp<2}]
\label{maintheoremA}

\begin{minipage}[t]{\linewidth}
\vspace{-0.3cm}\begin{enumerate}[label=\normalfont\arabic*)]
    \item When $p \in \ointerval{1}{2}$, the curvature exponent $N_p$ of the sub-Finsler Heisenberg group equipped with the Lebesgue measure is finite and greater than $2q+1$.
    In other words,
    it satisfies $\mathsf{MCP}(K,N)$ if and only if $K\leq 0$ and $N\geq N_p>2q+1$.
    \item When $p \in \ointerval{2}{\infty}$, the sub-Finsler $\ell^p$-Heisenberg group equipped with the Lebesgue measure does not satisfy the $\mathsf{MCP}(K, N)$ for any $K \in \mathds{R}$ and any $N \geq 1$.
    \end{enumerate}
\end{minipage}
\end{manualtheorem}

Note that the case $p = 2$ was studied in \cite{juillet2009}: the $\ell^2$-Heisenberg group satisfies $\mathsf{MCP}(K, N)$ if and only if $K \leq 0$ and $N \geq 5$. Notice that we do not know the exact value of the curvature exponent $N_p$ when $p \in \ointerval{1}{2}$ and the value $2 q + 1$ is a non-optimal lower bound on $N_p$. The failure of the measure contraction property for $p > 2$ is linked to the loss of regularity of the exponential map. Geometrically, when $p > 2$, geodesics from a given point have $C^1$-corners, and the $t$-homothety $\Omega_t$ infinitesimally shrinks to a small subset when passing through these corners.
This is also observed from the fact that the derivative of the Jacobian determinant diverges to the infinity at the $C^1$-corner, see \cref{regularitydF}.
When $p < 2$, the exponential map is $C^2$ which is enough regularity to ensure that the measure contraction property holds.
However,
its dual $q$-norm fails to be strongly convex,
and the curvature of a geodesic attains $0$ at non-smooth corners,
seen as a curve in the Euclidean space.
The lower bound $2q+1$ for the curvature exponent is observed along the horizontal line geodesic tangent to such non-smooth corners.
Indeed, we can observe that the $t$-homothety $\Omega_t$ along such a straight horizontal geodesic is small as $t$ goes to $0$.
This observation reflects the fact that the Jacobian determinant and its differental vanishes on such straight horizontal lines,
see \cref{asymptoticFat00} and \cref{regularitydF2}.
However, this lower bound $2q+1$ is not the curvature exponent.
Indeed, we will see that an exponent greater than $2 q + 1$ is always required along such horizontal lines if $p\neq 2$ (see the end of the proof of \cref{thm:MCPp<2} and \cref{fig:graphsdPP}).
Furthermore,
it appears numerically that the optimal exponent is not realized on horizontal lines, see \cref{fig:logderneg}.
If confirmed, this would represent a significant contrast to the sub-Riemannian context, where the curvature exponent has been consistently observed along horizontal lines. We leave open the question of which geodesics and at what value the curvature exponent is truly realized in the $\ell^p$-Heisenberg group.

As for the geodesics dimension, we have the following.

\begin{manualtheorem}{B}[Geodesic dimension for $p\in(1,\infty)$, \cref{thm:geodim}]

\begin{minipage}[t]{\linewidth}
\vspace{-0.3cm}\begin{enumerate}[label=\normalfont\arabic*)]
    \item When $p \in \linterval{1}{3}$, the geodesic dimension of the sub-Finsler $\ell^p$-Heisenberg group equipped with the Lebesgue measure is $5$.
    \item When $p \in \rinterval{3}{\infty}$, the geodesic dimension of the sub-Finsler $\ell^p$-Heisenberg group equipped with the Lebesgue measure is $2 q + 2$.
    \end{enumerate}
\end{minipage}
\end{manualtheorem}

While the measure contraction property relies on the differentiability properties of the Jacobian determinant of the exponential map, the geodesic dimension is related to its discontinuous singularity. If $p>2$, then the Jacobian determinant is discontinuous and diverges to infinity on the horizontal line tangent to the non-smooth corner (see \cref{asymptoticFat00}). The expansion of $\mathfrak{m}(\Omega_t)$ will be shown to have two competing leading terms as $t \to 0^+$: one of order $t^5$ outside this horizontal line,
or one of order $t^{2q+2}$ along it. The dominant term will be determined by whether $p$ is greater or smaller than 3. To derive the dominant term,
the full series expansion of the $p$-trigonometric functions from \cite{paredes-uchiyama} will be needed.

Finally, in \cref{sec:l1linfty}, we will consider the cases of $p = 1$ and $p = \infty$. Note that the $\ell^1$- and $\ell^\infty$-Heisenberg group are isometric metric measure spaces, so we only need to consider one of them.

\begin{manualtheorem}{C}[$\mathsf{MCP}$ and geodesic dimension for $p=1,\infty$, \cref{thm:ell1mcp} and \cref{thm:ell1geodim}]
    The $\ell^1$-Heisenberg group (resp. the $\ell^\infty$-Heisenberg group) equipped with the Lebesgue measure does not satisfy $\mathsf{MCP}(K,N)$ for all $K\in\R$ and $N\geq 1$. Its geodesic dimension is $4$, that is, the same as its Hausdorff dimension.
\end{manualtheorem}

Because of the fact that the norms $\ell^1$ or $\ell^\infty$ are convex but not strictly convex, the corresponding sub-Finsler structures on the Heisenberg group is highly branching, and has non-negligible cut locus.
Therefore, an exponential map can not be well-defined as a diffeomorphism onto a full-measure subset of the Heisenberg group.
However, since the norm is polygonal,
its geodesics can be explicitly written by using a quadratic polynomial (see \cite{breuillard--ledonne2013,duc} for a treatment of sub-Finsler Heisenberg groups with polygonal norms).
By using an explicit formula for the geodesics,
we can quantitatively prove that the branching property is the main reason why the measure contraction property is not satisfied: there exists a set $\Omega$ of positive measure such that $\mathfrak{m}(\Omega_t)$ vanishes for sufficiently small $t > 0$. Moreover, it has a large cut locus, which can be written out with quadratic polynomials. 
As will be seen in the proof of \cref{thm:ell1geodim}, this will force the geodesic dimension of the space to be as small as it possibly can,
and this is why the geodesic dimension attains its minimum value $4$,
which is the Hausdorff dimension of the Heisenberg group. Indeed, Theorem 6 of \cite{rizzi2016} shows that the Hausdorff dimension is a lower bound on the curvature exponent.

The results of this work, summarised in \cref{fig:summary}, are both unexpected and new. They deviate from what would typically be expected in sub-Riemannian geometry, where the curvature exponent and the geodesic dimension are consistently observed to be equal, as mentioned above. In the case of strictly sub-Finsler $\ell^p$-Heisenberg groups (i.e., when $p \neq 2$), we show that there is always a gap between the geodesic dimension and the curvature exponent. Moreover, it neatly differs from Finsler geometry as well. For instance, in Finsler spaces $(\mathds{R}^n, \|\cdot\|)$ equipped with any norm $\|\cdot\|$, the curvature exponent and geodesic dimension both equal $n$.

We anticipate that numerous specificities observed in the case of the $\ell^p$-Heisenberg group, such the impact of non-smoothness or branchingness, will constitute shared characteristics of sub-Finsler geometry. We would also like to emphasize the fact that, although the $\ell^p$-Heisenberg group appears to be a natural extension of sub-Riemannian geodesics, we have discovered that it is far from being a straightforward generalization. In addition to being considerably more technical, there are profound underlying differences (albeit the fact that the $p$-trigonometry functions look notationally similar the usual sine and cosine functions).

\begin{figure}
    \centering
    {\scalebox{0.50}{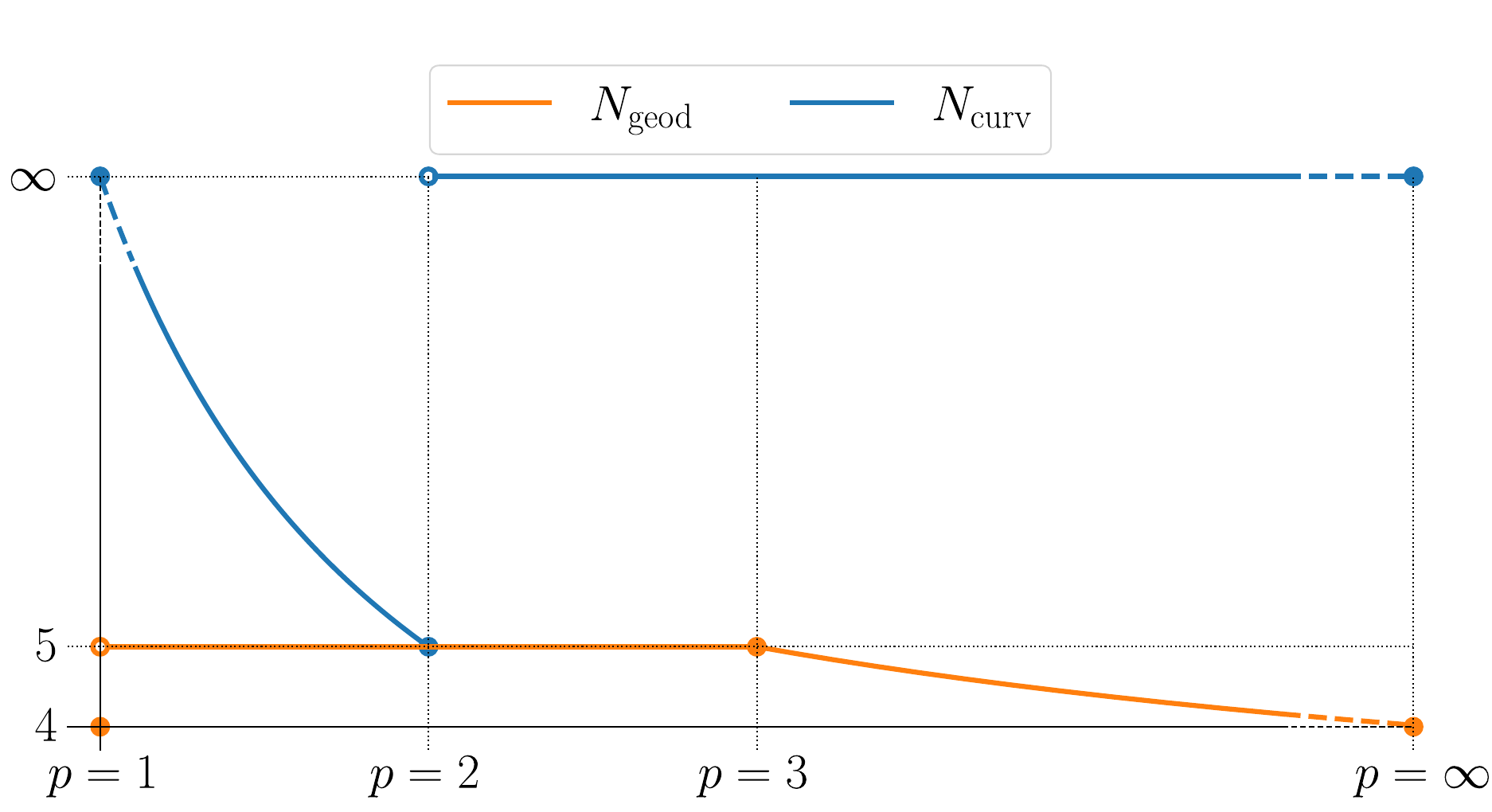}}
    \caption{A summary of the main theorems of this work on the curvature exponent and the geodesic dimension of the $\ell^p$-Heisenberg groups.}
    \label{fig:summary}
\end{figure}

\section*{Acknowledgements}
The authors would like to thank Luca Rizzi for introducing them to this problem and for the numerous stimulating discussions and comments.
The authors would also like to thank Andrei Agrachev, Enrico Le Donne and his postdocs/PhD students for many helpful discussions, and Peter Lindqvist as well as Koichi Uchiyama for valuable comments about the generalised trigonometric functions. We thank the anonymous reviewers for their careful reading of the manuscript. Their comments and suggestions helped improve and clarify this work.

This project has received funding from the European Research Council (ERC) under the European Union’s Horizon 2020 research and innovation programme (grant agreement No. 945655),
and the Academy of Finland
(grant 322898 ‘Sub-Riemannian Geometry via Metric-geometry and Lie-group theory’).
The second author is also supported by the  Japan Society for the Promotion of Science (JSPS) KAKENHI.		

\section{Preliminaries}\label{sec2}

\subsection{Measure contraction property and geodesic dimension}

\label{sec:MCP}

On a Riemannian manifold $(M, \diff_g, \mathrm{vol}_g)$, the condition 
\begin{equation}
    \label{RCurvBounds}
    \mathrm{Ric} \geq K, \text{ and} \dim(M) \leq N
\end{equation}
is often a natural assumption taken to prove significant theorems such as Bonnet-Myers theorem, Bishop-Gromov inequality, and Lévy-Gromov's isoperimetric inequality. The measure contraction property $\mathsf{MCP}(K, N)$ is one of the popular synthetic notions of curvature-dimension bounds generalising \cref{RCurvBounds} to abstract metric measure spaces. For reasons that we will soon see, this is the condition that we investigate in the present work for the Heisenberg group equipped with a sub-Finsler structure.

Defining the measure contraction property in full generality would require some understanding of optimal transport. In order to make this work more concise, we will therefore only provide an equivalent definition for spaces that have a negligible cut locus. For a more detailed explanation of the measure contraction property, we recommend that the reader refer to \cite{ohta2007} and \cite{sturm2006}.

For the rest of this section, $(X,\diff,\mathfrak{m})$ will denote a geodesic metric measure space.

\begin{definition}
    \label{def:negligiblecutlocus}
    We say that $(X,\diff,\mathfrak{m})$ has negligible cut locus if for any $x \in X$, there exists a negligible set $\mathcal{C}(x)$ and a measurable map $\Phi_x : X \setminus \mathcal{C}(x) \times \interval{0}{1} \to M$, such that the curve $\gamma : \interval{0}{1} \to M : t \mapsto \Phi_x(y, t)$ is the unique minimising geodesic between $x$ and $y$.
\end{definition}

For a bounded Borel set $\Omega \subseteq X$ such that $0 < \mathfrak{m}(\Omega) < +\infty$ and $t \in \interval{0}{1}$, the $t$-intermediate set $\Omega_t$ from a point $x \in X$ is the set
\[
\Omega_{t}:=\left\{\gamma(t) \mid \gamma : \interval{0}{1} \to X \text{ constant speed minimising geodesic}, \gamma(0) = x, \gamma(1) \in \Omega\right\}.
\]
It is not difficult to see that if $(X, \diff, \mathfrak{m})$ has negligible cut locus, then the $t$-intermediate set can be expressed, up to a set of measure zero, as
\[
\Omega_{t}=\left\{\Phi_x(y, t) \mid y \in \Omega \setminus \mathcal{C}(x)\right\}.
\]



\begin{definition}
    \label{def:branching}
    We say that $(X, \diff, \mathfrak{m})$ is \textit{non-branching} if two minimising constant speed geodesics $\gamma_1, \gamma_2 : \interval{0}{1} \to X$ are identically equal whenever $\gamma_1|_{\interval{a}{b}} = \gamma_2|_{\interval{a}{b}}$ for some $\interval{a}{b} \subseteq \interval{0}{1}$.
\end{definition}



For any $K \in \mathds{R}$, define the function
\[
\mathsf{s}_K(t):=
    \begin{cases}
    \sin(\sqrt{K}t)/\sqrt{K}, & \text{if}\ K > 0 \\
      t, & \text{if}\ K = 0 \\
      \sinh(\sqrt{-K}t)/\sqrt{-K}, & \text{if}\ K < 0
    \end{cases}.
\]

\begin{definition}[{\cite[Lemma 2.3]{ohta2007}}]
    \label{MCPnegligible}
    Let $K \in \mathds{R}$ and $N > 1$, or $K \neq 0$ and $N = 1$. A geodesic metric measure space $(X,\diff,\mathfrak{m})$ with negligible cut locus satisfies the \textit{$(K, N)$-measure contraction property}, or $\mathsf{MCP}(K, N)$, if for all $x \in X$, for all measurable set $\Omega \subseteq X$ with $0 < \mathfrak{m}(\Omega) < 0$ (and $\Omega \subseteq B(x, \pi \sqrt{(N-1)/K})$, and for all $t \in \interval{0}{1}$, it holds
\begin{equation}
    \label{def:MCP}
    \mathfrak{m}(\Omega_{t}) \geq \int_{\Omega} \left[\frac{\mathsf{s}_K(t \diff(x,y)/\sqrt{N-1})}{\mathsf{s}_K(\diff(x,y)/\sqrt{N-1})}\right]^{N-1} \diff\mathfrak{m}(y),
\end{equation}
where by convention $0/0 = 1$, and the term in square bracket is 1 if $K \leq 0$ and $N = 1$.
\end{definition}

\begin{remark}
    The condition $\Omega \subseteq B(x, \pi \sqrt{(N-1)/K})$ if $K > 0$ is reminiscent of the computation of the distortion coefficients for the $N$-dimensional model space with constant curvature $K$ (see for example \cite[Definition 14.19 and Theorem 14.20]{villanioldandnew}).
\end{remark}

The general definition of the measure contraction property can be found in \cite[Definition 2.1]{ohta2007}. We mention a few of its features, that are valid in general, not only for spaces with negligible cut locus. As proven in \cite[Corollary 3.3]{ohta2007}, a Riemannian manifold $(M, \diff_g, \mathrm{vol}_g)$ with $\mathrm{Ric} \geq K$ and $\dim(M) \leq N$ satisfies the $\mathsf{MCP}(K,N)$. Conversely, if $(M, \diff_g, \mathrm{vol}_g)$ satisfies the $\mathsf{MCP}(K,N)$ then $\dim(M) \leq N$ and if the $\mathsf{MCP}(K,N)$ holds with $N = \dim(M)$ then $\mathrm{Ric} \geq K$. The $\mathsf{MCP}(K, N)$ for Riemannian manifolds will generally not imply that $\mathrm{Ric} \geq K$, as remarked in \cite[Remark 5.6]{sturm2006}. If $(X, \diff, \mathfrak{m})$ satisfies the $\mathsf{MCP}(K, N)$, then it is shown in \cite[Lemma 2.4]{ohta2007} that it also satisfies the $\mathsf{MCP}(K', N')$ for $K' \leq K$ and $N' \geq N$. If the metric measure space satisfies $\mathsf{MCP}(K, N)$ with $K > 0$, then it must be compact. This is Bonnet-Myers' theorem, see \cite[Theorem 4.3]{ohta2007}.

The class of $\mathsf{MCP}(0, N)$-spaces is notable since it is designed to correspond to spaces with nonnegative curvature. Since sub-Finsler Carnot groups appear as metric tangents to sub-Finsler manifolds (see \cite[Theorem 3.5]{primerledonne}), it is expected that they exhibit properties of spaces with nonnegative curvature. \cref{maintheoremA} seems to indicate that the $\mathsf{MCP}$ partially fails to capture that in the sub-Finsler setting. However it is worth noting when the $\mathsf{MCP}$ actually holds because several important properties are satisfied for an $\mathsf{MCP}$-space. For instance, sharp Poincaré and $p$-Poincaré inequalities hold under measure contraction property (see \cite{ppoincaréHan} and \cite{poincarréHanMilman}), and the $L^1$-localisation technique is also available (see \cite[Section 3.8]{QCDMilman}).

From \cref{MCPnegligible}, it can be seen that when the space has negligible cut locus, the $\mathsf{MCP}(0, N)$ holds if and only if $\mathfrak{m}(\Omega_t) \geq t^N \mathfrak{m}(\Omega)$ for all $t \in \interval{0}{1}$, for all measurable $\Omega \subseteq X$ with $0 < \mathfrak{m}(\Omega) < +\infty$,  and all $x \in X$. This equivalence does not hold in general, that is when the space does not have a negligible cut locus. Yet, one implication remains true.

\begin{proposition}[{Same proof as in \cite[Proposition 2.1]{sturm2006}}]
\label{MCPimplies1sideBM}
    If $(X, \mathrm{d}, \mathfrak{m})$ is a metric measure space (not necessarily with a negligible cut locus) satisfying the $\mathsf{MCP}(0, N)$, then $\mathfrak{m}(\Omega_t) \geq t^N \mathfrak{m}(\Omega)$ for all measurable set $\Omega \subseteq X$ with $0 < \mathfrak{m}(\Omega) < +\infty$ and all $x \in X$.
\end{proposition}

We can now introduce the curvature exponent, which was first coined in \cite{Rifford2013}.

\begin{definition}
The \textit{curvature exponent} $N_{\mathrm{curv}}$ of a metric measure space $(X, \diff, \mathfrak{m})$ is the number defined by
\[
N_{\mathrm{curv}}:=\inf\left\{N>1\mid \textsf{MCP}(0,N)~\text{is satisfied}\right\}.
\]
\end{definition}

Besides the measure contraction property and the curvature exponent, we will also investigate the so-called geodesic dimension. This quantity was introduced in \cite{curv2018} for sub-Riemannian manifolds, and extended to metric measure spaces in \cite{rizzi2016} (see also \cite[Section 4.5]{unified2022}).

\begin{definition}
\label{def:geod_dim}
For $x \in X$ and $s > 0$,
define the number
\begin{equation}
    \label{Cs(x)}
    C_s(x) := \sup\left\{\limsup_{t\to 0^+}\frac{1}{t^s}\frac{\mathfrak{m}(\Omega_t)}{\mathfrak{m}(\Omega)} \Big| \ \Omega \text{ Borel, bounded, } \mathfrak{m}(\Omega) \in \ointerval{0}{\infty}\right\},
\end{equation}
where $\Omega_t$ denotes the $t$-intermediate set of $\Omega$ from $x$.
The geodesic dimension at $x$ is the number
\[
N_{\mathrm{geo}}(x) := \inf\{s>0 \mid C_s(x)=+\infty\}=\sup\{s>0\mid C_s(x)=0\} \in \interval{0}{\infty}.
\]
The geodesic dimension of $(X, \diff, \mathfrak{m})$, denoted by $N_{\mathrm{geo}}$, is given by
\[
N_{\mathrm{geo}} := \sup\left\{ N_{\mathrm{geo}}(x) \mid x \in X \right\}.
\]
\end{definition}

The rationale behind the definition of the geodesic dimension is that, as explained in \cite{curv2018}, for structures regular enough, say for sub-Riemannian manifolds $(M, \diff_{\mathrm{CC}}, \mu)$ equipped with the Carnot-Carathéodory distance $\diff_{\mathrm{CC}}$ and a smooth measure $\mu$, the geodesic dimension is the number such that
\[
\mu(\Omega_t) \sim t^{N_{\mathrm{geo}}(x)}, \text{  
 as } t \to 0^+,
\]
for every measurable set $\Omega \subseteq M$ and all $x \in M$. Here, we write $f(t) \sim g(t)$ (as $t \to 0^+$) if there exists $C \neq 0$ such that $f(t) = g(t)(C+o(1))$ (as $t \to 0^+$). Roughly speaking, the geodesic dimension is thus more local in essence than the curvature exponent. 

The relationship between curvature exponent, geodesic dimension, and Hausdorff dimension is made clearer by the following statement.

\begin{theorem}[{\cite[Theorem 6]{rizzi2016} and \cite[Theorem 4.19]{unified2022}}]
\label{theo:NcurvNgeoNH}
For a measure metric space $(X, \diff, \mathfrak{m})$, it holds
\[
N_{\mathrm{curv}} \geq N_{\mathrm{geo}} \geq \dim_\mathcal{H}(X, \diff),
\]
where $\dim_{\mathcal{H}}(X, \diff)$ denotes the Hausdorff dimension of $(X, \diff)$.
\end{theorem}

The curvature exponent is often found to be equal to the geodesic dimension. This is the case for every $n$-dimensional Riemannian manifold with $\mathrm{Ric} \geq 0$, for which one finds $N_{\mathrm{curv}} = N_{\mathrm{geo}} = n$.
It is also known that this equality also holds for a large class of Carnot groups equipped with a left-invariant sub-Riemannian metric, see \cite{rizzi2016} and \cite{sharpcarnot2018}. To the best of our knowledge, the present work in fact provides the very first example of sub-Finsler Carnot groups which has a curvature exponent strictly greater than its geodesic dimension.

\subsection{Shelupsky's  \texorpdfstring{$p$}{p}-trigonometric functions}\label{sec2-3}

In this section, we introduce special functions, known as Shelupsky's $p$-trigonometric functions, that will be used to describe the geometry of the $\ell^p$-Heisenberg group. They were first studied in \cite{shelupsky1959} (see also \cite{lok2}) and we also gather here some important properties about them. We firstly follow the geometric definition from \cite{lok}, before describing them differentially as in \cite{shelupsky1959}.

For $p \in \interval{1}{\infty}$, the $\ell^p$-norm on the Euclidean plane is denoted by $\|\cdot\|_p$,
while $\mathds{B}_p$ $(\text{resp.}~\mathds{S}_p)$ is the unit ball (resp. unit sphere) of $(\R^2,\|\cdot\|_p)$, centered at the origin.
The area $\pi_p$ of $\mathds{B}_p$ is given by
\[
\pi_p = 4\frac{\Gamma(1+\frac{1}{p})^2}{\Gamma(1+\frac{2}{p})},
\]
where $\Gamma$ is the Gamma function.

\begin{definition}[{\cite[Definition 1]{lok}}]
\label{def:sinpcosp}
For $\theta\in[0,2\pi_p)$, a point $P_\theta$ on $\mathds{S}_p$ is chosen as the unique one such that the area of the sector of $\mathds{B}_p$ comprised between the $x$-axis and the straight line from the origin to $P_\theta$ is $\theta/2$. By definition, the $p$-trigonometric functions $\cos_p(\theta)$ and $\sin_p(\theta)$ are the coordinates of $P_\theta$, that is $(\cos_p(\theta), \sin_p(\theta)) := P_\theta $. The domain of the $p$-trigonometric functions is finally extended to the whole real line $\mathds{R}$ by $2 \pi_p$-periodicity. 
\end{definition}

\begin{figure}
    \centering
	\captionsetup[subfigure]{justification=centering}
	\subcaptionbox{$(\cos_p\theta,\sin_p\theta)$}{\scalebox{0.465}{\includegraphics[width=\linewidth]{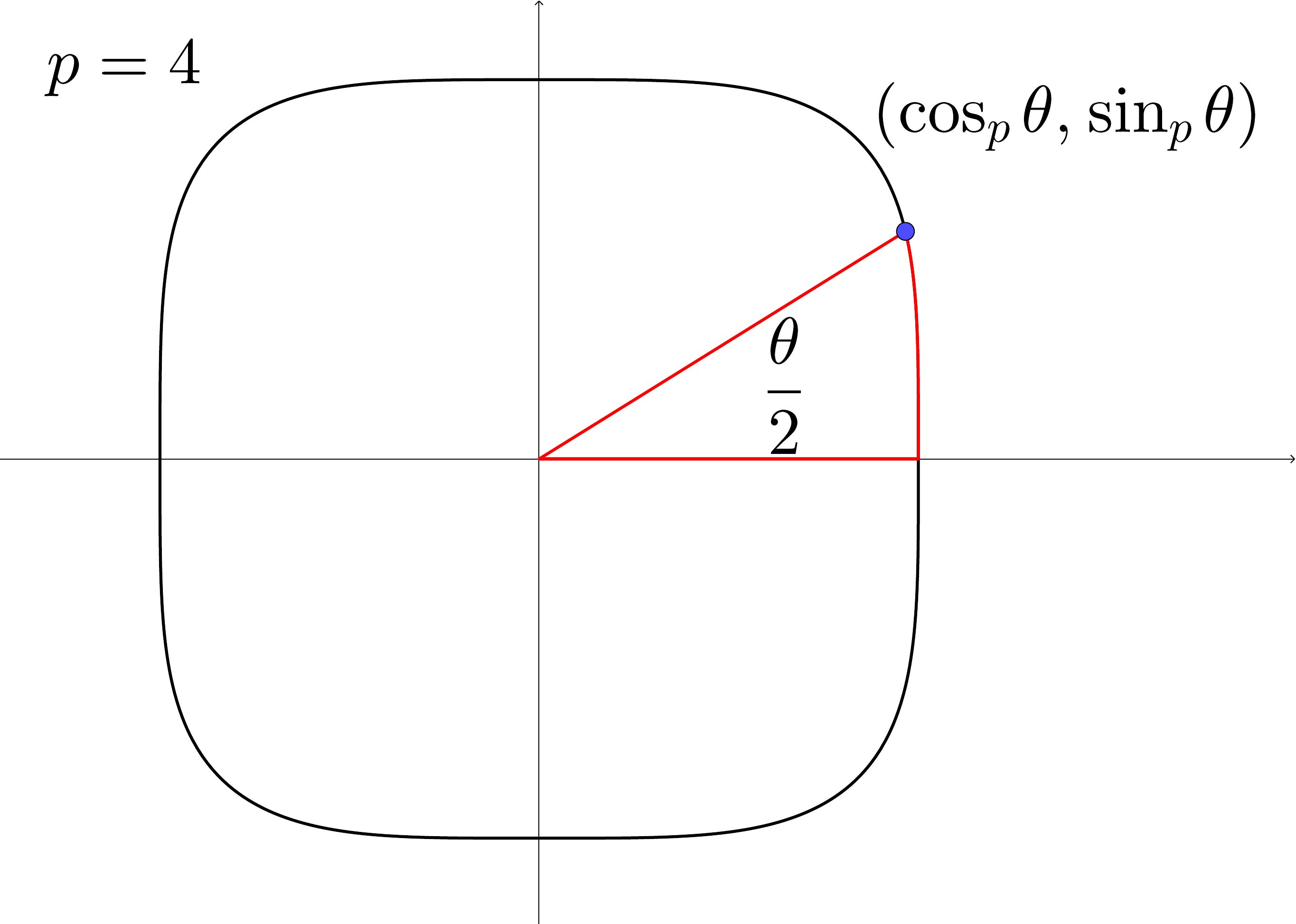}}}\hspace{1em}%
	\subcaptionbox{$(\cos_q\theta^\circ,\sin_q\theta^\circ)$}{\scalebox{0.465}{\includegraphics[width=\linewidth]{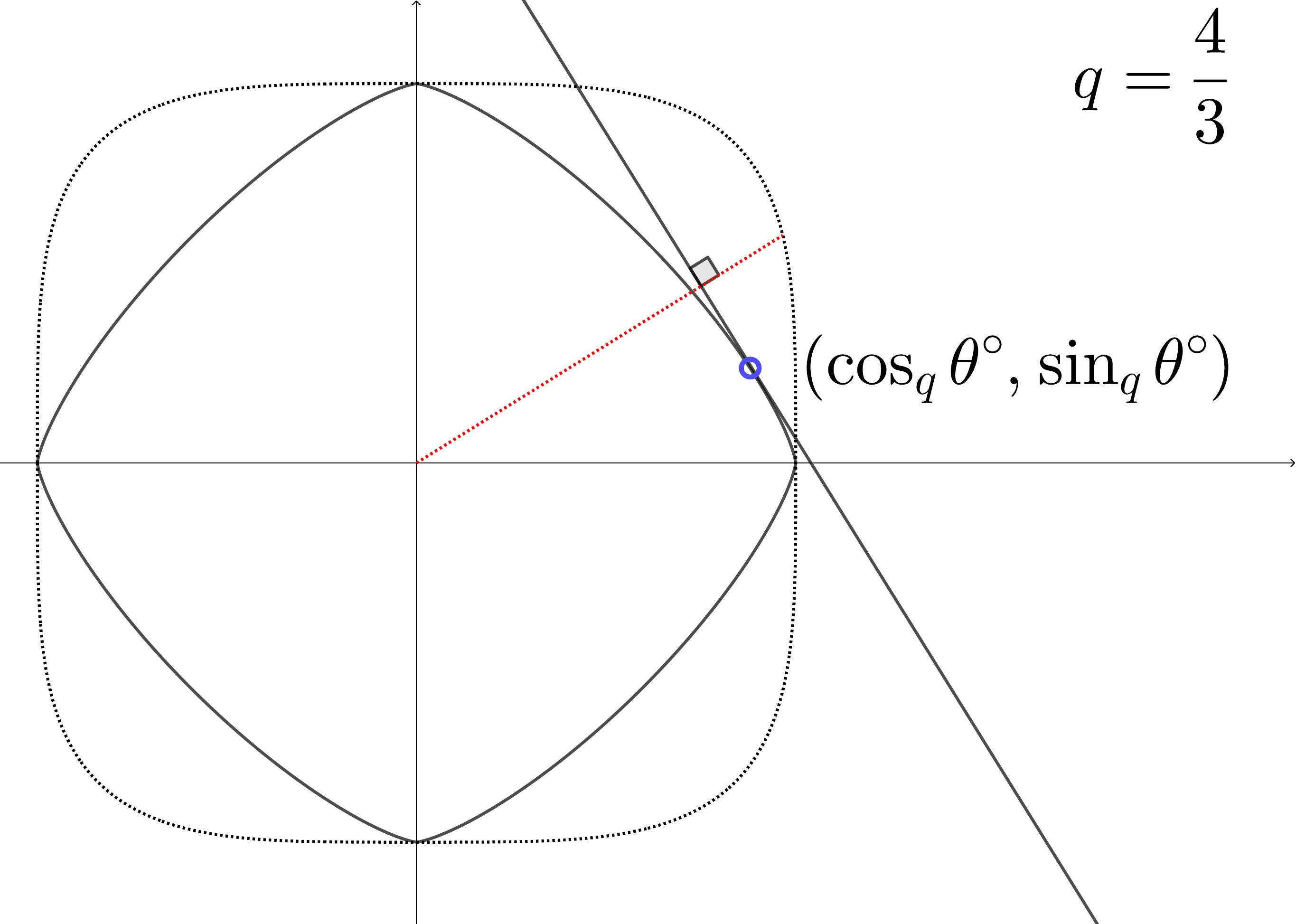}}}\hspace{1em}%
	\caption{Geometric definition of the $p$-trigonometric functions.}
    \label{geometricdeftrigp}
\end{figure}

From the definition above, one can easily see that the $2$-trigonometric functions coincide with the usual trigonometric functions. Furthermore, it holds $\sin_p(0) = 0, \cos_p(0) = 1$, and we clearly have the following $p$-trigonometric identity
\begin{equation}
\label{lemtri}
\abs{\cos_p(\theta)}^p + \abs{\sin_p(\theta)}^p = 1, \text{ for all }\theta \in \mathds{R}.
\end{equation}
By the symmetries of $\mathds{B}_p$, we also have
\[
\sin_p(\theta + \pi_p) = - \sin_p(\theta), \ \cos_p(\theta + \pi_p) = - \cos_p(\theta),
\]
as well as
\[
\sin_p(\theta + \pi_p/2) = \cos_p(\theta), \ \cos_p(\theta + \pi_p/2) = - \sin_p(\theta).
\]
The geometric definition of the $p$-trigonometric functions is illustrated in \cref{geometricdeftrigp} and their graphs are represented in \cref{fig:graphtrigp}.

\begin{figure}
    \centering
	\captionsetup[subfigure]{justification=centering}
	\subcaptionbox{$p<2$}{\scalebox{0.465}{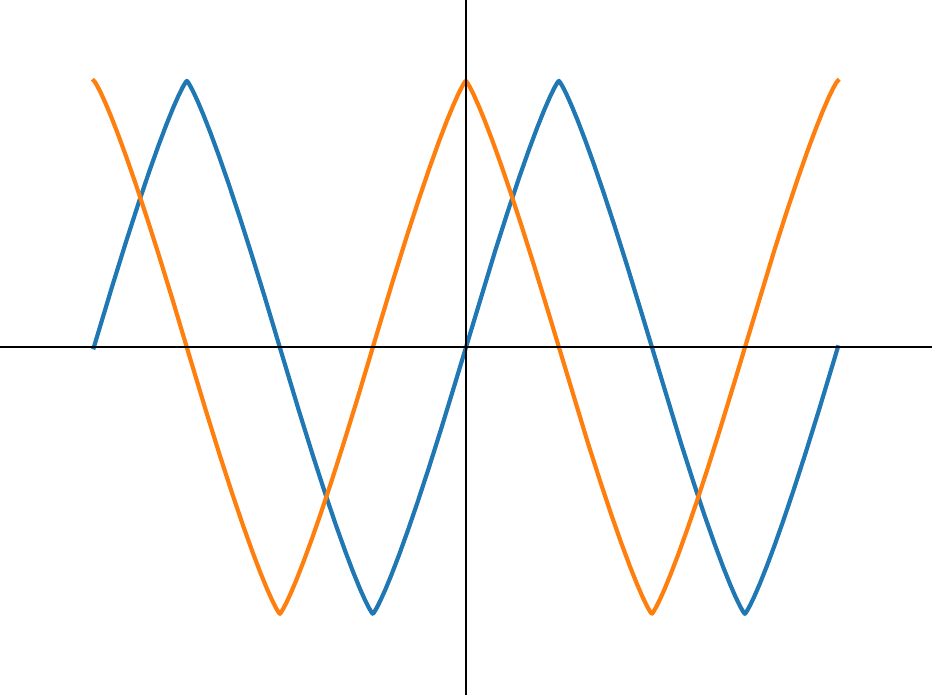}}\hspace{1em}%
	\subcaptionbox{$p>2$}{\scalebox{0.465}{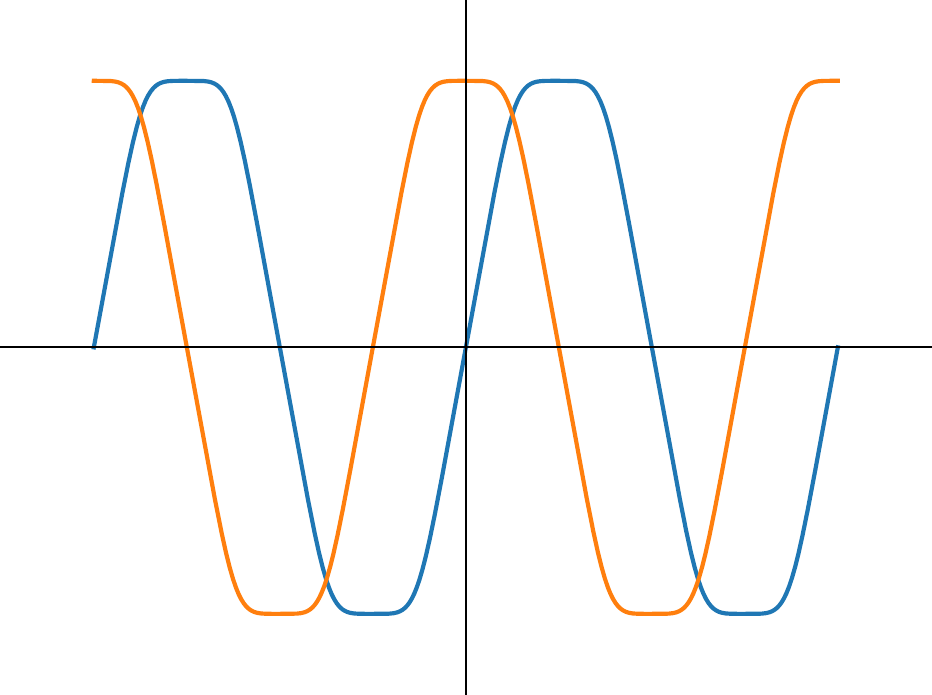}}\hspace{1em}%
	\caption{Graphs of the $p$-trigonometric functions $\sin_p(\theta)$ and $\cos_p(\theta)$}
    \label{fig:graphtrigp}
\end{figure}

\begin{remark}
    Instead of defining trigonometric functions with respect to the $p$-unit ball, one can replace $\mathds{B}_p$ in \cref{def:sinpcosp} by any compact convex set $\Omega$ such that $0 \in \mathrm{int}(\Omega)$, and obtain the corresponding sine and cosine functions, denoted by $\sin_\Omega$ and $\cos_\Omega$ in \cite{lok}. The polar set $\Omega^\circ$ of $\Omega$,
\[
\Omega^\circ := \left\{ (z, w) \mid xz+yw \leq 1, \text{ for all } (x, y) \in \Omega \right\},
\]
can be used to see that
\begin{equation}
    \label{eq:ineqtrigOmega}
    \cos_\Omega(\theta) \cos_{\Omega^\circ}(\phi) + \sin_\Omega(\theta) \sin_{\Omega^\circ}(\phi) \leq 1, \text{ for all } \theta, \phi \in \mathds{R}.
\end{equation}
\end{remark}
  
For the remainder of this work, unless stated otherwise, $q \in \interval{1}{\infty}$ will always be the Hölder conjugate of $p$, i.e., the number that satisfies the equation
\[
\frac{1}{p}+\frac{1}{q}=1.
\]
Observing that the polar set of $\mathds{B}_p$ is $\mathds{B}_q$, we can see that Shelupsky's trigonometric functions satisfy the following duality relation, similarly to \cref{eq:ineqtrigOmega}:
\begin{equation}
\label{eq:ineqtrigpq}
\cos_p(\theta) \cos_{q}(\phi) + \sin_p(\theta) \sin_{q}(\phi) \leq 1, \text{ for all } \theta, \phi \in \mathds{R}.
\end{equation}

It can be verified that for a given $\theta \in \rinterval{0}{2 \pi_p}$,
there is at least one $\theta^\circ \in \rinterval{0}{2 \pi_q}$ such that the equality holds in \cref{eq:ineqtrigpq}. This defines a monotonic multivalued map $\theta \mapsto \theta^\circ$ that we extend from $\rinterval{0}{2 \pi_p}$ to $\mathds{R}$ by periodicity. However, in general, it is difficult to compute $\theta^\circ$ explicitly. Nonetheless, it can also be verified that for $p \in \ointerval{1}{\infty}$, there is a unique $\theta^\circ$ for each given $\theta$, and that the maps $\theta \mapsto \theta^\circ$ is stricly monotonic and continuous, see \cite[Theorem 5, Figure 5]{lok2}. If $p = 1$, then we have $\theta^\circ = \frac{\pi_\infty}{4}$ when $\theta \in \ointerval{0}{\frac{\pi_1}{2}}$ and $\theta^\circ = \interval{\frac{\pi_\infty}{4}}{\frac{3\pi_\infty}{4}}$ when $\theta = \frac{\pi_1}{2}$. Similarly, if $p = \infty$, we have $\theta^\circ = \frac{\pi_1}{2}$ when $\theta \in \ointerval{\frac{\pi_\infty}{4}}{\frac{\pi_\infty}{4}}$ and $\theta^\circ = \interval{0}{\frac{\pi_2}{2}}$ when $\theta = \frac{\pi_\infty}{4}$.  

The theorem below provides a characterisation of the $p$-trigonometric functions, for $p \in \ointerval{1}{\infty}$, as the solution to a first-order system of differential equations.

\begin{theorem}[{\cite[Theorem 2]{lok} and \cite[Remark 2]{lok2}}]
\label{lemtridiff}
  For $p \in \ointerval{1}{\infty}$, the following system of differential equations holds:
  \begin{align*}
  \frac{\diff}{\diff\theta}\sin_p\theta &= \cos_q\theta^\circ = \abs{\cos_p \theta}^{p - 1} \sgn(\cos_p \theta),\\
  \frac{\diff}{\diff\theta}\cos_p\theta &= -\sin_q\theta^\circ = - \abs{\sin_p \theta}^{p - 1} \sgn(\sin_p \theta) .
  \end{align*}
\end{theorem}

\begin{remark}
\label{remark:regutrig}
    When $p \in \Z$, the $p$-trigonometric functions are analytic everywhere. When $p \notin \mathds{Z}$, the $p$-trigonometric functions are smooth everywhere except at points $\theta=0~\mathrm{mod}~\frac{\pi_p}{2}$, where they are only $C^{\lfloor p \rfloor}$. They are in fact analytic at $\theta\neq0~\mathrm{mod}~\frac{\pi_p}{2}$.
\end{remark}


\begin{lemma}
    \label{radiusconvsinp}
    Given $\theta_0 \in \ointerval{-\frac{\pi_p}{4}}{\frac{\pi_p}{4}} \setminus \{0\}$, the radius of convergence of $\sin_p$ at $\theta = \theta_0$ is no less than $\abs{ \theta_0}/(p+1)$.
\end{lemma}

\begin{proof}
From \cref{lemtridiff}, the $n$-th derivative ($n\geq 2$) of $\sin_p$ can be written as follows:
\begin{equation*}
\sin_p^{(n)}\theta = \sum_{k=1}^{n-1}a_{n,k}\abs{\cos_p\theta}^{(n-k)p-n}\abs{\sin_p\theta}^{kp-n+1}\sigma_n,
\end{equation*}
for some $a_{n, k} \in \R$, and where
\[
\sigma_n=
\begin{cases*}
    \sgn(\sin_p\theta) & if $n$ is even,\\
    \sgn(\cos_p\theta) & if $n$ is odd.
\end{cases*}
\]

Since $\abs{\sin_p\theta} < \abs{\cos_p\theta}$ by assumption, it is easy to check that

\begin{equation}
\label{eq:maxl}
\max_{k=1,\dots,n-1}\{\abs{\cos_p\theta}^{(n-k)p-n}\abs{\sin_p\theta}^{kp-n+1}|\}=\abs{\cos_p\theta}^{(n-1)p-n}\abs{\sin_p\theta}^{p-n+1}.
\end{equation}

Next, we compute an upper bound for the coefficients $\abs{a_{n, k}}$. By the Leibniz rule, the $n$-th derivative of $\sin_p\theta$ is a priori a sum of $2^{n-2}$ terms. Among those,
the number of the terms that have a common factor $\abs{\cos_p\theta}^{(n-k)p-n}\abs{\sin_p\theta}^{kp-n+1}$ is $\binom{n-2}{k-1}$. Moreover, the coefficients in front of all these factors take the form
\[
(p-1)\prod_{i=1}^{k-1}(c_ip-i)\prod_{j=1}^{n-k-1}(d_jp-j+1),
\]
where $c_i \in \{1,\dots,i\}$ and $d_j=\{1,\dots,j\}$. Clearly, we have
\begin{align*}
    \abs{(p-1)\prod_{i=1}^{k-1}(c_ip-i)\prod_{j=1}^{n-k-1}(d_jp-j+1)} \leq{}& (p-1)\prod_{i=1}^{k-1}(ip+i)\prod_{j=1}^{n-k-1}(jp+j) \\
    ={}&(p-1)(p+1)^{n-2}(k-1)!(n-k-1)!.
\end{align*}

Therefore, we obtain an upper bound of $|a_k^n|$ as
\[
|a_{n,k}|\leq \binom{n-2}{k-1}(k-1)!(n-k-1)!(p-1)(p+1)^{n-2}\leq (n-2)!(p+1)^n.
\]
Combined with \cref{eq:maxl}, we can estimate the radius of convergence using the root test
\begin{align*}
    \sqrt[n]{\abs{\frac{1}{n!}\sin_p^{(n)}\theta}} ={}& \sqrt[n]{\abs{\frac{1}{n!}\sum_{k=1}^{n-1}a_{n,k}\abs{\cos_p\theta}^{(n-k)p-n}\abs{\sin_p\theta}^{kp-n+1}\sigma_n}} \\
    \leq{}& \sqrt[n]{\frac{(p+1)^n}{n}\abs{\cos_p\theta}^{(n-1)p-n}\abs{\sin_p\theta}^{p-n}} \\
    &\longrightarrow (p+1)\frac{\abs{\cos_p\theta}^{p-1}}{\abs{\sin_p\theta}} \leq \frac{p+1}{\abs{\theta}}.
\end{align*}
The last inequality follows from analysing 
\[
\tan_p \theta := \frac{\sin_p \theta}{\sin'_p \theta} = \frac{\sin_p \theta}{\abs{\cos_p \theta}^{p - 1}\sgn(\cos_p \theta)}.
\]
If $\theta > 0$ (the other case is similar), then performing a second derivative on $\tan_p \theta$ reveals that it is convex for $0 \leq \theta \leq \frac{\pi_p}{4}$. The tangent to $\tan_p$ at $\theta = 0$ which must lie below the $p$-tangent by convexity, is observed to be $\theta$, which concludes the proof.
\end{proof}

Although the $p$-trigonometric functions are not analytic at $\theta = 0
$, it is possible to write a convergent expansion of these functions near $\theta = 0$. Indeed, both $\sin_p$ and $\cos_p$ are solutions to the $q$-Laplace equation
\[
(\abs{u'}^{q - 2} u')' + \abs{u}^{p - 2} u = 0,
\]
with their respective initial values. In \cite{paredes-uchiyama},
it is shown that such solutions of the $q$-Laplace equation have the following convergent expansion near $\theta = 0$.

\begin{theorem}[\cite{paredes-uchiyama}, Theorem 1 and 2]
\label{lemtriexpansion}
  For $p\in\ointerval{1}{\infty}$,
  the $p$-trigonometric functions have the following expansion, converging uniformly in a neighbourhood of $\theta = 0$:
  \begin{align*}
      \sin_p\theta{}& = \theta - \frac{p-1}{p(p+1)} \abs{\theta}^{p}\theta + \frac{(p-1)(2p^2-3p-1)}{2p^2(p+1)(2p+1)} \abs{\theta}^{2 p}\theta + \sum_{k = 3}^{\infty} a_k \abs{\theta}^{k p}\theta,\\
      \cos_p\theta {}& = 1-\frac{1}{p}|\theta|^p+\frac{(p-1)^2}{2p^2(p+1)}|\theta|^{2p} + \sum_{k = 3}^{\infty} b_k \abs{\theta}^{k p},\\
      \sin_q \theta^\circ {}& = \abs{\theta}^{p - 2} \theta -\frac{(p-1)^2}{p(p+1)} \abs{\theta}^{2 p - 2} \theta- \sum_{k = 3}^{\infty} k p b_k \abs{\theta}^{k p - 2} \theta,\\
      \cos_q \theta^\circ{}& = 1-\frac{p-1}{p}|\theta|^p+\frac{(p-1)(2p^2-3p-1)}{2p^2(p+1)}|\theta|^{2p}+ \sum_{k = 3}^{\infty} (k p + 1) a_k \abs{\theta}^{k p}. \\
  \end{align*}
\end{theorem}

The radius of convergence of $\theta \mapsto \sin_p \theta$, (resp. $\theta \mapsto \cos_p \theta$, $\theta \mapsto \sin_q \theta^\circ$ or $\theta \mapsto \cos_q \theta^\circ$) at $\theta = 0$ is defined as the supremum $|\theta|$ such that the corresponding series in \cref{lemtriexpansion}, which we call \textit{fractional power series}, converges. As for classical power series, a fractional power series converges uniformly within their radius of convergence. More generally, it will be enough for us to say that a function $f(x)$ has a \textit{fractional analytic expansion} at $x = 0$ if there is $s > 0$ and a function $F$, analytic at $x = 0$, such that $f(x) = F(|x|^s)$ or $f(x) = x F(|x|^s)$ for all $x$ in the radius of convergence of $F$. For additional information on fractional analytic functions and fractional derivatives, we refer the reader to \cite{JUMARIE}, for example, and the references therein.
\begin{remark}
    Shelupsky's $p$-trigonometric functions are related to the generalised $(p, q)$-trigonometric functions introduced in \cite{edmunds2012}. For arbitrary $p, q \in \ointerval{1}{\infty}$, i.e., not necessarily Hölder conjugates of each other, consider the following eigenvalue problem for the $(p, q)$-Laplacian:
    \begin{equation}
        \label{pqLaplace}
        (\abs{u'}^{p - 2} u')' + \frac{q}{p^*} \abs{u}^{q - 2} u = 0.
    \end{equation}
    The solution of \cref{pqLaplace} with initial value $u(0) = 0$ and $u'(0) = 1$ is the so-called $(p, q)$-sine function $\sin_{p, q}$, and the $(p, q)$-cosine is defined as $\cos_{p, q} = \sin'_{p, q}$. Therefore, we have
    \[
    \sin_{p}(\theta) = \sin_{p^*, p}(\theta), \text{ and } \cos_{p}(\theta) = \cos_{p, p^*}(\theta^{\circ}),
    \]
    where $\theta^\circ$ is the dual angle of $\theta$, going from the $p$-geometry to the $p^*$-geometry.
\end{remark}

\section{Geometry of the \texorpdfstring{$\ell^p$}{ellp}-Heisenberg group}
\label{sec:geometry}

\subsection{Sub-Finsler structures on the Heisenberg group}

The Heisenberg group $\mathds{H}$ is the connected and simply connected Lie group whose Lie algebra is graded, nilpotent, satisfying $\mathfrak{h} = \mathfrak{h}_1 \oplus \mathfrak{h}_2$ and such that $\mathfrak{h}_1$ is two-dimensional and generates $\mathfrak{h}$. In particular, $\mathds{H}$ is three-dimensional and fixing a basis $X, Y,$ and $Z = [X, Y]$ of $\mathfrak{h}$ induces global coordinates on $\mathds{H}$, called \emph{exponential coordinates}, through the map $(x, y, z) \mapsto \exp(x X + y Y + z Z)$. By the Campbell-Hausdorff formula, the law group in these coordinates writes as
\[
(x, y, z) \cdot (x', y', z') = (x + x', y + y', z + z' + \frac{1}{2}(x y' - x' y)).
\]
With a slight abuse of notation, the left-invariant vector fields on $\mathds{H}$ whose values at the identity are $X, Y$ and $Z$ will also be denoted by $X, Y$ and $Z$ respectively. Their expression in exponential coordinates is
\[
X = \partial_x - \frac{y}{2} \partial_z, \ Y = \partial_y + \frac{x}{2} \partial_z, \ Z = \partial_z.
\]
The distribution $\mathcal{D} := \left\{X, Y\right\}$ is left-invariant and bracket generating, that is to say $\mathrm{Lie}(\mathcal{D}_g) = \T_g(\mathds{H})$ for all $g \in \mathds{H}$.

The Heisenberg group $\mathds{H}$ has been studied extensively when equipped with its natural sub-Riemannian structure, i.e. with a scalar product on $\mathcal{D}$ that turns $X$ and $Y$ into an orthonormal basis. In this work, we are going to study the Heisenberg endowed with a \emph{sub-Finsler} structure.
Carnot groups endowed with sub-Finsler metrics have been extensively studied in connection with geometric group theory (see \cite{stoll1998,breuillard--ledonne2013,duc, tashiro2022}), with optimal control theory (see \cite{barilari2017,lok,lok2}), and more recently with isoperimetric problems (see \cite{Monti2023}).

In the case of the Heisenberg group, fixing a sub-Finsler structure generally means that we associate a norm $\|\cdot\|$ to the distribution $\mathcal{D}$. Central to the present work is the standard $p$-norm: for $v = u_1 X + u_2 Y \in \mathcal{D}$ and $p \in \interval{1}{\infty}$, the $\ell^p$-sub-Finsler metric is defined by
\begin{equation}
\label{pnorm}
\|v\|_p =
\begin{cases}
|u_1| + |u_2| & \text{ if } p = 1;\\
\left(|u_1|^p+|u_2|^p\right)^\frac{1}{p} & \text{ if } p \in \ointerval{1}{\infty};\\
\max(u_1, u_2) & \text{ if } p = \infty.
\end{cases}
\end{equation}

We will refer to the sub-Finsler manifold formed by the Heisenberg group and its $\ell^p$-sub-Finsler metric as the $\ell^p$-Heisenberg group.

Once a sub-Finsler structure with distribution $\mathcal{D}$ and norm $\|\cdot\|$ is fixed, an absolutely continuous curve $\gamma : \interval{0}{1} \to \mathds{H}$ is said to be \emph{horizontal} is $\dot{\gamma}(t) \in \mathcal{D}_{\gamma(t)}$ for almost every $t \in \interval{0}{1}$, and its length is given by
\[
\mathrm{L}(\gamma) := \int_0^1 \|\dot{\gamma}(t)\| \diff t,
\]
while the induced distance of $\mathds{H}$ is
\[
\diff(g_0,g_1) := \inf \left\{ \mathrm{L}(\gamma) \mid \gamma : \interval{0}{1} \to \mathds{H} \text{ horizontal}, \gamma(0) = g_0, \gamma(1) = g_1 \right\}.
\]
By the Chow–Rashevskii theorem, this distance is well-defined and finite since $\mathcal{D}$ is bracket-generating.
Notice that the length of a horizontal curve is invariant up to reparametrization.
Therefore, we sometimes require a horizontal curve to have a constant speed i.e. $\|\dot{\gamma}\|\equiv \mathrm{L}(\gamma)$.

\begin{remark}
Note that the $\ell^p$-Heisenberg groups are bi-Lipschitz equivalent to each other, so they all have the same Hausdorff dimension of 4.
\end{remark}

\subsection{Pontryagin's Maximum Principle and Hamilton's Equations}
\label{subsec:PMP}
The metric geometry of a sub-Finsler Heisenberg group is governed by the study of sub-Finsler minimising geodesics, i.e., curves that minimise the length
between two points $g_0$ and $g_1$, which are found by solving the following minimisation problem:
\begin{equation}
    \label{lengthminproblem}
    \left\{
    \begin{aligned}
    &\dot{\gamma}(t) = u_1(t) X(\gamma(t)) + u_2(t) Y(\gamma(t)) \\
    &\mathrm{L}(\gamma) \to \min \\
    &\gamma(0) = g_0, \ \gamma(1) = g_1
    \end{aligned}.
    \right.
\end{equation}

Instead of minimising the length between two points as in Problem \cref{lengthminproblem}, 
it is possible to alternatively look to minimise the energy.
The energy of a horizontal curve $\gamma : \interval{0}{1} \to \mathds{H}$ is defined by
\[
\mathrm{E}(\gamma) := \frac{1}{2} \int_0^1 \|\dot{\gamma}(t)\|^2 \diff t.
\]
The corresponding minimising problem is then the following:
\begin{equation}
    \label{energyminproblem}
    \left\{
    \begin{aligned}
    &\dot{\gamma}(t) = u_1(t) X(\gamma(t)) + u_2(t) Y(\gamma(t)) \\
    &\mathrm{E}(\gamma) \to \min \\
    &\gamma(0) = g_0, \ \gamma(1) = g_1
    \end{aligned}.
    \right.
\end{equation}

The equivalence between the minimisation problems \cref{lengthminproblem} and \cref{energyminproblem} is proven as in the sub-Riemannian case, see \cite[Lemma 3.64.]{agrachevbook2020}. By using the Cauchy--Schwarz inequality, it can be shown that a horizontal curve $\gamma : \interval{0}{1} \to \mathds{H}$ is a minimiser of the energy functional $\mathrm{E}$ if and only if it is a constant speed minimiser of the length functional $\mathrm{L}$.

\begin{remark}
Equivalently, the minimisation problem \cref{lengthminproblem} can also be rewritten as the following time-optimal control problem for a horizontal curve $\gamma:[0,T]\to \mathds{H}$:

\begin{equation}
    \label{convexminproblem}
    \left\{
    \begin{aligned}
    &\dot{\gamma}(t) = u_1(t) X(\gamma(t)) + u_2(t) Y(\gamma(t)) \in  (L_{\gamma(t)})_\ast  B \\
    &T \to \min \\
    &\gamma(0) = g_0, \ \gamma(T) = g_1
    \end{aligned}.
    \right.
\end{equation}
where $L_{P}$ is the left-translation, and $B\subset \mathcal{D}_e$ is a compact convex neighbourhood of $0$ (with other technical assumptions, see \cite[Main Assumption]{lok2}).
This minimising problem gives length minimisers for the non-reversible sub-Finsler Heisenberg group whose unit ball is $B$.
In this setting, the exponential map cannot be defined in general. However, length minimisers on non-strictly convex sub-Finsler structure, such as $\ell^1$- or $\ell^\infty$ structures, can still be computed. 

\end{remark}

We apply Pontryagin's maximum principle to obtain necessary conditions that a minimiser of \cref{energyminproblem} must satisfy. We refer the reader to \cite{agrachevsachkovbook2004}, for example, for a reference on Pontryagin's maximum principle.


For $\nu \in \mathds{R}$, the Hamitonian $\mathcal{H}^\nu : \T^*(\mathds{H}) \times \mathds{R}^2 \to \mathds{R}$ is defined as
\[
\mathcal{H}^\nu(\lambda, u) = u_1 h_X(\lambda) + u_2 h_Y(\lambda) + \frac{\nu}{2} \|u_1 X(\pi(\lambda)) + u_2 Y(\pi(\lambda))\|^2,
\]
where $\pi: \T^\ast(\mathds{H}) \to \mathds{H}$ is the canonical projection and  for a vector field $V$ on $\mathds{H}$, we have denoted by $h_V$ the function on $\T^*(\mathds{H})$ defined by $h_V(\lambda) := \langle \lambda, V(\pi(\lambda)) \rangle$. Note that $h_X$, $h_Y$, and $h_Z$ form a system of coordinates on each fibers of $\T^*(\mathds{H})$.

For a fixed control $u \in \mathds{R}^2$, the map $\lambda \mapsto \mathcal{H}^\nu(\lambda, u)$ is smooth and the Hamiltonian vector field $\overrightarrow{\mathcal{H}}^\nu(\cdot, u) : \T^*(\mathds{H}) \to \T^*(\mathds{H})$ is then the unique vector field on $\T^*(\mathds{H})$ satisfying 
\[
\sigma(\cdot, \overrightarrow{\mathcal{H}}^\nu(\lambda, u)) = \diff_{\lambda} \mathcal{H}^\nu(\cdot, u),
\]
where $\sigma$ denotes the canonical symplectic form of $\T^*(\mathds{H})$.

\begin{theorem}[{Pontryagin's Maximum Principle}]
    \label{PMP}
    Let $(u(t), \gamma(t))$ be a solution to the minimisation problem \cref{energyminproblem}. Then there exists an absolutely continuous curve $\lambda : \interval{0}{1} \to \T^*(\mathds{H})$ satisfying $\pi(\lambda(t)) = \gamma(t)$, and a number $\nu \in \{0, -1\}$, such that for almost all $t \in \interval{0}{1}$
    \begin{enumerate}[label=\normalfont(\roman*)]
    \item $\lambda(t) \neq 0$ if $\nu = 0$;
    \item $\dot{\lambda}(t) = \overrightarrow{\mathcal{H}}^\nu(\lambda(t), u(t))$;
    \item $\mathcal{H}^\nu(\lambda(t), u(t)) = \max_{u \in \mathds{R}^2} \mathcal{H}^\nu(\lambda(t), u)$.
    \end{enumerate}
\end{theorem}

We will say that $\lambda : \interval{0}{1} \to \T^*(\mathds{H})$ is a normal (resp. abnormal) extremal trajectory if it satisfies the conclusions of \cref{PMP} with $\nu = -1$ (resp. $\nu = 0$). The sub-Finsler Heisenberg group does not admit any non-trivial abnormal extremal, exactly as in the sub-Riemannian case. This fact does not depend on the cost functional, but only on the distribution $\mathcal{D}$.




We now turn to the study of normal extremals. From now on, we will focus on the $\ell^p$-Heisenberg group, that is $\mathds{H}$ equipped with the sub-Finsler norm $\|\cdot\|_p$ with $p \in \interval{1}{\infty}$. In fact, we will first restrict ourselves to the case $p \in \ointerval{1}{\infty}$. Here, Shelupsky's $p$-trigonometric described in \cref{sec2-3} functions will be prominent.

We can obtain an explicit expression for the maximised Hamiltonian.

\begin{proposition}
    \label{prop:maxHp}
    If we denote by $H : \T^*(\mathds{H}) \to \mathds{R}$ the maximised Hamiltonian from part (iii) of \cref{PMP} with $\nu = -1$, we have that
    \[
    H(\lambda) := \max_{u \in \mathds{R}^2}\mathcal{H}^{\nu}(\lambda, u) = \frac{1}{2} (|h_X(\lambda)|^q+ |h_Y(\lambda)|^q)^{2/q}.
    \]
    Furthermore, for a given $\lambda \in \T^*(\mathds{H})$, the maximum above is attained by taking
    \[
    u_1 = r\cos_p(\theta^\circ), \text{ and } u_2 = r  \sin_p(\theta^\circ),
    \]
    where $(r, \theta)$ are the $\ell^q$-polar coordinates of $(h_X(\lambda), h_Y(\lambda))$.
\end{proposition}

\begin{proof}
    Let $\lambda \in \T^*(\mathds{H})$, $u \in \mathds{R}^2$, and introduce $A, B > 0$, $ \phi \in \rinterval{0}{2 \pi_p}$, and $\theta \in \rinterval{0}{2 \pi_q}$ by setting
    \[
    u_1 = A \cos_p(\phi), \ u_2 = A \sin_p(\phi), \ h_X(\lambda) = B \cos_q(\theta), \ h_Y(\lambda) = B \sin_q(\theta).
    \]
    Then, by \cref{lemtri} we find that
    \begin{align*}
        \mathcal{H}^\nu(\lambda, u) ={}& u_1 h_X(\lambda) + u_2 h_Y(\lambda) - \frac{1}{2} (|u_1|^p + |u_2|^p)^{2/p} \\
        ={}& A B (\cos_p(\phi) \cos_q(\theta) + \sin_p(\phi) \sin_q(\theta)) - \frac{1}{2} A^2 \leq A B - \frac{1}{2} A^2,
    \end{align*}
    with equality if and only if $\phi = \theta^\circ$. The function $A \mapsto A B - \frac{1}{2} A^2$ attains its unique maximum at $A = B$, which means that
    \begin{align*}
        H(\lambda) :={}& \max_{u \in \mathds{R}^2} \mathcal{H}^\nu(\lambda, u) = \frac{1}{2} B^{2} = \frac{1}{2} (|B \cos_q(\theta^\circ)|^q+ |B \sin_q(\theta^\circ)|^q)^{2/q} \\
        ={}& \frac{1}{2} (|\langle \lambda, X(\pi(\lambda)) \rangle|^q+ |\langle \lambda, Y(\pi(\lambda)) \rangle|^q)^{2/q}.
    \end{align*}
\end{proof}
\subsection{Exponential map of the \texorpdfstring{$\ell^p$}{ellp}-Heisenberg group}

We are going to write the expression for geodesics of the $\ell^p$-Heisenberg group, when $p \in \ointerval{1}{\infty}$, with the help of Shelupsky's $p$-trigonometric functions introduced in \cref{def:sinpcosp}. 
When $\lambda(t)$ is a normal extremal associated with a solution $(u(t), \gamma(t))$ of the minimisation problem, we will denote $h_X(t) := h_X(\lambda(t))$ (similarly for $h_Y(t)$, $h_Z(t)$, etc.).
In view of part (ii) of \cref{PMP}, the following identities must hold:
\begin{equation}
\label{hameqh}
\begin{aligned}
    \dot{h}_X ={}& \{\mathcal{H}^\nu(\cdot, u), h_X\} = u_2 \{h_Y, h_X\} = -u_2 h_Z,\\
    \dot{h}_Y ={}& \{\mathcal{H}^\nu(\cdot, u), h_Y\} = u_1 \{h_X, h_Y\} = u_1 h_Z,\\
    \dot{h}_Z ={}& \{\mathcal{H}^\nu(\cdot, u), h_Z\} = 0,
\end{aligned}
\end{equation}
where $\{\cdot, \cdot\}$ is the canonical Poisson bracket of $\T^\star(\mathds{H})$.
The last equation implies that $h_Z(t) = w$ is constant, while arguing as in the proof of \cref{prop:maxHp} shows that the maximum principle, i.e., part (iii) of \cref{PMP}, forces the control $u(t)$ to satisfy
\[
u_1 = r \cos_p(\theta^\circ), \ u_2 =  r \sin_p(\theta^\circ),
\]
where $(r, \theta)$ are the $\ell^q$-polar coordinates of $h_X$ and $h_Y$:
\[
h_X = r \cos_q(\theta), \ h_Y = r \sin_q(\theta).
\]
In particular, we have
\[
\dot{h}_X = \dot{r} \cos_q(\theta) - r \sin_p(\theta^\circ) \dot{\theta}, \ \dot{h}_Y = \dot{r} \sin_q(\theta) + r \cos_p(\theta^\circ) \dot{\theta}.
\]
Using these identities, we can write \cref{hameqh} as
\begin{align*}
    \dot{r} ={}& \dot{h}_X \cos_p(\theta^\circ) + \dot{h}_Y \sin_p(\theta^\circ) \\
    ={}& \dot{h}_X \abs{\cos_q(\theta)}^{q - 1} \sgn \cos_q(\theta) + \dot{h}_Y \abs{\sin_q(\theta)}^{q - 1} \sgn \sin_q(\theta) \\
    ={}& - u_2 w \cos_p(\theta^\circ) + u_1 w \sin_p(\theta^\circ) \\
    ={}& r w (\sin_p(\theta^\circ) \cos_p(\theta^\circ) - \cos_p(\theta^\circ) \sin_p(\theta^\circ)) = 0,
\end{align*}
and
\begin{align*}
    \dot{\theta} ={}& \frac{1}{r}(\dot{h}_Y \cos_q(\theta) - \dot{h}_X \sin_q(\theta)) = \frac{w}{r}(u_1 \cos_q(\theta) + u_2 \sin_q(\theta)) \\
    ={}& \frac{w}{r}(u_1 \cos_q(\theta) + u_2 \sin_q(\theta)) = w (\cos_p(\theta^\circ) \cos_q(\theta) + \sin_p(\theta^\circ) \sin_q(\theta)) = w .
\end{align*}

That yields
\begin{equation}
    \label{hXhY}
    h_X(t) = r \cos_q(w t + \theta), \  h_{Y}(t) = r \sin_q(w t + \theta).
\end{equation}

From there, we can find the projection of the normal extremal in the exponential coordinates of $\mathds{H}$. By left invariance, it is enough to consider those that start from the identity of $\mathds{H}$. 

\begin{proposition}
    \label{geodH}
    Let $\gamma : \interval{0}{1} \to \mathds{H}$ the projection of a normal extremal $\lambda(t)$ starting from the identity, then in exponential coordinates $\gamma(t) = (x(t), y(t), z(t))$ where
\begin{equation}
    \label{eq:x(t)y(t)z(t)}
    \left\{
    \begin{aligned}
    x(t) ={}& \frac{r}{w}\Big(\sin_q(w t + \theta)-\sin_q (\theta) \Big), \\
    y(t) ={}& -\frac{r}{w}\Big(\cos_q(w t + \theta) - \cos_q (\theta) \Big),\\
    z(t) ={}& \frac{r^2}{2 w^2}\Big(wt+\cos_q(w t + \theta) \sin_q (\theta) -\sin_q(w t + \theta) \cos_q (\theta) \Big)
    \end{aligned},
    \right.
\end{equation}
for some $r \geq 0$, $w \in \mathds{R}$ and $\theta \in \rinterval{0}{2 \pi_q}$,
where at $w = 0$ the equations above are understood to be
    \begin{equation}
    \label{eq:x(t)y(t)z(t)atw=0}
    \left\{
    \begin{aligned}
    x(t) ={}& (r\cos_p\theta^\circ)t, \\
    y(t) ={}& (r\sin_p\theta^\circ)t,\\
    z(t) ={}& 0,    \end{aligned}
    \right.
\end{equation}
which is the continuous extension of \cref{eq:x(t)y(t)z(t)} as $w \to 0$.
\end{proposition}

\begin{remark}
    The parameter $r$ vanishes if and only if $\gamma$ describes a trivial curve.
\end{remark}

\begin{proof}
    From \cref{hameqh}, we know that
    \begin{align*}
        \dot{\gamma}(t) ={}& u_1(t) X(\gamma(t)) + u_2(t) Y(\gamma(t)) = \frac{1}{w} \Big(\dot{h}_Y(t) X(\gamma(t)) - \dot{h}_X(t) Y(\gamma(t)) \Big) \\
        ={}& \frac{1}{w} \Big(\dot{h}_Y(t) \partial_x - \dot{h}_X(t) \partial_y - \frac{1}{2} \left( y(t) \dot{h}_Y(t) + x(t) \dot{h}_X(t) \right) \partial_z \Big).
    \end{align*}
    Considering \cref{hXhY} and integrating this differential equation with the initial condition $\gamma(0) = (x(0), y(0), z(0)) = 0$ gives the expression \cref{eq:x(t)y(t)z(t)}.
    
   Next we shall see the continuity at $(r,\theta,0)$ i.e. the limit $w\to 0$ with fixed $t\in(0,1]$.
   Since $q$-trigonometric function is $C^1$ (and smooth if $\theta\notin\frac{\pi_q}{2}\Z$),
    we can compute the limit of $x(t)$ and $y(t)$ by considering the derivative with respect to $w$ at $(r,\theta,0)$.
    To prove $z(t)\to 0$ we need to consider the second derivative.
    A priori the $q$-trigonometric function may not be $C^2$,
    however the function
    \[wt+\cos_q(w t + \theta) \sin_q (\theta) -\sin_q(w t + \theta) \cos_q (\theta)\]
    is twice differentiable at $w=0$ for all $\theta\in[0,2\pi_q]$,
    and its first and second derivative are $0$.
    Therefore $z(t)\to 0$ as $w\to 0$.
\end{proof}

\begin{figure}
    \centering
	\captionsetup[subfigure]{justification=centering}
	\subcaptionbox{$p<2$}{\scalebox{0.65}{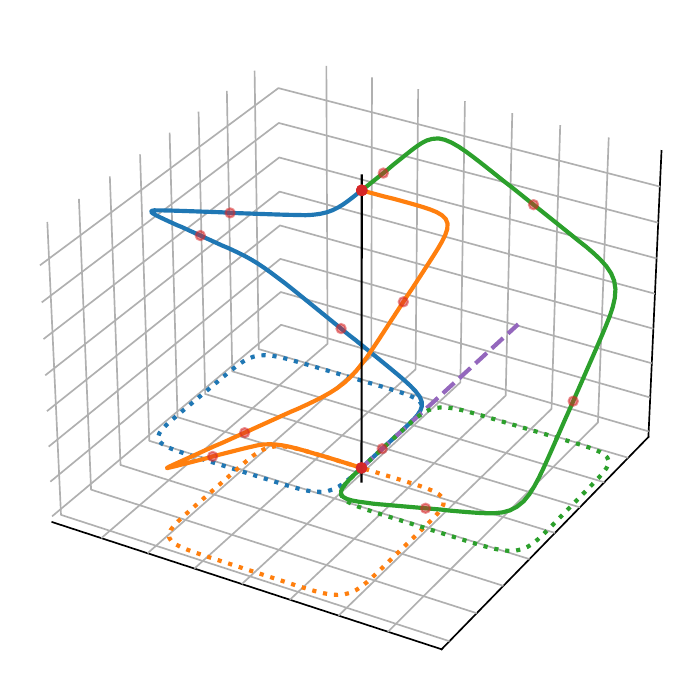}}\hspace{1em}%
	\subcaptionbox{$p>2$}{\scalebox{0.65}{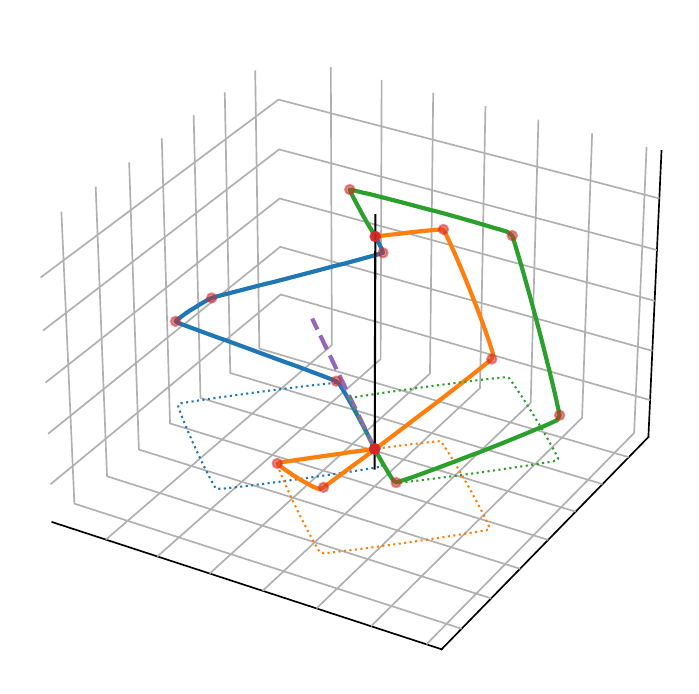}}\hspace{1em}%
	\caption{Some geodesics of the $\ell^p$-Heisenberg group from the origin. The dotted curves represents the projections of the geodesics onto the plane, while the red points are where the geodesics lose smoothness.}
    \label{fig:geodesicslp}
\end{figure}

\begin{figure}
    \centering
	\includegraphics[width=0.8\textwidth]{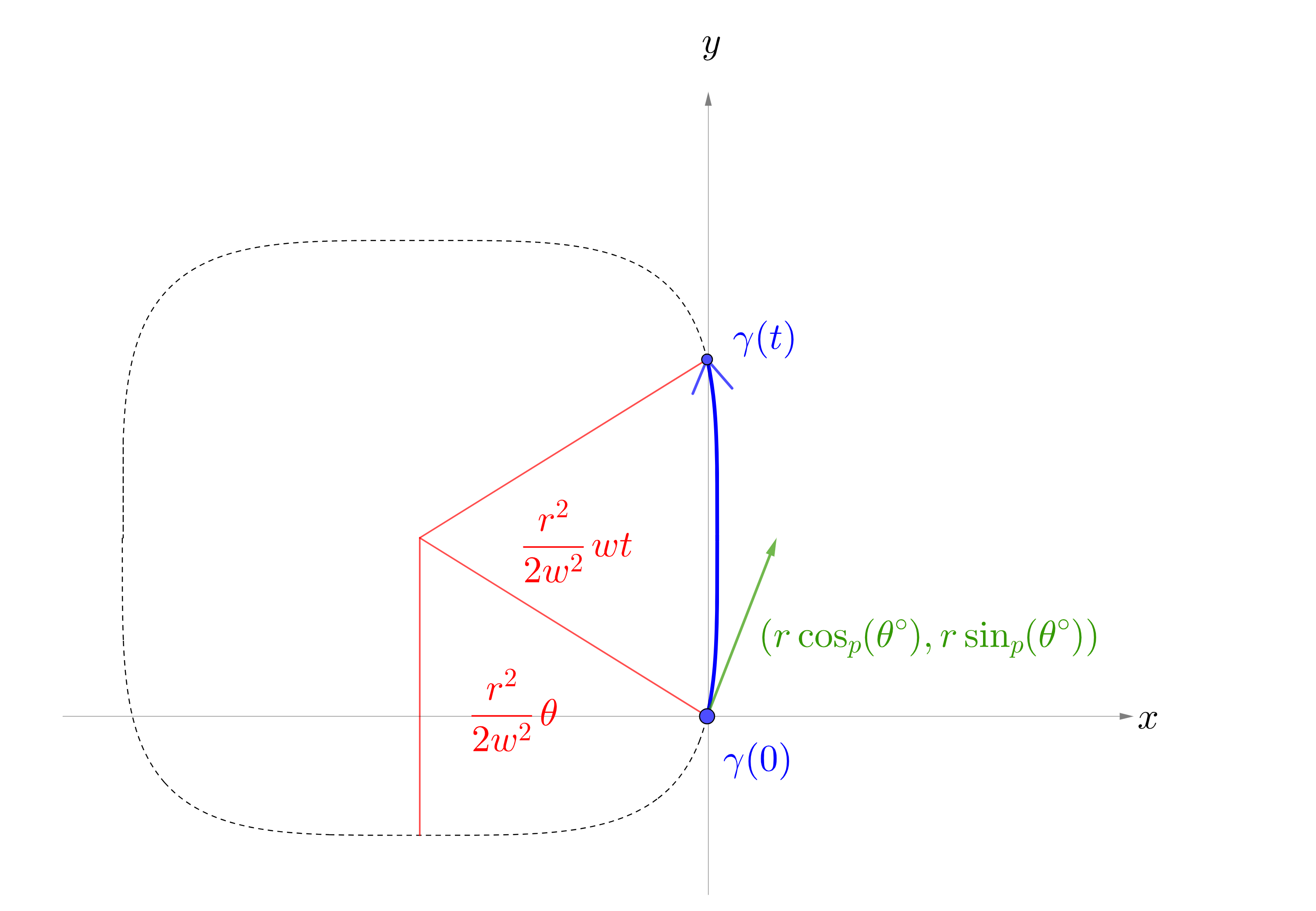}
	\caption{The projection of a geodesic $t \mapsto \gamma(t) := \exp^t_e(r,\theta,w)$ of the $\ell^p$-Heisenberg group is an $\ell^q$-arc. Here the red denotes the area, the blue does the geodesic, and the green does the initial vector.}
    \label{fig:geodesicsproj}
\end{figure}

The second derivative of a geodesic with respect to the time $t$ is explicitly written by
\[\begin{cases}
    \ddot{x}(t)=-rw(q-1)|\cos_q(wt+\theta)\sin_q(wt+\theta)|^{q-2}\sin_q(wt+\theta),\\
    \ddot{y}(t)=rw(q-1)|\cos_q(wt+\theta)\sin_q(wt+\theta)|^{q-2}\cos_q(wt+\theta),\\
    \ddot{z}(t)=\frac{-r^2(q-1)}{2}|\cos_q(wt+\theta)\sin_q(wt+\theta)|^{q-2} \\
    \qquad \qquad \qquad \qquad \times \left(\cos_q(wt+\theta)\sin_q\theta-\sin_q(wt+\theta)\cos_q\theta\right)
\end{cases}\]
The singularity appears when $wt+\theta\in\frac{\pi_q}{2} \mathds{Z}$.
If $p>2$ i.e. $q<2$,
then the second derivative $\ddot{x}(t)$ (or $\ddot{y}(t)$) and $\ddot{z}(t)$ diverges around $wt+\theta\in\frac{\pi_q}
{2} \mathds{Z}$ with $w\neq 0$.
On the other hand,
if $p<2$ i.e. $q>2$,
then all the second derivative converges to $0$ as $wt+\theta$ goes to $\frac{\pi_q}{2}\Z$.
In other words,
the curvature of a geodesic is $0$ at this point.
This type of singularity also appears when we consider the derivative of the Jacobian of the exponential map in \cref{sec:derivativejac}.
We illustrate a few geodesics of the $\ell^p$-Heisenberg group and its singular points in \cref{fig:geodesicslp}.

With \cref{geodH}, a notion of exponential map can be defined similarly to what is done in (sub-)Riemannian geometry.

\begin{definition}
    For $t \in \mathds{R}$, the exponential map (at time $t$) from the identity element of the $\ell^p$-Heisenberg group is the map
    \[
    \exp_e^t : \T_e^*(\mathds{H}) \cong \rinterval{0}{\infty} \times \rinterval{0}{2 \pi_q} \times \mathds{R} \to \mathds{H} : (r,\theta, w) \mapsto (x(t), y(t), z(t)),
    \]
    where $x(t)$, $y(t)$, and $z(t)$ are given by \cref{eq:x(t)y(t)z(t)}. The identification between $\T_e^*(\mathds{H})$ and $\rinterval{0}{\infty} \times \rinterval{0}{2 \pi_q} \times \mathds{R}$ is made via the $\ell^q$-cylindrical coordinates.
\end{definition}

\begin{remark}
    We denote by $\exp_e$ the map $\exp_e^1$. Similarly to the construction of this section, we can define the exponential $\exp_g^t$ from any point $g \in \mathds{H}$ from the same minimisation problem \cref{energyminproblem}. However, the exponential map $\exp_g^t$ can be obtained from $\exp_e^t$ by left-translation so that we can safely work with $\exp_e^t$ only, without loss of generality.
\end{remark}

It can be seen from \cref{eq:x(t)y(t)z(t)} that the exponential map enjoys the following homogeneity property:
\begin{equation}
    \label{exphomog}
    \exp_e^t(r, w, \theta) = \exp_e^{c t}(r/c, \theta, w/c), \ \ \text{ for all } c \neq 0.
\end{equation}
In particular, 
$\exp_e(t \lambda_0) = \exp_e(t r, \theta, t w) = \exp_e^t(r, \theta, w) = \exp_e^t(\lambda_0)$. Moreover, the exponential map also satisfies the reflection property:
\begin{equation}
    \label{exprefl}
    \exp_e^t(r,  - \theta, -w ) = (- x(t), y(t), -z(t)), 
\end{equation}
as well as the rotation properties:
\begin{equation}
\label{exprot}
\begin{gathered}
\exp_e^t(r, \theta + \tfrac{\pi_q}{2}, w) = (-y(t), x(t), z(t)), \ \exp_e^t(r,\theta + \pi_q,w) = (-x(t), -y(t), z(t)), \\
\exp_e^t(r,  \theta + \tfrac{3 \pi_q}{2},w) = (y(t), -x(t), z(t)).
\end{gathered}
\end{equation}
\cref{fig:geodesicsproj} can also help understand geometrically the symmetries \cref{exphomog}, \cref{exprefl} and \cref{exprot}.





From Pontryagin's maximum principle, we can conclude that if $\gamma(t)$ is a length-minimiser parametrised by constant speed starting from $e$, then $\gamma(t) = \exp^t(r, \theta, w)$ for some $(r,\theta, w) \in \T^*(\mathds{H})$. 

If $\gamma(t) = (x(t), y(t), z(t))$ is a horizontal curve ($t \in \interval{0}{1}$) from the origin, then 
\[
z(t) = \frac{1}{2}\int_0^t (x(s) \dot{y}(s) - \dot{x}(s) y(s)) \diff s,
\]
and $z(1)$ corresponds to the (signed) area swept by $(x(t),y(t))$ in $\R^2$. Furthermore, $\gamma$ is a length minimiser between $\gamma(0)$ and $\gamma(1)$ if and only if the curve $(x(t), y(t))$ has minimal $\ell^p$-length in $\R^2$ among all the paths in $\R^2$ joining $(0, 0)$ to $(x(1), y(1))$. To find a length minimiser between $e$ and $(x_1, y_1, z_1) \in \mathds{H}$, one must therefore solve an isoperimetric problem: find a path $(x(t), y(t))$ from $(0, 0)$ to $(x_1, y_1)$ of minimal $\ell^p$-length that sweep an area $z_1$. 

Busemann proved the existence and uniqueness to such isoperimetric problems. In \cite{busemann1947}, the classical isoperimetric inequality in the Euclidean plane was extended to a arbitrary normed planes. The curve of a fixed area that minimises length is not the unit circle, but the dual unit circle. In our case (see also \cite[Section 4.]{noskov2008}), this means that if $(x_1, y_1) = (0, 0)$, then $(x(t), y(t))$ must be obtained by the $\ell^q$-arc sweeping an area of $z_1$, i.e. by dilating and translating $\mathds{S}_q$ in order to achieve an area of $z_1$ (see \cref{fig:geodesicsproj}). There are infinitely many ways to accomplish that. If $(x_1, y_1) \neq (0, 0)$, then $(x(t), y(t))$ is uniquely obtained as a subpath of a dilated and translated $\mathds{S}_q$ that achieves an area of $z_1$. Therefore, every two points in the $\ell^p$-Heisenberg group can be joined by a length minimiser: there is a unique one if the two points don't have the same $z$-coordinate, otherwise there are infinitely many length minimisers between them.

\begin{remark}
    Busemann's solution to the Finsler isoperimetric problem in the plane could be used to obtain the shape of length-minimisers and cut point,
    however,
    it does not directly give the constant-speed parametrisation of those curves. Applying Pontryagin's Maximum Principle to the squared energy instead gives the right parametrisation of the extremals. Here, we recall the equivalence between finding length minimisers and the $\ell^p$-isoperimetric problem in order to obtain the cut time easily.
\end{remark}

The values of $r$, $w$, and $\theta$ such that the corresponding curve $\interval{0}{1} \to \mathds{H} : t \mapsto \exp_e^t(r, \theta, w)$ is the unique (up to reparametrisation) length-minimiser between its endpoints are easily found from the considerations above. The subset of $\T^*(\mathds{H})$ for which this property is satisfied is often called the \textit{cotangent injectivity domain} of the exponential map. This domain corresponds to determining the \textit{optimal synthesis} of the $\ell^p$-Heisenberg, the set of all geodesics together with their cut time:
\[
t_{\mathrm{cut}}[r, \theta, w] := \sup\{ T > 0 \mid \gamma : \interval{0}{T} \to \mathds{H} : t \mapsto \exp_e^t(r, \theta,w) \text{ is length minimising} \}.
\]

\begin{proposition}
    \label{cotinjdom}
    The cut time for the $\ell^p$-Heisenberg group is given by
    \[
    t_{\mathrm{cut}}[r, \theta,w] = \frac{2 \pi_q}{|w|},
    \]
    where by convention $\frac{2\pi_q}{0}=+\infty$.
    Furthermore, the cotangent injectivity domain from $e$ is the set
    \[
    D := \left\{ (r, \theta, w) \in \T^*_e(\mathds{H}) \mid r > 0, \text{and } |w| < 2 \pi_q \right\} \subseteq \T^*_e(\mathds{H}).
    \]
    The map $\exp_e : D \to \exp_e(D)$ is a homeomorphism and $\mathds{H} \setminus \exp_e(D) = \{ (x, y, z) \in \mathds{H} \mid x = y = 0\}$ is a negligible set.
    In particular,
    the $\ell^p$-Heisenberg group has a negligible cut locus, in the sense of \cref{def:negligiblecutlocus}, when $p \in \ointerval{1}{\infty}$.
\end{proposition}


If $t > 0$, the exponential map $\exp_e^t$ is a homeomorphism from
\[
D^t := \left\{ (r, \theta, w) \in \T^*_e(\mathds{H}) \mid r > 0, \text{and } \abs{wt} < 2 \pi_q \right\}
\]
onto its image, by \cref{cotinjdom}. We can improve the regularity of $\exp_e^t$ by removing a negligible set of initial covectors corresponding to the loss of smoothness in \cref{eq:x(t)y(t)z(t)} due to \cref{def:sinpcosp}. Define the sets
\[
\mathcal{S}_0:=\left\{(r,\theta,0)\in D^t \mid  \theta\in \frac{\pi_q}{2}\Z\right\}, \ \ \mathcal{S}^t_1:=\left\{(r,\theta,w)\in D^t \mid wt+\theta\in\frac{\pi_q}{2}\Z,w\neq 0\right\},
\]
and
\[
\mathcal{R}^t:=D^t\setminus(\mathcal{S}_0\cup\mathcal{S}_1^t).
\]

The image of $\mathcal{S}_0$ by $\exp^t$ is the points on the horizontal lines of direction $\theta\in\frac{\pi_q}{2}\Z$,
and that of $\mathcal{S}_1^t$ is the points at the corner of a geodesic $\gamma(t)=\exp_e^t(r,\theta,w)$.
On these singular sets and complementary regular points,
the exponential map has the following regularity.
\begin{lemma}
    For $t>0$,
    the exponential map $\exp_e^t:D^t\to \mathds{H}$ has the following regularity:
    \begin{enumerate}[label=\normalfont(\roman*)]
    \item For all $q \in \ointerval{1}{\infty}$, the exponential map $\exp_e^t$ is analytic on $\mathcal{R}^t$,
    \item If $q \notin \Z$, then $\exp_e^t$ is $C^{\lfloor q \rfloor}$ on $D^t \setminus \mathcal{S}_0 = \mathcal{R}^t \cup \mathcal{S}_1^t$, 
    \item If $q \notin \Z$, then $\exp_e^t$ is $C^{\lfloor q - 1 \rfloor}$ on $D^t = \mathcal{R}^t \cup \mathcal{S}_1^t \cup \mathcal{S}_0$.
    \end{enumerate}
    If $q \in \Z$, then $\exp_e^t$ is analytic on $D^t$.
\end{lemma}
\begin{proof}
    Recall from \cref{remark:regutrig} that when $q \notin \mathds{Z}$, the map $\theta \mapsto \sin_q(\theta)$ (resp. $\theta \mapsto \sin_p(\theta^\circ)$) is analytic everywhere except at $\theta \in \frac{\pi_q}{2} \mathds{Z}$ where it is $C^{\lfloor q \rfloor}$ (resp. $C^{\lfloor q - 1 \rfloor}$). 
    
    When $q \in \mathds{Z}$, the $p$-trigonometric functions are analytic everywhere. In (i), the set $\mathcal{R}^t$ does not contain any point at which the $p$-trigonometric functions lose analycity. In (ii), the parameter $w$ does not vanish and so the expression \cref{eq:x(t)y(t)z(t)} is only $C^{\lfloor q \rfloor}$ at $w t + \theta \in \frac{\pi_q}{2}\mathds{Z}$. In (iii), we are also allowing for $w = 0$, and thus \cref{eq:x(t)y(t)z(t)atw=0} is $C^{\lfloor q - 1\rfloor}$ when $\theta \in \frac{\pi_q}{2}\mathds{Z}$.
\end{proof}
    In the following sections,
    these singular sets play a central role.


\section{Jacobian of the exponential map}

\label{sec:jacobian}

By $\mathcal{J}^t(r, \theta,w)$, we will denote the Jacobian determinant of the exponential map at the identity, i.e.,
\[
\mathcal{J}^t(r, \theta,w) := \det \mathrm{Jac}(\exp_e^t)(r, \theta,w).
\]
The discussion of the previous section shows that for a given $t \in (0,1]$, the exponential map $\exp_e^t$ is a diffeomorphism on $\mathcal{R}^t$, and thus $\mathcal{J}^t$ is smooth on the same domain.
On the singular set $\mathcal{S}_1^t$,
the exponential map $\exp^t_e$ is at least $C^1$.
Therefore its Jacobian $\J^t$ is well defined, but its derivative has singularities when $p>2$. On the singular point $\mathcal{S}_0$,
the exponential map is only continuous,
hence the Jacobian may not be well-defined.
In this section, we shall consider the following three cases, which are sorted by regularity; the Jacobian on $\mathcal{R}^t\cup\mathcal{S}_1^t$ in \cref{sec:4.1}, on $\mathcal{S}_0$ in \cref{sec:4.2}, and its differential on $\mathcal{S}_0\cup\mathcal{S}_1^t$ in \cref{sec:derivativejac}.

\subsection{Jacobian at regular points}
\label{sec:4.1}

If $(r,\theta,w)\in \mathcal{R}^t\cup\mathcal{S}_1^t$,
the exponential map $\exp_e^t$ is at least $C^1$,
thus its Jacobian $\J^t$ is continuous.
The homogeneity and symmetry properties \cref{exphomog}, \cref{exprefl}, and \cref{exprot} implies that
\begin{equation}
\label{jacsymmetry}
\begin{gathered}
\mathcal{J}^t(r, \theta, w) = \frac{1}{c^2} \mathcal{J}^{c t}(r/c,  \theta,w/c), \ \mathcal{J}^t(r, -\theta,-w) = \mathcal{J}^{t}(r, \theta, w),\\
\mathcal{J}^t(r, \theta + n \tfrac{\pi_q}{2},w) = \mathcal{J}^{t}(r, \theta,w), \text{ for } c \neq 0 \text{ and } n \in \mathds{Z}.
\end{gathered}
\end{equation}
We shall also write $\mathcal{J}$ instead of $\mathcal{J}^1$ for simplicity (resp. $D$ for $D^1$). We can obtain an expression for the Jacobian by performing a simple computation involving Shelupsky's trigonometric functions.

\begin{lemma}
    \label{regularityF}
    For all $(r, \theta,w) \in \mathcal{R}^t\cup\mathcal{S}_1^t$ with $w \neq 0$, we have
    \begin{equation}
        \label{DetJacpolar}
        \begin{aligned}
        \mathcal{J}^t(r, \theta,w)= \frac{r^3 t}{w^4} \Big[ 2 &- \big( \sin_q(w t + \theta) \sin_p(\theta^\circ) + \cos_q(w t + \theta) \cos_p(\theta^\circ) \big)\\
        {}&- \big( \sin_p(w t + \theta)^\circ \sin_q(\theta) + \cos_p(w t + \theta)^\circ \cos_q(\theta) \big) \\
        {}& - w t \big( \sin_p(w t + \theta)^\circ \cos_p(\theta^\circ) - \cos_p(w t + \theta)^\circ \sin_p(\theta^\circ)\big) \Big],
    \end{aligned}
    \end{equation}
    and $\mathcal{J}^t(r, \theta,0)$ for $\theta \notin \frac{\pi_q}{2}\Z$ is obtained by smoothly extending \cref{DetJacpolar} as $w\to 0:$
    \begin{equation}\label{jacobianwis0}
    \J^t ( r , \theta , 0 ) = \frac{ (q-1)^2 r^3 t^5}{12} \left| \cos_q \theta \sin_q \theta \right|^{ 2 q - 4 }.
    \end{equation}
\end{lemma}

In the identity \cref{DetJacpolar}, it is worth recalling that the functions $\sin_p(x^\circ)$ and $\cos_p(x^\circ)$ have explicit forms given by:
\[
\sin_p(x^\circ) := \abs{\sin_q(x)}^{q-1} \sgn(\sin_q(x)), \ \cos_p(x^\circ) := \abs{\cos_q(x)}^{q-1} \sgn(\cos_q(x)).
\]

For its proof,
we introduce the reduced Jacobian function $\J_R : \R^2 \to \mathds{R}$ by setting
\begin{equation}
        \label{Fwtheta}
        \begin{aligned}
        \J_R(\theta,w) :={} 2& - \big( \sin_q(w + \theta) \sin_p(\theta^\circ) + \cos_q(w + \theta) \cos_p(\theta^\circ) \big)\\
        {}&- \big( \sin_p(w + \theta)^\circ \sin_q(\theta) + \cos_p(w + \theta)^\circ \cos_q(\theta) \big) \\
        {}& - w \big( \sin_p(w + \theta)^\circ \cos_p(\theta^\circ) - \cos_p(w + \theta)^\circ \sin_p(\theta^\circ) \big).
    \end{aligned}
\end{equation}
In other words, we define $\J_R(\theta,w)$ so that
\[
\mathcal{J}^t(r, \theta,w) = \frac{r^3 t}{w^4} \J_R(\theta,w t)
\]
holds for all $(r, \theta, w) \in D^t$. This definition of $\J_R$ allows us to simplify the expression for $\mathcal{J}^t$ in terms of $\J_R$, which will be useful in the remainder of the work.

\begin{remark}
When $p=2$ (and thus $q = 2$), the expression \cref{DetJacpolar} simplifies to
\[
\mathcal{J}^t(r, \theta,w) = \frac{r^3 t}{w^4} \left(\sin\left(\frac{w t}{2}\right) - \frac{w t}{2} \cos\left(\frac{w t}{2}\right)\right)\sin\left(\frac{w t}{2}\right),
\]
by the standard trigonometric addition formulas. This coincides with the well-known computations in the sub-Riemannian Heisenberg group, see \cite[Proposition 1.12.]{juillet2009}.  
\end{remark}

\begin{proof}
The case $w\neq 0$ is a straightforward and long computation, so we omit the detail.
We will compute the limit to $w\to 0$ by using L'Hospital's rule.
    By induction, we obtain that
    \begin{align}
    \label{formuladiffJR}
        \begin{split}
        \partial_w^k\J_R(\theta,0)=(k-1)(\cos_q^{(k)}\theta\sin_q^{(1)}\theta-&\sin_q^{(k)} \theta \cos_q^{(1)}\theta)\\
        &+\cos_q^{(k+1)}\theta\sin_q\theta-\sin_q^{(k+1)}\theta\cos_q\theta.\\
        \end{split}
    \end{align}
    Here, $\sin_q^{(k)}\theta$ and $\cos_q^{(k)}\theta$ denote the $k$-th derivatives of $\sin_q\theta$ and $\cos_q\theta$, respectively. Using \cref{lemtridiff}, we can evaluate the derivatives of $\J_R(\theta, w)$ with respect to $w$ at $w=0$. We find that $\partial_w\J_R(\theta,0) = \partial_w^2\J_R(\theta,0) = \partial_w^3\J_R(\theta,0) = 0$, while
    \[
    \partial_w^4\J_R(\theta,0) = 2 (q-1)^2 \abs{\cos_q(\theta)\sin_q(\theta)}^{2q-4}.
    \]
    The conclusion follows by L'Hospital's rule.
\end{proof}

\begin{remark}
\label{remark:JRcontinuous}
    The function $\mathcal{J}_R$ is continuous everywhere, and it is smooth on any open domain where (the projection of) the singular sets $\mathcal{S}_{0} \cup \mathcal{S}_{1}^1$ have been removed.
\end{remark}

It will be useful for the rest of this work to write down the derivative of $\J_R$ with respect to $w$:
\begin{equation}
        \label{dwFwtheta}
        \begin{aligned}
        \partial_w \J_R(\theta,w) :=& (q - 1) \abs{\cos_q(w + \theta)\sin_q(w + \theta)}^{q - 2} \\
        & \qquad \times  \Big[ \sin_q(w + \theta) \cos_q(\theta) - \cos_q(w + \theta)\sin_q(\theta) \\
        & \qquad \qquad - w \big( \sin_q(w + \theta) \sin_p(\theta^\circ) + \cos_q(w + \theta) \cos_p(\theta^\circ) \big)\Big].
    \end{aligned}
\end{equation}

On the region $\mathcal{R}^t\cup\mathcal{S}_1^t$,
we are going to check that the Jacobian $\mathcal{J}^t$ is positive by studying the sign of the reduced Jacobian $\J_R$.

\begin{lemma}
    \label{positivityJac}
    For every $(\theta, w) \in \interval{0}{2\pi_q}\times \ointerval{0}{2 \pi_q} $, we have $\mathcal{J}_R(\theta, w) > 0$. Furthermore, $\J_R(\theta, w) = 0$ if and only if $w \in 2 \pi_q \mathds{Z}$.
\end{lemma}

\begin{proof}
By the symmetries \cref{jacsymmetry},
we can assume $\theta\in[0,\frac{\pi_q}{2}]$.
We express \cref{Fwtheta} in a more compact notation as $\J_R(\theta,w) = \J_{R,1}(\theta, w) + w \J_{R,2}(\theta, w)$.

By the duality inequality \cref{eq:ineqtrigpq}, it is clear that $\J_{R,1}(\theta, w)$ is non-negative for every $(\theta, w)$, and this term vanishes if and only if $w + \theta = \theta \operatorname{mod} 2 \pi_q$, that is, $w \in 2 \pi_q \mathds{Z}$.

Geometrically, the term $\J_{R,2}(\theta, w)$ can be interpreted (see \cref{scalarproduct}) as the scalar product of the two vectors given by
\[
(\cos_p\theta^\circ,\sin_p\theta^\circ)
~~\text{and}~~
(-\sin_p(w+\theta)^\circ,\cos_p(w+\theta)^\circ).
\]
\begin{figure}
  \centering
  \includegraphics[width=0.8\textwidth]{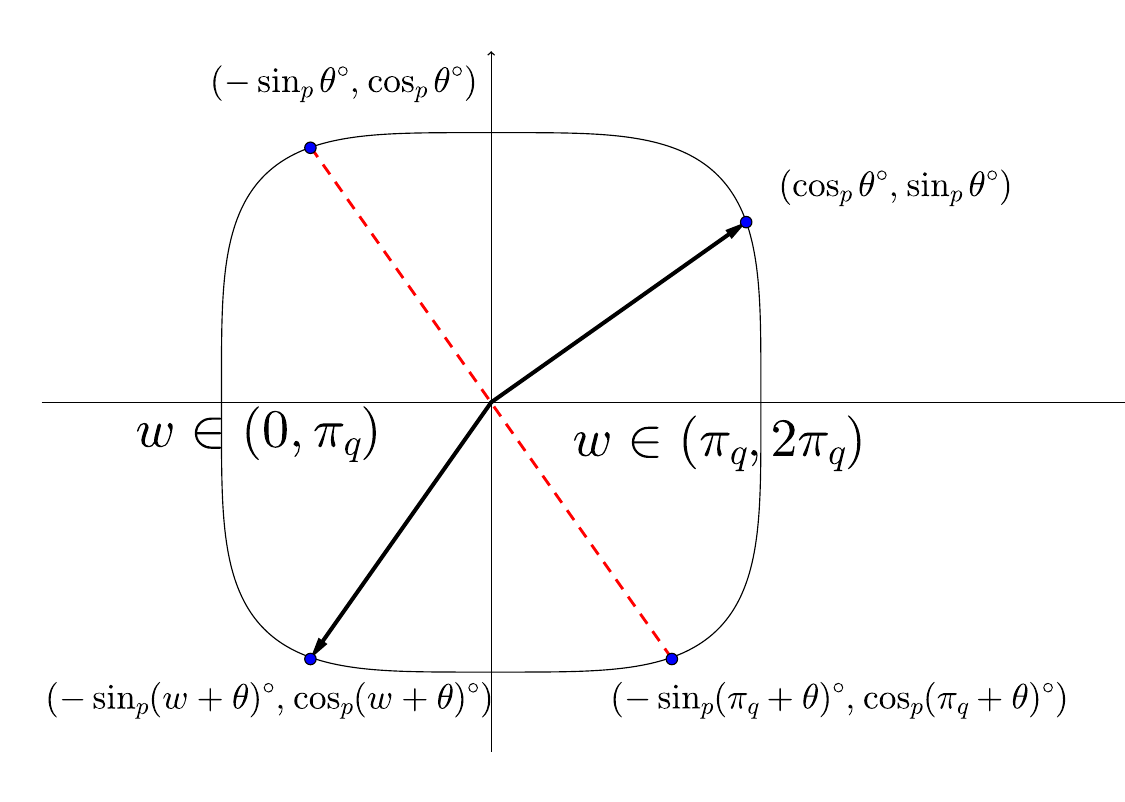}
  \caption{The positivity of $\J_{R,2}$ for $w\in(\pi_q,2\pi_q)$}
  \label{scalarproduct}
\end{figure}
As such, we know that $\J_{R,2}(\theta, w) < 0$ if $w \in \ointerval{0}{\pi_q}$, $\J_{R,2}(\theta, w) > 0$ if $w \in \ointerval{\pi_q}{2 \pi_q}$ and $\J_{R,2}(\theta, w) = 0$ if $w \in \pi_q \mathds{Z}$. Therefore, $\J_R(\theta, w) > 0$ for $(\theta, w) \in \interval{0}{\frac{\pi_q}{2}} \times \ointerval{\pi_q}{2 \pi_q}$.

To prove the positivity of $\J_R$ for $(\theta, w) \in \interval{0}{\frac{\pi_q}{2}} \times \ointerval{0}{\pi_q}$, we now need to consider the partial derivative of $\J_R$ with respect to $w$, already seen in \cref{dwFwtheta}, which we express as $\partial_w \J_R(\theta, w) = (q - 1) \abs{\cos_q(w + \theta)\sin_q(w + \theta)}^{q - 2} \K(\theta, w)$.

The directional derivative of $\K$ towards $w-\theta$ is
\[
(\partial_w - \partial_\theta) \K(\theta, w) = (q-1)w\abs{\cos_q\theta\sin_q\theta}^{q-2}\left(\sin_q(w+\theta)\cos_q\theta-\cos_q(w+\theta)\sin_q\theta\right).
\]
Using the same reasoning as before and considering the scalar product of the two vectors:
\[
(\cos_q\theta,\sin_q\theta)
~~\text{and}~~
(\sin_q(w+\theta),-\cos_q(w+\theta)),
\]
we can conclude that $(\partial_w - \partial_\theta) \K(\theta, w)$ is positive for $w \in \ointerval{0}{\pi_q}$ and $\theta \in \interval{0}{\frac{\pi_q}{2}}$.

Therefore, to complete the proof, we need to verify that $\K(\theta, w)$ is non-negative for $w = 0$ and $\theta \in \interval{0}{\frac{\pi_q}{2}}$, as well as for $\theta = \frac{\pi_q}{2}$ and $w \in \ointerval{0}{\pi_q}$. We find that
\[
\K(\tfrac{\pi_q}{2}, w)=\sin_qw-w\cos_qw, \text{ and } \K(\theta, 0) = 0.
\]

By considering the scalar product of the two vectors
\[(\cos_qw,\sin_qw)~~\text{and}~~(-w,1),\]
we can show that $\mathcal{K}(\frac{\pi_q}{2},w)\geq 0$, see \cref{proofpositivityfigure}. If $\cos_q w \leq 0$, there is nothing to prove so we only consider $\cos_q w \geq 0$.
\begin{figure}
\centering
\includegraphics[width=0.8\textwidth]{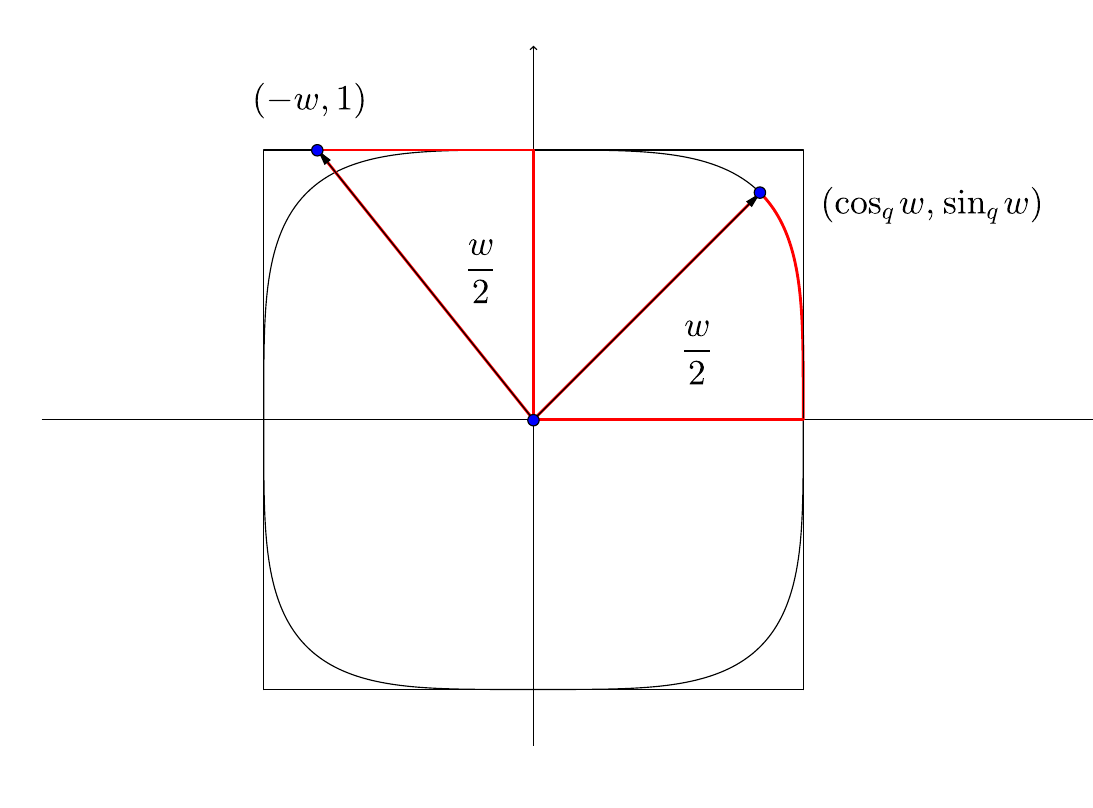}
\caption{The positivity of $\sin_qw-w\cos_qw$.}
\label{proofpositivityfigure}
\end{figure}
Indeed, recall that the area of the $\ell^q$-fan spanned by $(0,0),(1,0),(\cos_qw,\sin_qw)$ is $\frac{w}{2}$, as shown in \cref{def:sinpcosp}. On the other hand, $(-w,1)$ is the point such that the area of the $\ell^{\infty}$-fan spanned by $(0,0),(0,1),(-w,1)$ is also $\frac{w}{2}$. It is clear that the $\ell^{\infty}$ ball includes the $\ell^q$ ball, so the $\ell^2$-angle between $(\cos_qw,\sin_qw)$ and $(-w,1)$ must be less than $\frac{\pi_2}{2}$. This shows that $\K(\frac{\pi_q}{2},w)\geq 0$ for all $w \in \ointerval{0}{\pi_q}$.

In conclusion, since $(\partial_w-\partial_\theta)\K(\theta,w)$ is positive, it implies that $\K(\theta,w)$ is positive for $(\theta, w) \in \interval{0}{\frac{\pi_q}{2}} \times \ointerval{0}{\pi_q}$. Therefore, $\partial_w\J_R(\theta,w)$ is positive on the same domain, and checking that $\J_R(\theta,0) = 0$ shows that $\J_R > 0$ when $w \notin 2 \pi_q \Z$.

To complete the proof, it suffices to observe $\J_R(\theta, w) = 0$ when $w \in 2 \pi_q \Z$ by simply substituting these values into \cref{Fwtheta}.
\end{proof}

\begin{remark}
The derivative $\partial_w\J_R(\theta,w)$ may be discontinuous and infinite on the singular set $\{w+\theta\in\frac{\pi_q}{2}\Z\}$, however,
the monotonicity arguments still follow by the continuity of $\J_R$.
\end{remark}

\begin{proposition}[Positivity of Jacobian on $\mathcal{R}^t\cup\mathcal{S}_1^t$]\label{prop:jacobianpositive}
\label{prop:positivityJt}
For $t > 0$ and  $(r,\theta,w)\in \mathcal{R}^t\cup\mathcal{S}_1^t$,
$\J^t(r,\theta,w)>0$.
\end{proposition}

\begin{proof}

    Recall that $\J^t(r,\theta,w)=\frac{r^3t}{w^4}\J_R(\theta,wt)$,
    where $\theta\in[0,2\pi_q]$ and $wt\in(0,2\pi_q)$.
    Therefore by Lemma \ref{positivityJac},
    $\J^t$ is positive when $w \neq 0$.
    When $w = 0$, we must have $\theta \notin \frac{\pi_q}{2}\Z$ by definition of $\mathcal{R}^t\cup\mathcal{S}_1^t$ and the positivity then follows from \cref{jacobianwis0}.
\end{proof}

\subsection{Jacobian at singular points}
\label{sec:4.2}

Next we consider the singular points $\mathcal{S}_0$.
Since the exponential map is not differentiable on $\mathcal{S}_0  \subseteq D^t$ for $p>2$,
the Jacobian $\J^t$ itself has singularities.
The singular set $\mathcal{S}_0$ is independent of $t$,
we restrict ourselves to $t=1$.
The computation of limits of $\J_R$ as $(\theta, w)$ approaches the singular set $\mathcal{S}_0$ is achieved by considering asymptotic estimates of $\J_R$.
By the symmetry of the Jacobian,
we can assume $\theta=0$.
We will compute the asymptotic of $\J_R(\theta,w)$ at $(0,0)$.
\begin{lemma}
    \label{asymptoticFat00}
    Suppose that both $|w+\theta|$ and $|\theta|$ are strictly smaller than the radii of convergence at $\theta = 0$ of the $q$-trigonometric functions $\sin_q(\theta)$ and $\cos_q(\theta)$.
    Then, the function $\J_R$ has the uniformly convergent series representation
    \begin{align*}
        \mathcal{J}_R(\theta, w) ={}& \frac{q - 1}{q^2 (q + 1)} \big(\abs{w + \theta}^{2 q} + \abs{\theta}^{2 q}\big) + \frac{2 (q - 1)^2}{q^2} \abs{w + \theta}^q \abs{\theta}^q \\
        & \qquad - \frac{q - 1}{q + 1} \left(\abs{w + \theta}^{q}\abs{\theta}^{q - 2} + |w+\theta|^{q-2}|\theta|^q \right)(w + \theta)\theta + E(\theta,w),
    \end{align*}
    where the remainder term $E$ is the infinite series
    \begin{equation}
        \label{asymptoticFat00remainder}
        E(\theta,w)=\sum_{\alpha,\beta,\gamma,\delta}c_{\alpha,\beta,\gamma,\delta}|w+\theta|^\alpha|\theta|^\beta(w+\theta)^\gamma\theta^\delta
    \end{equation}
    with $\alpha+\beta+\gamma+\delta\geq 3q$ and $c_{\alpha,\beta,\gamma,\delta}\in\R$.
\end{lemma}

\begin{proof}
    The proof of this result involves substituting the expansions found in \cref{lemtriexpansion} into the expression \cref{Fwtheta} for $\mathcal{J}_R$. 
\end{proof}

To estimate the limit of $\J$,
\cref{asymptoticFat00} is not enough.
Indeed,
since $\J=\frac{r^3t}{w^4}\J_R(\theta,w)$,
we need to consider the limit of $\frac{\J_R(\theta,w)}{w^4}$ as $(\theta,w)\to (0,0)$.
If a sequence of points $(\theta,w)$ converges to $(0,0)$ with $|w|=o(|\theta|)$,
then we cannot control the error term $\frac{E_3}{w^4}=O\left(\frac{|\theta|^{3q}}{w^4}\right)$.

It is quite hard to compute the higher order terms of the fractional Taylor series of $q$-trigonometric functions.
Hence, we will use the Taylor series expansion of $\J_R$ at $(\theta,0)$.
In the following lemma,
we can see that the radius of convergence of the Taylor expansion around $(\theta,0)$ can cover the above cases.

\begin{lemma}\label{lem:case1}
    For $\theta\in \left(-\frac{\pi_q}{4},\frac{\pi_q}{4}\right) \setminus \{0\}$,
    the radius of convergence of the Taylor series at $(\theta,0)$
    \[
    \J_R(\theta,w) = \frac{w^4}{12} (q-1)^2 \abs{\cos_q(\theta)\sin_q(\theta)}^{2q-4} + \sum_{k=5}^\infty\frac{1}{k!}\partial_w^k\J_R(\theta,0) w^k,
    \]
    is at least $|\theta|/(q+1)$.
\end{lemma}

\begin{proof}

  From the formula \cref{formuladiffJR},
  the radius of convergence of $w \mapsto \J_R(\theta, w)$ is equal to that of $q$-trigonometric functions,
  for which a bound is given in \cref{radiusconvsinp}.
\end{proof}

With these expansions, we are able to compute some limits of $\mathcal{J}(r, \theta, w)$ as points approach $\mathcal{S}_0$. The techniques used in the proof will also be used later on in \cref{sec:geoddim}, when dealing with the geodesic dimension.

\begin{proposition}\label{prop:jacobianons0}
For $(r_0,\theta_0,0)\in \mathcal{S}_0$,
we have
\[\lim_{(r,\theta,w)\to(r_0,\theta_0,0)}\J(r,\theta,w)=\begin{cases}
    \frac{r_0^3}{12} & (p=2),\\
    0 & (p<2).
\end{cases}\]
If $p>2$,
then for almost all directions $(u,v)\in \mathds{S}
^1 \subseteq \R^2$,
we have
\[\liminf_{s\to 0}\J(r_0,\theta_0+su,sv)=+\infty.\]
\end{proposition}

Computing the full limits when $p>2$ would require a more detailed analysis, which we do not pursue here.

\begin{proof}
If $p\leq 2$,
then the Jacobian is continuous on these points and the limit is given by the limit $\theta\to \theta_0$ of \cref{jacobianwis0}.

From now on, we will thus assume that $p>2$.
Since $\J(r,\theta,w)=r^3\frac{\J_R(\theta,w)}{w^4}$,
we consider the limit of $\frac{\J_R(\theta,w)}{
w^4}$ as $(\theta,w)\to (0,0)$.
    Fixing a constant $\epsilon_q\in\ointerval{0}{q+1}$, we can classify points around $(0,0)$ into the following three cases.
    \begin{enumerate}[label=\normalfont\arabic*)]
    \item $|\epsilon_qw|\leq |\theta|$;
    \item $|w|\leq |\theta|\leq |\epsilon_qw|$;
    \item $|\theta|\leq |w|$.
\end{enumerate}
Up to a subsequence, we only need to consider the limit $(\theta,w)\to (0,0)$ in these three cases respectively.\smallskip

\noindent\textbf{Case 1} ($|\epsilon_qw|\leq |\theta|$)\textbf{.}\smallskip

In this case,
we can apply \cref{lem:case1}.
In addition,
suppose that $|\theta|$ is smaller than the radius of convergence of the $q$-trigonometric functions at $0$.
Combining \cref{lemtriexpansion} and \cref{formuladiffJR},
we can write
\begin{equation}
   \label{eq:expansiontheta1}
   \partial_w^k\J_R(\theta,0)=\sum_{l=2}^{\infty}c_{k,l}^1|\theta|^{lq-k}\tau_k, 
\end{equation}
where
\[\tau_k=\begin{cases}
   1 & (k:~\text{even}),\\
     \sgn(\theta) & (k:~\text{odd}),
\end{cases}\]
and where $c_{k,l}^1$ are constants explicitly given by
\begin{align*}
c_{k,l}^1=\sum_{n=0}^la_nb_{l-n}\{&(k-1)((l-n)q)_k(nq+1)-(k-1)(nq+1)_k(l-n)q\\
&+((l-n)q)_{k+1}-(nq+1)_{k+1}\}.
\end{align*}
In the formula above, we have used the descending product notation
\[
(\alpha)_k:=\alpha(\alpha-1)\cdots(\alpha-k+1).
\]
Notice that a priori the series start from $l=0$.
However, it can easily be checked that $c_{k,0}^1=c_{k,1}^1=0$ for all $k\geq 4$,
and $c_{4,2}^1=2(q-1)^2>0$.

For fixed $\theta$ and $w$ satisfying $|\epsilon_qw|\leq |\theta|$, we can substitute \cref{eq:expansiontheta1} in \cref{lem:case1} to obtain
\[
\J_R(\theta,w)=\sum_{k=4}^{\infty}\sum_{l=2}^{\infty}\frac{c_{k,l}^1}{k!}w^{k}|\theta|^{lq-k}\tau_k.
\]
We note that the double series is absolutely convergent in $k$ and in $l$ by Tonelli's theorem. Consequently,
we can exchange the sums and we get
\begin{equation}
\label{series:case1}
\begin{aligned}
\J_R(\theta,w)={}&w^4|\theta|^{2q-4}\sum_{k=4}^{\infty}\frac{c_{k,2}^1}{k!}\left(\frac{w}{\theta}\right)^{k-4}+\sum_{k=4}^{\infty}\sum_{l=3}^{\infty}\frac{c_{k,l}^1}{k!}w^{k}|\theta|^{lq-k}\tau_k\\
=:{}&w^4|\theta|^{2q-4}P_1\left(\frac{w}{\theta}\right)+E_1(\theta,w),
\end{aligned}
\end{equation}
where $P_1 : \interval{-\frac{1}{\epsilon_q}}{\frac{1}{\epsilon_q}} \to \R$ is a real analytic function,
and $E_1$ is an error term of order at most $w^4|\theta|^{3q-4}$.

By the positivity of Jacobian on regular points seen in \cref{prop:positivityJt},
the principal term is non-negative on this domain.
Moreover,
since $P_1$ is an analytic function on a closed interval,
there are (if they exist at all) finitely many zero points.
In particular,
the analytic function $P_1\left(\frac{w}{\theta}\right)$ is positive almost everywhere.

By the compactness of the domain $\interval{-\frac{1}{\epsilon_q}}{\frac{1}{\epsilon_q}}$,
we can assume that $P_1\left(\frac{w}{\theta}\right)$ converges to some non-negative number.
If the limit is positive,
we have
\[\frac{\J_R(\theta,w)}{w^4}\sim |\theta|^{2q-4}P_1\left(\frac{w}{\theta}\right)\to +\infty.\]
The Jacobian diverges to $+\infty$ almost surely in the sense of the statement,
since there are only finitely many zero points of $P_1$.\smallskip

\noindent\textbf{Case 2} ($|w|\leq |\theta|\leq |\epsilon_qw|$)\textbf{.}\smallskip

In this case, we apply \cref{asymptoticFat00}.
Combined with the binomial theorem
\[
|w+\theta|^\alpha=|\theta|^{\alpha}\sum_{k=0}^\infty\binom{\alpha}{k}\left(\frac{w}{\theta}\right)^k,
\]
we have
\begin{equation}\label{series:case2}
    \begin{aligned}
        \J_R(\theta,w)={}&w^4|\theta|^{2q-4}\sum_{k=4}^{\infty}c_k^2\left(\frac{w}{\theta}\right)^{k-4}+E(\theta,w)\\
        =:{}&w^4|\theta|^{2q-4}P_2\left(\frac{w}{\theta}\right)+E(\theta,w),
    \end{aligned}
\end{equation}
where the constants $c_k^2$ are explicitly written by
\[
    c_k^2= \frac{q-1}{q^2(q+1)}\left[\binom{2 q}{k}+\binom{0}{k}\right]+\frac{2(q-1)^2}{q^2}\binom{q}{k} -\frac{q-1}{q+1}\left[\binom{q+1}{k}+ \binom{q-1}{k}\right],
\]
where the function $P_2:[-1,1]\to\R$ is real analytic, and $E$ is given in \eqref{asymptoticFat00remainder}.
A priori the series starts from $k=0$,
however we can check that $c_k^2=0$ for $k=0,1,2,3$ and $c_4^2=2(q-1)^2 >0$.
Since $|\theta|\leq \epsilon_q|w|$,
the error term $E(\theta,w)$ is of order at most $|w|^{3q}$.
The rest of the arguments follow in the same way as in Case 1.\smallskip

\noindent\textbf{Case 3} ($|\theta|\leq |w|$)\textbf{.}\smallskip

We apply again \cref{asymptoticFat00}, together with the binomial series theorem
\[|w+\theta|^\alpha=|w|^{\alpha}\sum_{k=0}^\infty\binom{\alpha}{k}\left(\frac{\theta}{w}\right)^k.\]
Therefore, we have the following series representation
\begin{equation}
    \label{series:case3}
    \begin{aligned}
        \J_R(\theta,w)={}&|w|^{2q}\sum_{k=0}^{\infty}\Bigg[c_{k,0}^3+c_{k,2q}^3\left|\frac{\theta}{w}\right|^{2q}+c_{k,q}^3\left|\frac{\theta}{w}\right|\\
        &\hspace{60pt}+c_{k,q-1}^3\left|\frac{\theta}{w}\right|^{q-2}\frac{\theta}{w}+c_{k,q+1}^3\left|\frac{\theta}{w}\right|^{q}\frac{\theta}{w}\Bigg]\left(\frac{\theta}{w}\right)^k +E(\theta,w) \\
        =:{}&|w|^{2q}P_3\left(\frac{\theta}{w}\right)+E(\theta,w),
    \end{aligned}
\end{equation}
where $c_{k,\alpha}^3$ are constants given by
\begin{align*}
    &c_{k,0}^3=\frac{q-1}{q^2(q+1)}\binom{2 q}{k}, ~~~~ c_{k,2q}^3=\frac{q-1}{q^2(q+1)}\binom{0}{k},~~ ~~c_{k,q}^3=\frac{2(q-1)^2}{q^2}\binom{q}{k},\\
    &c_{k,q-1}^3=-\frac{q+1}{q-1}\binom{q+1}{k},~~~~c_{k,q+1}^3=-\frac{q+1}{q-1}\binom{q-1}{k},
\end{align*}
where the function $P_3:[-1,1]\to\R$ is a finite sum of fractional analytic functions (see the discussion after \cref{lemtriexpansion}),
and $E$ is given in \eqref{asymptoticFat00remainder}. Since $\J_R$ is nonnegative by \cref{positivityJac}, the function $P_3$ is also nonnegative on $\interval{-1}{1}$ (considering $|w|$ small enough). 
The only difference from cases 1 and 2 is that the role of analytic function of the leading term is replaced with $P_3$, which is a sum of fractional analytic functions. 
In cases 1 and 2,
we used the fact that the number of zero points of an analytic function on a closed interval is finite. 
We now explain that the same property holds also for $P_3$.
Indeed, $P_3$ is analytic (in the classical sense) on $\interval{-1}{1} \setminus \{0\}$,
so we only need to check that $0$ is not the accumulation point of the zero points of $P_3$.
Since the set of exponents
\[\{k,k+q-1,k+q,k+q+1,k+2q\}_{k=0,1,\dots}\]
is discrete in $\R$, we can rearrange it to be an increasing set of exponents and the desired property follows by adapting the proof of
the classical identity theorem for analytic functions.
This observation shows that the number of zero points of $P_3$ is finite.
The rest of the proof follows in the same way with cases 1 and 2.
\end{proof}

\subsection{Derivative of Jacobian at singular points}\label{sec:derivativejac}

As was observed in the previous section,
the Jacobian $\J^t$ is continuous and positive on the singular points $\mathcal{S}_1^t$. If $p>2$,
the singularity appears when we consider its derivative (more precisely, the derivative of the reduced Jacobian).

\begin{lemma}\label{regularitydF}
Fix $t\in(0,1]$ and $p>2$.
  Then there is a point $(\theta_0,w_0t)\in \{(\theta,wt)\mid wt+\theta\in\frac{\pi_q}{2}\Z,w\neq 0\}$ such that
   \[\partial_w\J_R(\theta,wt)\to+\infty\]
   as $(\theta,wt)\to (\theta_0,w_0t)$.
   
\end{lemma}

\begin{proof}
With the change of coordinates $wt$ to $w$,
We only need to prove the case $t=1$.
   Recall that the derivative of the reduced Jacobian is written by
   \[\partial_w\J_R(\theta,w)=(q-1)|\cos_q(w+\theta)\sin_q(w+\theta)|^{q-2}\K(\theta,w),\]
   where $\K(\theta,w)$ is the continuous bounded function given in the proof of Lemma \ref{positivityJac}.
   Since $p>2$,
   the H\"older conjugate $q$ is less than $2$,
   and the first factor $|\cos_q(w+\theta)\sin_q(w+\theta)|^{q-2}$ diverges to $+\infty$.
   As we saw in the proof of \cref{positivityJac},
   the function $\K$ is positive and bounded on $[0,\frac{\pi_q}{2}] \times (0,\pi_q)$.
   This concludes the lemma.
\end{proof}

\begin{remark}
If $w > \pi_q$ or $w = 0$, then the differential of the Jacobian may either converge to a constant or diverge to $-\infty$, depending on the behavior of the function $\K$. However, we will not pursue this analysis as it is not necessary for this work.
\end{remark}

Similarly, we can observe that if $p < 2$, the derivative of the reduced Jacobian $\partial_w \J_R (\theta, wt)$ vanishes on the singular set.

    \begin{lemma}\label{regularitydF2}
        Fix $t\in(0,1]$ and $p<2$.
  Then for any point $(\theta_0,w_0t)\in \{(\theta,wt)\mid wt+\theta\in\frac{\pi_q}{2}\Z\}$, we have
   \[\partial_w\J_R(\theta,wt)\to 0\]
   as $(\theta,wt)\to (\theta_0,w_0t)$.
    \end{lemma}

In the case $p<2$,
we will need to study $\partial_w \J_R(\theta, w)$ in the neighbourhood of the singular point $(0,0)$. Around this point, we have a quantitative improved version of the previous lemma
\begin{equation}
\label{asymptoticdwFat00}
    \begin{aligned}
        \partial_w \mathcal{J}_R(\theta, w) ={}& \frac{2(q - 1)}{q (q + 1)} \abs{w + \theta}^{2 (q - 1)} (w + \theta) + \frac{2 (q - 1)^2}{q} \abs{w + \theta}^{q-2} (w + \theta) \abs{\theta}^q \\
        &  - \frac{q - 1}{q + 1} \left[(q+1)|w+\theta|^2+(q-1)|\theta|^2\right]|w+\theta|^{q-2}|\theta|^{q-2}\theta + \partial_w E(\theta,w).
    \end{aligned}
    \end{equation}
    This is obtained by simply taking the derivative of \cref{asymptoticFat00}, which is allowed since the convergence in \cref{asymptoticFat00} is uniform in $w$. We will use the leading term of \cref{asymptoticdwFat00} to obtain a lower bound on the curvature exponent (see the proof of \cref{thm:MCPp<2}).

\section{
Measure contraction property of the \texorpdfstring{$\ell^p$}{lp}-Heisenberg group}

We turn our attention to the validity of synthetic notions of curvature bounds in the $\ell^p$-Heisenberg group. The metric measure space $(\mathds{H}, \diff, \mathcal{L}^3)$ is formed by equipping the Heisenberg group with its $p$-sub-Finsler distance and its Haar measure, namely, the Lebesgue measure $\mathcal{L}^3$.

The main focus of this section will be on the measure contraction property $\mathsf{MCP}(K, N)$.

In the next proposition, we show that the $\mathsf{MCP}(0, N)$ is, in fact, equivalent to a specific inequality on the Jacobian determinant of the $\ell^p$-Heisenberg group. 

\begin{proposition}
    \label{MCPcharacterisation1}
    The $\ell^p$-Heisenberg group $(\mathds{H}, \diff, \mathcal{L}^3)$ satisfies the $\mathsf{MCP}(0, N)$ for $N \geq 1$ if and only if
    \begin{equation}
        \label{MCPJacobian}
        \J_R(\theta,w t) \geq t^{N - 1} \J_R(\theta,w),
    \end{equation}
    for all $(t, \theta,w) \in \interval{0}{1}\times \ointerval{0}{\tfrac{\pi_q}{2}} \times \ointerval{0}{2 \pi_q}  \setminus \left\{ (t, \theta,w) \mid w t + \theta \in \tfrac{\pi_q}{2} \mathds{Z} \right\}$.
\end{proposition}

\begin{proof}
    Inequality \cref{def:MCP} with $K = 0$ reduces to
    \begin{equation}
    \label{MCPK=0}
    \mathfrak{m}(\Omega_t) \geq t^N \mathfrak{m}(\Omega),
    \end{equation}
    for all $g_0 \in \mathds{H}$, all $\Omega \subseteq \mathds{H}$, and all $t \in \interval{0}{1}$, where $\Omega_t$ denotes the $t$-intermediate set of $\Omega$ from $g_0$. Since $(\mathds{H}, \diff, \mathcal{L}^3)$ is invariant by left-translation, it is enough to consider $t$-intermediate sets from the identity, i.e. $g_0 = e$.
    
    We first show that \cref{MCPJacobian} implies \cref{MCPK=0}. From the symmetry properties of the Jacobian determinant \cref{jacsymmetry}, the inequality \cref{MCPJacobian} is equivalent to
    \[
    \abs{\mathcal{J}^t(r, \theta,w)} \geq t^N \abs{\mathcal{J}(r, \theta,w)},
    \]
    for all $(t, r, \theta,w) \in \interval{0}{1} \times \R_{>0}\times \ointerval{0}{2 \pi_q} \times \ointerval{-2 \pi_q}{2 \pi_q} \setminus \left\{ (t, r, \theta,w) \mid w t + \theta \in \tfrac{\pi_q}{2} \mathds{Z} \right\}=:\mathcal{D}$.

    By \cref{cotinjdom}, the measure of a Borel set $\Omega \subseteq \mathds{H}$ coincides with the measure of $\tilde{\Omega} := \Omega \setminus \exp_e^1(\mathcal{R}^1) = \Omega \setminus (\{ x = y = 0\} \cup \exp_e^1(\mathcal{S}_0) \cup \exp^1_e(\mathcal{S}_1^1))$. Moreover, the set $\tilde{\Omega}$ is equal to $\exp_e(A)$ for some $A \subseteq D = D^1$. Fix $t \in \interval{0}{1}$ and write
    \[
    \tilde{A} := A \setminus \exp_e^t(\mathcal{R}^t) = A \setminus \left\{(r,  \theta,w) \mid w t + \theta \in \tfrac{\pi_q}{2} \mathds{Z}\right\}
    \] 
    Then by a change of variables, we find that
    \begin{align*}
        \mathfrak{m}(\Omega_t) ={}& \mathfrak{m}(\tilde{\Omega}_t) = \mathfrak{m}(\exp_e^t(\tilde{A})) = \int_{\exp_e^t(\tilde{A})} \diff \mathfrak{m} = \int_{\tilde{A}} \diff (\exp_e^t)^* \mathfrak{m} = \int_{\tilde{A}} \abs{\mathcal{J}^t(r, \theta,w)}  \diff r \diff \theta \diff w \\
        \geq{}& t^N \int_{\tilde{A}} \abs{\mathcal{J}(r, \theta, w)}  \diff r \diff \theta \diff w = t^N \mathfrak{m}(\exp_e^1(\tilde{A})) = t^N \mathfrak{m}(\tilde{\Omega}) = t^N \mathfrak{m}(\Omega).
     \end{align*}
    
    We now prove the converse implication. For any $(t,r,\theta,w) \in \mathcal{D}$, we find by Lebesgue's differentiability theorem and using the inequality \cref{MCPK=0} that for $g_1 := \exp_e(r, \theta, w)$,
    \[
    \mathcal{J}^t(r, \theta, w) = \lim_{\epsilon \to 0} \frac{\mathfrak{m}(\exp_e^t(B(g_1, \epsilon)))}{\mathfrak{m}(B(g_1, \epsilon))} 
    \geq t^N \lim_{\epsilon \to 0} \frac{\mathfrak{m}(\exp_e^1(B(g_1, \epsilon))}{\mathfrak{m}(B(g_1, \epsilon))}
    = t^N \mathcal{J}(r, \theta, w).
    \]
\end{proof}

We now wish to express the inequality \cref{MCPJacobian} characterising the measure-contraction property as a differential inequality.

\begin{proposition}\label{proposition:mcp}
    The $\ell^p$-Heisenberg group $(\mathds{H}, \diff, \mathcal{L}^3)$ satisfies the $\mathsf{MCP}(0, N)$ for $N \geq 1$ if and only if for all $(\theta, w) \in   \ointerval{0}{\frac{\pi_q}{2}} \times \ointerval{0}{2 \pi_q}\setminus \{w + \theta \in \frac{\pi_q}{2} \mathds{Z}\}$, it holds
    \begin{equation}
        \label{MCPdifJac}
        N_p(\theta, w) := 1 + \frac{w \partial_w \J_R(\theta,w)}{\J_R(\theta,w)} \leq N.
    \end{equation}
\end{proposition}
    
\begin{proof}
    We are going to argue from \cref{MCPJacobian} to \cref{MCPdifJac}, and carefully justify that each steps are actual logical equivalences. We wish to reduce \cref{MCPJacobian} to a differential inequality, which is basically based on Grönwall's lemma. However, as $\J_R$ is not always differentiable at points $(\theta,w)$ such that $w + \theta \in \tfrac{\pi_q}{2} \mathds{Z}$ (more specifically this happens when $p > 2$), we need to proceed cautiously. 
    
    Let $(\theta,w)$ be in one of the five disjoint open connected components of
    \[
     \left(0,\frac{\pi_q}{2}\right)\times\ointerval{0}{2 \pi_q} \setminus \left\{w + \theta \in \frac{\pi_q}{2} \mathds{Z}\right\}=:\bigcup_{k=1}^5 C_k,
    \]
    that we denote by $C_k$, $k = 1, \dots, 5$ (see \cref{fig:my_label}). The map $z \mapsto \J_R(\theta,z)$ is smooth for $z \in \interval{w t}{w}$ if $(\theta,wt)$ lies in the same connected component same as $(\theta,w)$, that is if
    \[
    wt  \in I_{\theta,w} := \left( \max\left(0, k \frac{\pi_q}{2} - \theta\right),w\right).
    \] 
    By using the fundamental theorem of calculus and the monotonicity of the logarithm,
    we can write the following.
    \begin{equation}
        \label{MCPgronwall}
        \begin{aligned}
        &\J_R(\theta,wt)\geq t^{N-1}\J_R(\theta,w) && \text{ for all }(\theta,w)\in \bigcup_{k=1}^5 C_k\text{ and }t\in[0,1]\\
        \Rightarrow \ &\J_R(\theta,wt)\geq t^{N-1}\J_R(\theta,w) && \text{ for all }k,~(\theta,w)\in C_k\text{ and }wt\in I_{\theta,w}\\
        \Leftrightarrow \ &\log \J_R(\theta,w)-\log \J_R(\theta,wt)\leq -(N-1)\log t && \text{ for all }k,~(\theta,w)\in C_k\text{ and }wt\in I_{\theta,w}\\
        \Leftrightarrow \ &
        \int^w_{w t} \frac{\diff }{\diff z} \log \J_R( \theta,z) \diff z \leq (N - 1) \int^w_{w t} \frac{\diff}{\diff z} \log z \diff z && \text{ for all }k,~(\theta,w)\in C_k\text{ and }wt\in I_{\theta,w}\\
        \Leftrightarrow \ &
    \frac{\partial_z \J_R(\theta,z)}{\J_R(\theta,z)} \leq (N - 1) \frac{1}{z}, &&\text{ for all } k,~(\theta,w)\in C_k\text{ and } z \in I_{\theta,w}\\
    \Leftrightarrow \ &
    \frac{\partial_z \J_R(\theta,z)}{\J_R(\theta,z)} \leq (N - 1) \frac{1}{z} &&\text{ for all }k,~(\theta,z)\in C_k.
        \end{aligned}
    \end{equation}
    
    Since the choice of $k$ is arbitrary, we obtain
    \begin{equation}\label{ineq:eachCk}
    \frac{\partial_z \J_R(\theta,z)}{\J_R(\theta,z)} \leq (N - 1) \frac{1}{z} ~~\text{for all}~~(\theta,z)\in \bigcup _{k=1}^5C_k,
    \end{equation}
    and one implication \cref{MCPJacobian} $\Rightarrow$ \cref{MCPdifJac} holds.

    For the converse implication, we need to make sure that if \cref{ineq:eachCk} holds for all $(\theta,w) \in \bigcup_{k=1}^5C_k$, then \cref{MCPJacobian} also holds for all $(\theta,w)\in \bigcup_{k}C_k$ and $t\in[0,1]$.
    Note that the implication in the second line of \cref{MCPgronwall} is in fact an equivalence if $t\in[0,1]$ is such that $(\theta,w)$ and $(\theta, w t)$ are both in the same connected component $C_k$.
    Therefore, we only need to consider the case where $(\theta,w)$ and $(\theta,wt)$ are in different connected components.
    Let us assume without loss of generality that $t\in[0,1]$ and $(\theta, w) \in \bigcup_{k=1}^5C_k$ is such that $(\theta, w t) \in C_2$ and $(\theta, w) \in C_3$. Then, since $\J_R$ is continuous everywhere on $\R^2$ (see \cref{remark:JRcontinuous}), we have that
    \begin{align}
    \label{eq:continequ}
    \begin{split}
        \frac{\J_R(\theta,wt)}{\J_R( \theta,w)} ={}& \lim_{\substack{w_1 \to (\pi_q - \theta)^- \\ w_2 \to (\pi_q - \theta)^+}} \frac{\J_R( \theta, wt)}{\J_R(\theta, w_1)} \frac{\J_R(\theta, w_2)}{\J_R( \theta,w)} \\
        ={}& \lim_{\substack{w_1 \to (\pi_q - \theta)^- \\ w_2 \to (\pi_q - \theta)^+}} \frac{\J_R( \theta, w_1 (w t/w_1))}{\J_R( \theta, w_1)} \frac{\J_R(\theta,w(w_2/w))}{\J_R(\theta,w)} \\
        \geq{}& \lim_{\substack{w_1 \to (\pi_q - \theta)^- \\ w_2 \to (\pi_q - \theta)^+}} \left(\frac{w t}{w_1}\right)^{N-1} \left(\frac{w_2}{w}\right)^{N-1}  = t^{N-1}.
    \end{split}
    \end{align}
    The inequality in the last line follows from the fact that $( \theta, w_1 (w t/w_1))$ and $(\theta, w_1)$ (resp. $(\theta,w(w_2/w) )$ and $( \theta,w)$) are in the same connected component $C_2$ (resp. $C_3$).
\begin{figure}
    \centering
    \includegraphics[width=12cm]{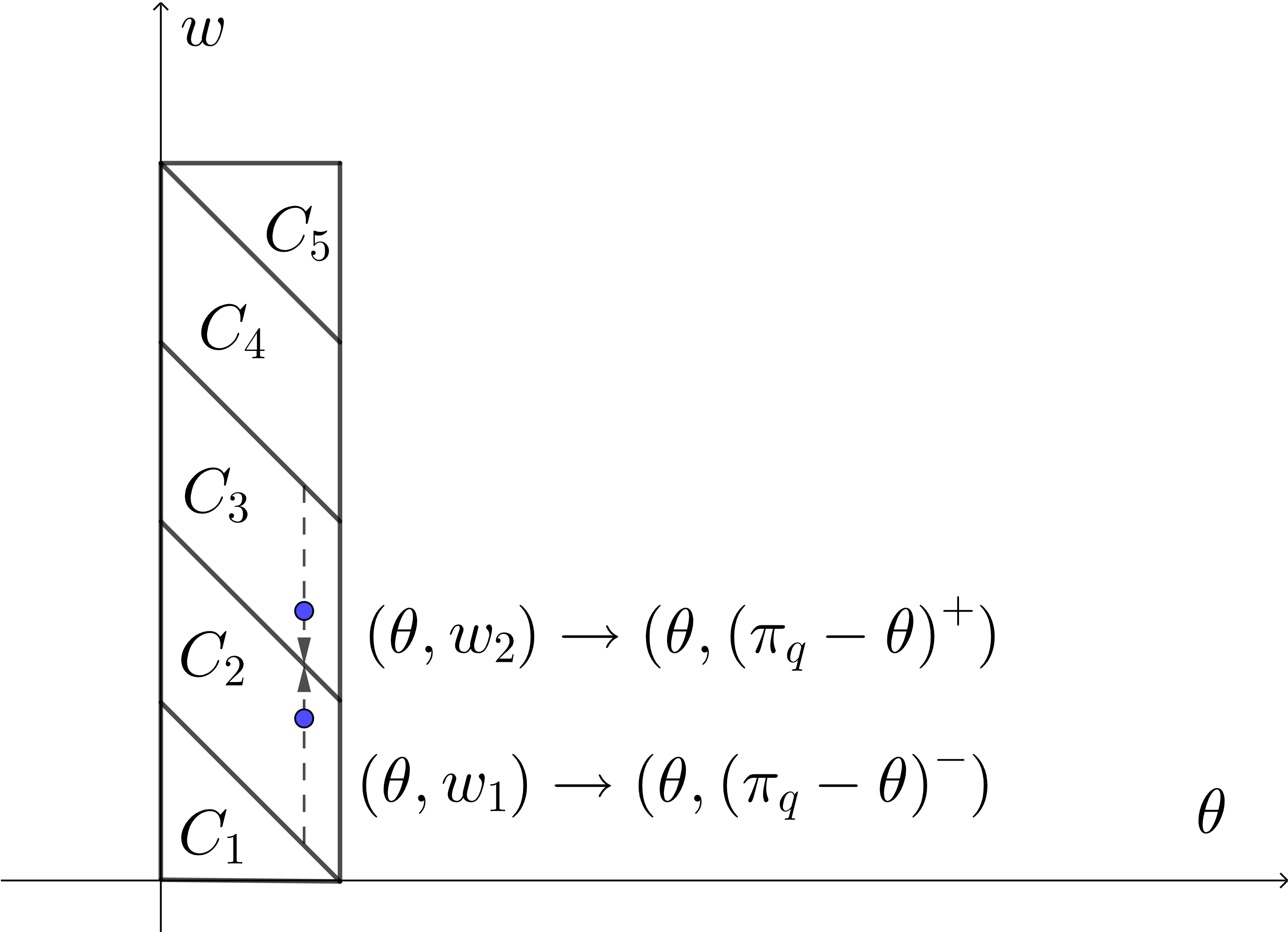}
    \caption{Depiction of the connected components $C_k$ for $k = 1, \dots, 5$ appearing the proof of \cref{proposition:mcp}. As $(\theta, w_1)$ (resp. $(\theta, w_2)$)  tends to the boundary between $C_2$ and $C_3$ from $C_2$ (resp. $C_3$), the continuity of $\J_R$ on  $\mathds{R}^2$ ensures the validity of \cref{eq:continequ}.
    }
    \label{fig:my_label}
\end{figure} 
\end{proof}

\begin{remark}
Based on the definitions introduced in this section, one can readily observe that
\[
N_p(\theta, w) = 1+\frac{w\partial_w\J_R(\theta, w)}{\J_R(\theta, w)} = \frac{\partial_t \J^t(r, \theta, w)}{\J^t(r, \theta, w)}\Big|_{t = 1}.
\]
Rewritten in this form, the relationship between the curvature exponent and the regularity of the geodesics becomes clearer. The curvature exponent may attain an arbitrarily large value if there exists a geodesic along which the differential $\partial_t\J^t$ diverges or the Jacobian $\J^t$ becomes zero. We have observed that, due to the nature of the regularity of the geodesics, there exists a covector for which $\partial_t\J^t$ reaches infinity when $p>2$, and $\J^t$ becomes zero when $p<2$. To determine whether the measure contraction property holds, we will estimate the supremum of $N_p(\theta, w)$ around these covectors.
\end{remark}

The characterisation in \cref{proposition:mcp} will now be used to prove the $\mathsf{MCP}$ for the $\ell^p$-Heisenberg group when $1 < p \leq 2$, and disprove it when $p > 2$.

\begin{theorem}\label{thm:MCPp>2}
     For $p > 2$,
     the $\ell^p$-Heisenberg group $(\mathds{H}, \diff, \mathcal{L}^3)$ does not satisfy the $\mathsf{MCP}(K, N)$ for any $K \in \mathds{R}$ and any $N \geq 1$.
\end{theorem}

\begin{proof}
 Using the same argument as in the beginning of the proof of \cref{thm:MCPp<2}, it is enough to show that $\mathsf{MCP}(0, N)$ does not hold for any $N \geq 1$.
The failure of $\mathsf{MCP}(0,N)$ for a finite $N$ is a consequence of 
\[
    \limsup_{(\theta, w) \to (\theta_0, w_0)} 1 + \frac{\partial_w \J_R(\theta, w)}{\J_R(\theta, w)} = +\infty,
    \]
and \cref{proposition:mcp}, which follows from
\cref{positivityJac}, \cref{regularitydF}.
\end{proof}

Geometrically speaking, this means that when $p > 2$, the $t$-geodesic homothety shrinks faster than any polynomials around the $C^1$-corner of the sub-Finsler geodesics.

The case of $p < 2$ is addressed in the following result.
\begin{theorem}
    \label{thm:MCPp<2}
    When $p \in \linterval{1}{2}$, the $\ell^p$-Heisenberg group $(\mathds{H}, \diff, \mathcal{L}^3)$ has the $\mathsf{MCP}(K, N)$ if and only if $K \leq 0$ and $N \geq N_p$, where $N_p < +\infty$ is the curvature exponent for which the lower bound $N_p > 2 q + 1$ holds if $p \neq 2$, and $N_2 = 5$.
\end{theorem}

\begin{proof}
    The $\ell^p$-Heisenberg group can't satisfy the $\mathsf{MCP}(K, N)$ for $K > 0$ and $N > 1$, since spaces that do are bounded by Bonnet-Myers theorem (see \cite[Theorem 4.3]{ohta2007}). The measure contraction property $\mathsf{MCP}(0, N)$ implies the $\mathsf{MCP}(K, N)$ for all $K < 0$ by \cite[Lemma 2.4]{ohta2007}. Conversely, if it satisfies the $\mathsf{MCP}(K, N)$ with $K < 0$, then it will also satisfy the $\mathsf{MCP}(0, N)$. Indeed, should the metric measure space $(\mathds{H}, \diff, \mathcal{L}^3)$ satisfy the $\mathsf{MCP}(K, N)$ with $K < 0$, the scaled space $(\mathds{H}, \epsilon^{-1} \diff, \epsilon^{-4} \mathcal{L}^3)$ would satisfy the $\mathsf{MCP}(\epsilon^2 K, N)$ for all $\epsilon > 0$ by \cite[Lemma 2.4]{ohta2007}. Since $(\mathds{H}, \diff, \mathcal{L}^3)$ and $(\mathds{H}, \epsilon^{-1} \diff, \epsilon^{-4} \mathcal{L}^3)$ are isomorphic as metric measure spaces through the dilations of the Heisenberg group $\delta_\epsilon(x, y, z) := (\epsilon x, \epsilon y, \epsilon^2 z)$, it follows that $(\mathds{H}, \diff, \mathcal{L}^3)$ satisfies the $\mathsf{MCP}(\epsilon^2 K, N)$ for all $\epsilon > 0$, and thus also the $\mathsf{MCP}(0, N)$ by taking the limit $\epsilon \to 0^+$ in \cref{def:MCP}.

    It therefore remains to prove that the curvature exponent of $(\mathds{H}, \diff, \mathcal{L}^3)$ is finite and bounded from below by $2 q + 1$. From \cref{proposition:mcp} (and the symmetry of the Jacobian), it follows that the curvature exponent is given by
    \begin{equation}
        \label{eq:Np}
        N_p := \sup\left\{ 1 + \frac{w \partial_w \J_R(\theta, w)}{\J_R(\theta, w)} \ \Big| \ (\theta, w) \in \left(-\frac{\pi_q}{4},\frac{\pi_q}{4}\right) \times \ointerval{0}{2 \pi_q} \setminus \left\{w + \theta \in \frac{\pi_q}{2} \mathds{Z}\right\}  \right\}.
    \end{equation}
    When $p \leq 2$, the function $(\theta, w) \mapsto N_p(\theta, w)$ can be extended continuously to $\left(-\frac{\pi_q}{4},\frac{\pi_q}{4}\right) \times \ointerval{0}{2 \pi_q}$ by \cref{regularitydF}. We therefore only need to show that
    \begin{equation*}
        \limsup_{( \theta,w) \to ( \theta_0,0)} 1 + \frac{w \partial_w \J_R(\theta,w)}{\J_R(\theta,w)} < +\infty, \text{ for all } \theta_0 \in \left(-\frac{\pi_q}{4},\frac{\pi_q}{4}\right).
    \end{equation*}

    If $\theta_0\neq 0$,
    then \cref{lem:case1} 
    implies that 
    \[
    \lim_{(\theta, w) \to (\theta_0, 0)}1 + \frac{w \partial_w \J_R(\theta,w)}{\J_R(\theta,w)} = 5.
    \]

    Next we check that the limit superior of $N_p(\theta, w)$ as $( \theta,w) \to (0, 0)$ is finite and bounded from below by $2q+1$. We can assume, without loss of generality, that the ratio $w/\theta$ converges to $s \in \mathds{R} \cup \pm \infty$. Then, by using the principal term of \cref{asymptoticFat00} and \cref{asymptoticdwFat00}, and for $s \neq 0, \pm \infty$ (the case $s = 0, +\infty$ or $-\infty$ will be recovered as a limit), we have that
    \[
    N_p(\theta, w) \to 1 + \frac{s \partial_s P(s)}{P(s)}, \text{ as } (\theta, w) \to (0, 0),
    \]
    where $P(s)$ is the function defined by
    \[
    P(s) := \abs{1+s}^{2q} - q^2 \abs{1 + s}^{q}(1+s) + 2 (q^2 - 1) \abs{1+  s}^q - q^2 \abs{1 + s}^{q - 2}(1 + s) + 1,
    \]
    and thus also
    \begin{align*}
        N_p ={}& \sup\Big\{ N_p(\theta, w) \ \Big| \ (\theta, w) \in \left(-\frac{\pi_q}{4},\frac{\pi_q}{4}\right) \times \ointerval{0}{2 \pi_q} \setminus \left\{w + \theta \in \frac{\pi_q}{2} \mathds{Z}\right\}  \Big\}\\
        &\geq \limsup_{(\theta, w) \to (0, 0)} N_p(\theta, w) = \sup \left\{ 1 + \frac{s \partial_s P(s)}{P(s)} \ \mid s \in \R \setminus \{0\}\right\}.
    \end{align*}
    The graph of the map $s \mapsto 1 + \frac{s \partial_s P(s)}{P(s)}$ is illustrated in \cref{fig:graphsdPP}.

    \begin{figure}
    \centering
    {\scalebox{0.6}{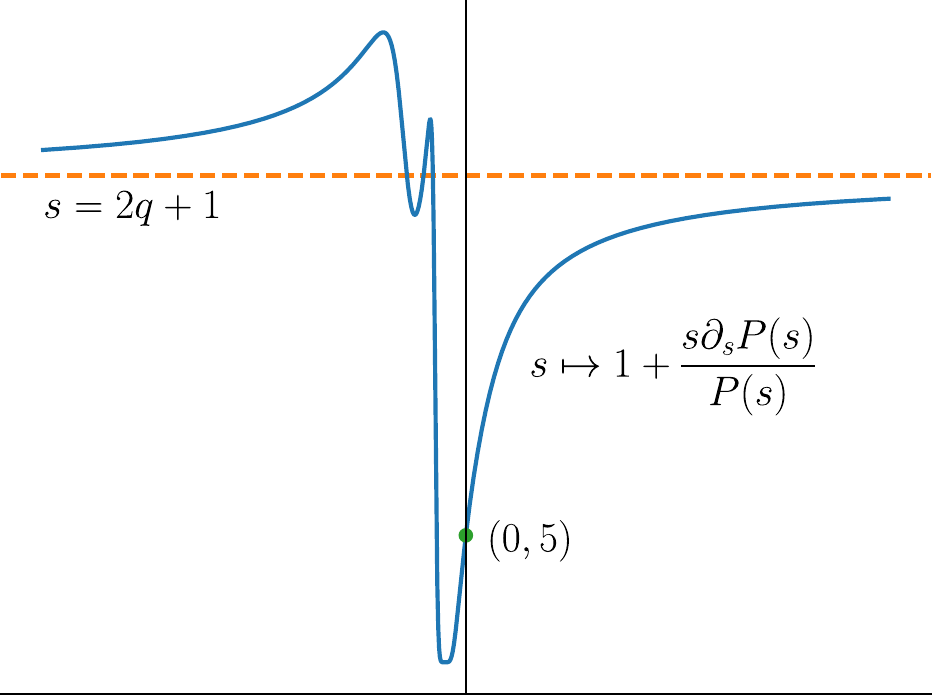}}
    \caption{The graph of $s \mapsto 1 + \frac{s \partial_s P(s)}{P(s)}$.}
    \label{fig:graphsdPP}
    \end{figure}

    We claim that $P(s)$ vanishes if and only if $s = 0$. Indeed, if $1 + s < 0$, then it is easy to see that all the terms of $P(s)$ are positive. The differential of $P(s)$ is
    \begin{align*}
    \partial_s P(s) :={}& 2 q \abs{1 + s}^{q - 2} \Big[\abs{1 + s}^q (1 + s) - 1 - (q + 1) s -\frac{q (q + 1)}{2} s^2 \Big] \\
    ={}& 2 q (1 + s)^{q - 2} \Big[(1 + s)^{q + 1} - 1 - (q + 1) s -\frac{q (q + 1)}{2} s^2 \Big], && \text{if } 1 + s > 0.
    \end{align*}
    Notice that $1 + (q + 1) s + \frac{q(q + 1)}{2} s^2$ is the first three term of the binomial expansion of $(1+s)^{q+1}$. Therefore we can see that the content in the bracket of $\partial_sP(s)$ is concave on $\ointerval{-1}{\infty}$, and $P(s)$ has a unique zero point at $s=0$. The following limits are easily obtained:
    \[
    \lim_{s\to 0}1+\frac{s\partial_sP(s)}{P(s)}=5, \ \lim_{s\to+\infty}1+\frac{s\partial_sP(s)}{P(s)}=\lim_{s\to-\infty}1+\frac{s\partial_sP(s)}{P(s)}=2q+1.
    \]
    The map $s \mapsto 1 + \frac{s \partial_s P(s)}{P(s)}$ is thus continuous for all $s \in \mathds{R}$ and bounded at infinity. 
    
    To establish that $N_p > 2q + 1$ when $p < 2$, it is sufficient to observe, following a lengthy computation, that
    \[
    \partial_s \left(\frac{s \partial_s P(s)}{P(s)}\right) = \frac{2 q \abs{1 + s}^{4 q + 2} + \cdots \text{ finitely many lower order terms} \cdots}{(\cdots)^2}.
    \]
    This expression is therefore strictly positive for $s$ small enough, which means that $1 + \frac{s \partial_s P(s)}{P(s)}$ converges to $2 q + 1$ as $s$ approaches $- \infty$ by values strictly greater than $2 q + 1$.
\end{proof}

\begin{remark}
    If $p=2$,
    then $P(s)=s^4$ and the ratio $1+\frac{s\partial_sP(s)}{P(s)}\equiv 5$. This shows that $N_2 \geq 5$. The equality is proven in \cite{juillet2009}.
\end{remark}

    From \cref{fig:graphsdPP}, one may be tempted to conjecture that the maximum of the function $s \mapsto 1 + s \partial_s P(s)/P(s)$ is the curvature exponent. 
    Graphically it appears otherwise.
    Assume that the curvature exponent $N_p^*$ is the the maximum of the function $1+\frac{\partial_s P(s)}{P(s)}$.
    Then by \cref{proposition:mcp},
    the function $(N_p^*-1)\J_R(\theta,w)-\partial_w\J_R(\theta,w)$ should be non-negative everywhere. However, for every $p < 2$, we are always able to find numerically a value of $\theta$ such that the graph of $w \mapsto (N_p^*-1)\J_R(\theta,w)-\partial_w\J_R(\theta,w)$ is negative, see \cref{fig:logderneg}. 
    This implies that the maximum of the function $1+\frac{\partial_sP(s)}{P(s)}$ is not the curvature exponent,
    and the optimal $N$ in the $\mathsf{MCP}(0, N)$ is to be found in the interior of the domain of the exponential map.

\begin{figure}
    \centering
	\captionsetup[subfigure]{justification=centering}
	\subcaptionbox{3D representation of the graph of the function $(w, \theta) \mapsto (N_p^*-1)\J_R(\theta,w)-\partial_w\J_R(\theta,w)$.}{\scalebox{0.465}{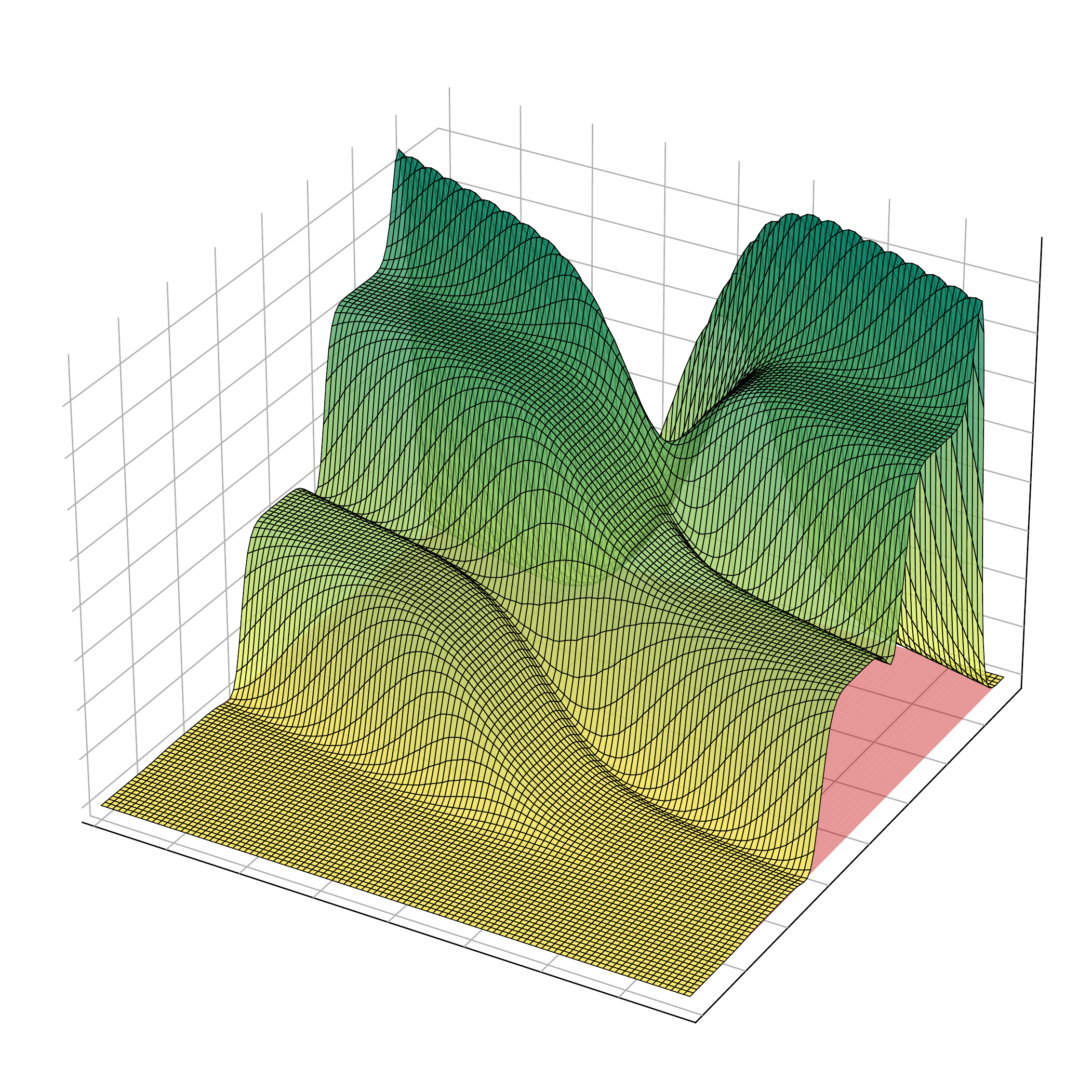}}\hspace{1em}%
	\subcaptionbox{Depiction of a section of the graph above, that is to say, a 2D representation of $w \mapsto (N_p^*-1)\J_R(\theta,w)-\partial_w\J_R(\theta,w)$ for a fixed $\theta$.}{\scalebox{0.465}{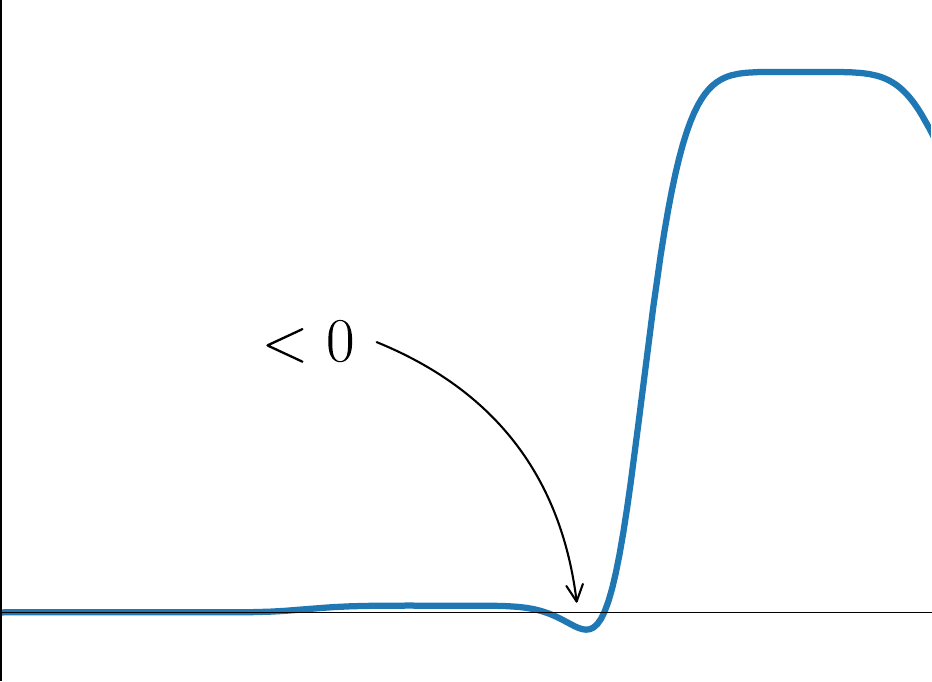}}\hspace{1em}%
	\caption{The maximum $N_p^*$ of the function depicted in \cref{fig:graphsdPP} does not seem to be the curvature exponent.}
    \label{fig:logderneg}
\end{figure}

\section{Geodesic dimension of the \texorpdfstring{$\ell^p$}{lp}-Heisenberg group}
\label{sec:geoddim}

Since the $\ell^p$-Heisenberg group is left-invariant, the geodesic dimension as introduced in \cref{def:geod_dim} is given by 
\begin{equation}
    \label{Ngeoidentity}
    N_{\mathrm{geo}}=\inf\{s>0 \mid C_s=+\infty\}=\sup\{s>0\mid C_s=0\},
\end{equation}
where
\begin{equation}
    \label{Csidentity}
    C_s := \sup\left\{\limsup_{t\to 0^+}\frac{1}{t^s}\frac{\mathfrak{m}(\Omega_t)}{\mathfrak{m}(\Omega)} \Big| \ \Omega \text{ Borel, bounded, } \mathfrak{m}(\Omega) \in \ointerval{0}{\infty}\right\},
\end{equation}
and $\Omega_t$ is the $t$-geodesic homothety of $\Omega$ from the identity.

We first show that the geodesic dimension can be can be determined by solely considering the case where $\Omega$ is the image of a cube.

\begin{proposition}
    \label{Cscube}
    In Expression \cref{Ngeoidentity} characterising $N_{\mathrm{geo}}$, the supremum defining $C_s$ in \cref{Csidentity} can be taken over the sets of the form $\Omega = \exp_e(A)$ with
    \begin{equation}
        \label{cube}
        A = \ointerval{0}{R} \times \interval{-\Theta}{\Theta} \times \interval{-W}{W},
    \end{equation}
for some $R > 0$, $\Theta \in \interval{0}{2 \pi_q}$ and $W \in \interval{0}{2\pi_q}$.
\end{proposition}

\begin{proof}
    Denote by $C_s^{\mathrm{cube}}$ the same supremum as in \cref{Csidentity}, except that it is taken over $\Omega = \exp_e(A)$ where $A$ is a cube of the form \cref{cube}. Clearly, $C_s^{\mathrm{cube}} = +\infty$ implies that $C_s = +\infty$ (resp. $C_s = 0$ implies that $C_s^{\mathrm{cube}} = 0$).
    
    A bounded subset $\Omega$ is, up to removing the (negligible) $z$-axis, contained in the image of such a cube $A$, by \cref{cotinjdom}. Thus, we have
    \[
    \limsup_{t\to 0^+}\frac{1}{t^s}\mathfrak{m}(\exp_e(A)_t) \geq \limsup_{t\to 0^+}\frac{1}{t^s}\mathfrak{m}(\Omega_t).
    \]
    In particular, $C_s = +\infty$ implies that $C_s^{\mathrm{cube}} = +\infty$ (resp. $C_s^{\mathrm{cube}} = 0$ implies that $C_s = 0$), and
    \[
    N_{\mathrm{geo}}=\inf\{s>0 \mid C_s^{\mathrm{cube}}=+\infty\}=\sup\{s>0\mid C_s^{\mathrm{cube}}=0\}.
    \]
\end{proof}

\begin{remark}
    We can in fact assume that $\Theta < \frac{\pi_q}{2}$ by the symmetries of the Jacobian stated in \cref{jacsymmetry}.
\end{remark} 

Let $t \in \ointerval{0}{1}$, $R > 0$, $\Theta \in \interval{0}{\frac{\pi_q}{2}}$, $W \in \interval{0}{2\pi_q}$, $A$ as in \cref{cube}, and $\Omega = \exp_e(A)$. Since the Jacobian of the exponential map is non-negative by \cref{positivityJac},
we can apply Tonelli's theorem to obtain
\begin{align}
    \label{measureOmegat}
    \begin{split}
    \m(\Omega_t)&=\int_A\J^t\diff r\diff \theta\diff w\\
    &=\int_{-W}^W\int_{-\Theta}^{\Theta}\int_0^R\frac{tr^3}{w^4}\J_R(\theta,wt)\diff r\diff \theta\diff w\\
    &=\frac{t^4 R^4}{4}\int_{-W t}^{W t}\int_{-\Theta}^{\Theta}\frac{1}{w^4}\J_R(\theta,w)\diff \theta\diff w.
    \end{split}
\end{align}

We therefore observe that the asymptotic behaviour of $\mathfrak{m}(\Omega_t)$ as $t \to 0^+$ will correspond to the integral asymptotic of $\frac{1}{w^4}\J_R(\theta,w)$ over the domain $A_t := \interval{-\Theta}{\Theta} \times \interval{-W t}{W t}$. Let $\delta > 0$ be a number chosen smaller than the radius of convergence of the $q$-trigonometric functions at $0$, and $\epsilon_q \in \ointerval{1}{q + 1}$ given in the proof of \cref{prop:jacobianons0}. The symmetries of the Jacobian \cref{jacsymmetry} allow us once more to narrow down our study to the asymptotic behavior of the following four cases, which are illustrated in \cref{fig:domaingeodsim}:
\begin{enumerate}[label=\normalfont\arabic*)]
\setcounter{enumi}{-1}
    \item $A_{t}^0 = \interval{\delta}{\Theta} \times \interval{-W t}{W t}$;
    \item $A_{t}^1 = \{ (\theta, w) \mid -W t \leq w \leq W t, \epsilon_q \abs{w} \leq \theta \leq \delta \}$;
    \item $A_{t}^2 = \{ (\theta, w) \mid -W t \leq w \leq W t, \abs{w} \leq \theta \leq \epsilon_q \abs{w}\}$;
    \item $A_{t}^3 = \{ (\theta, w) \mid 0 \leq w \leq W t, \abs{w} \leq \theta \leq \epsilon_q \abs{w}\}$.
\end{enumerate}

\begin{figure}
    \centering
    {\scalebox{0.6}{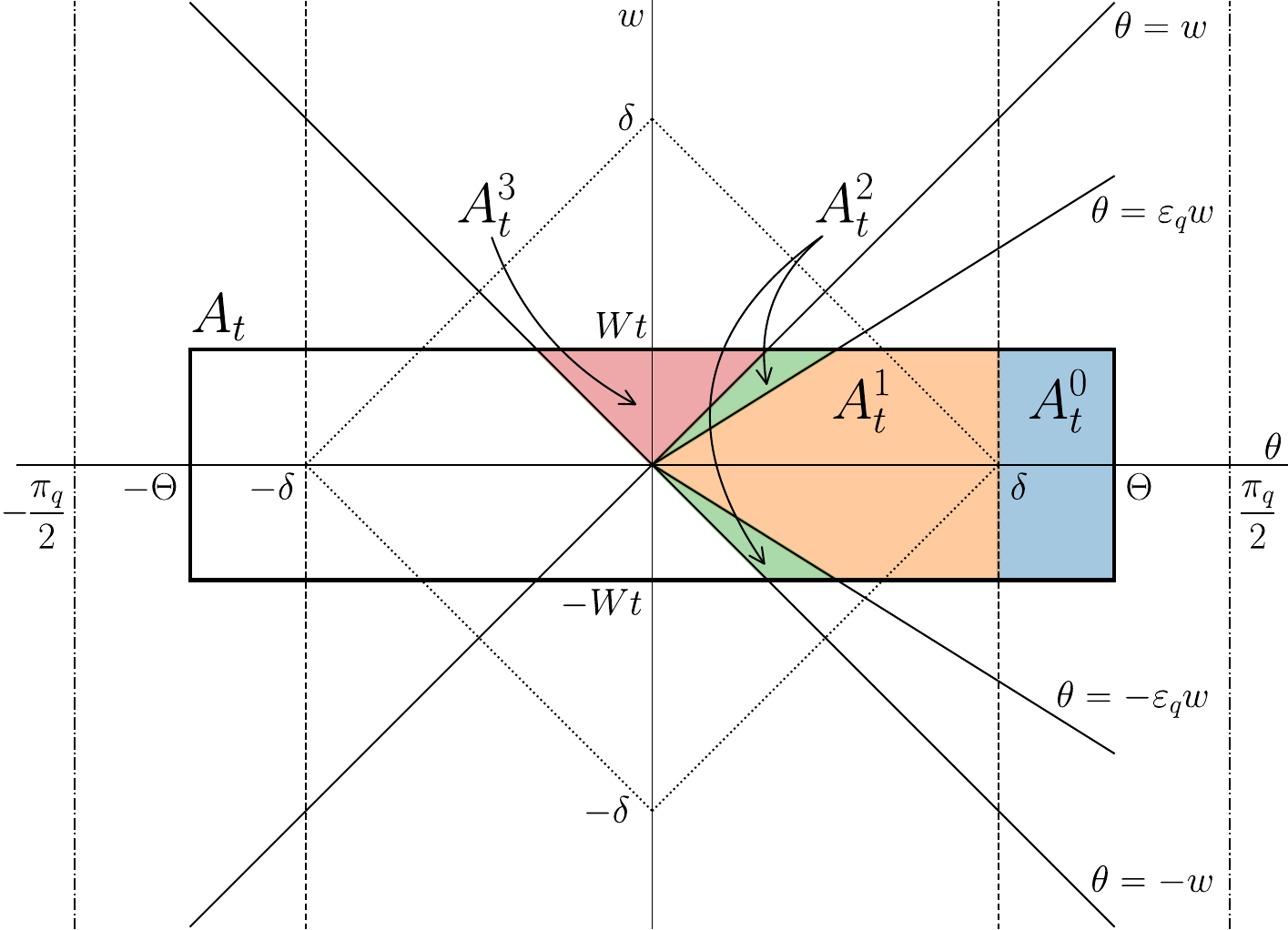}}
    \caption{The subdivisions of the domain $A_t$. The dotted square represents the domain for which the series representation of \cref{asymptoticFat00} holds. }
    \label{fig:domaingeodsim}
\end{figure}

In the following four lemmas, we are going to obtain the asymptotics, as $t \to 0^+$, corresponding to the sets $A_{t}^0$, $A_{t}^1$, $A_{t}^2$, and $A_{t}^3$ respectively. For simplicity, we will often use the notation $d_k := 1 + (-1)^k$.
Recall that we write $f(t) \sim g(t)$ (as $t \to 0^+$) if there exists $C \neq 0$ such that $f(t) = g(t)(C+o(1))$ (as $t \to 0^+$). 
\begin{lemma}[Integral asymptotic over the domain $A_t^0$]
\label{geoddimcase0}
    \[
    \int_{-W t}^{W t} \int_\delta^{\Theta} \frac{1}{w^4} \J_R(\theta, w) \diff \theta \diff w \underset{\ \ t \to 0^+}{\sim} t.
    \]
\end{lemma}

\begin{proof}
Since the integrand is bounded on the domain $A_t^0$, the bounded convergence theorem and \cref{lem:case1} implies that
\begin{align*}
    \lim_{t \to 0} \frac{1}{t} \int_{-W t}^{W t} \int_\delta^{\Theta} \frac{1}{w^4} \J_R(\theta, w) \diff \theta \diff w ={}& \lim_{t \to 0} \int_{-W}^{W} \int_\delta^{\Theta} \frac{1}{(w t)^4} \J_R(\theta, w t) \diff \theta \diff w \\
    ={}& \int_{-W}^{W} \int_\delta^{\Theta} \lim_{t \to 0} \frac{1}{(w t)^4} \J_R(\theta, w t) \diff \theta \diff w \\
    ={}& \int_{-W}^W\int_{\delta}^{\Theta}\abs{\cos_q\theta\sin_q\theta}^{2q-4}\diff \theta \diff w > 0.
\end{align*}
\end{proof}

\begin{lemma}[Integral asymptotic over the domain $A_t^1$]
\label{geoddimcase1}
    \[
    \int_{-W t}^{W t} \int_{\epsilon_q \abs{w}}^{\delta} \frac{1}{w^4} \J_R(\theta, w) \diff \theta \diff w \underset{\ \ t \to 0^+}{\sim}
    \begin{cases*}
    t & if  $p < 3$  \\
  -  t \log(t) & if $p = 3$ \\
    t^{2 q - 2} & if $p > 3$.
    \end{cases*}
    \]
\end{lemma}

\begin{proof}

We assume that $p \neq 3$ since the case $p = 3$ is established in the same way. 

Inside the domain $A_1^t$, we can use \cref{lem:case1} and write $\partial^k_w \J_R(\theta, 0)$ as a uniformly convergent series by expanding the $q$-trigonometric functions $\sin_q$ and $\cos_q$. The coefficient $\epsilon_q \in \ointerval{1}{q + 1}$ is chosen precisely so that  \cref{radiusconvsinp} can be used, as was done with the series \cref{series:case1} in the proof of \cref{prop:jacobianons0} (Case 1). We obtain
\begin{align*}
    \int_{-W t}^{W t} \int_{\epsilon_q \abs{w}}^{\delta} \frac{1}{w^4}& \J_R(\theta, w) \diff \theta \diff w= \int_{-W t}^{W t} \int_{\epsilon_q \abs{w}}^{\delta} \sum_{k = 4}^{\infty} \sum_{l = 2}^\infty  \frac{c_{k, l}}{k!} \theta^{l q - k} w^{k - 4} \diff \theta \diff w \\
    &= \sum_{k = 4}^{\infty} \sum_{l = 2}^\infty \frac{c_{k, l}}{k!(l q - k + 1)} \int_{-W t}^{W t} \left( \delta^{l q - k + 1} - (\epsilon_q \abs{w})^{l q - k + 1} \right) w^{k - 4} \diff w \\
    &= \sum_{k = 4}^{\infty} \sum_{l = 2}^\infty \frac{c_{k, l}d_k}{k!(l q - k + 1)} \left( \frac{\delta^{l q - k + 1} W^{k - 3}}{k - 3} t^{k - 3} - \frac{\epsilon_q^{l q - k + 1} W^{l q - 2}}{l q - 2} t^{l q - 2} \right).
\end{align*}

The coefficient in front of the term in $t$ is obtained, in the integration above, by fixing $k = 4$ and summing over all $l$:
\begin{align*}
    \sum_{l = 2}^\infty \frac{c_{k, l}d_k}{k!(l q - k + 1)} \frac{\delta^{l q - k + 1} W^{k - 3}}{k - 3} t ={}& \int_{-W t}^{W t} \int_{\epsilon_q \abs{w}}^{\delta} \sum_{l=2}^\infty \frac{c_{4,l}^1}{4!}\theta^{lq-4} \\
    ={}& \int_{-W t}^{W t} \int_{\epsilon_q \abs{w}}^{\delta} \frac{(q-1)^2}{12}|\cos_q\theta\sin_q\theta|^{2q-4} > 0.
\end{align*}
The coefficient in front of the term in $t^{2 q - 2}$, on the other hand, is given by fixing $l = 2$ and summing over all $k$:
\begin{align*}
    - \sum_{k = 4}^{\infty} \frac{c_{k, 2}d_k}{k!(2 q - k + 1)} \frac{\epsilon_q^{2 q - k + 1} W^{2 q - 2}}{l q - 2} t^{2 q - 2} ={}& \int_{-W t}^{W t} \int_{\epsilon_q \abs{w}}^{\delta} \sum_{k = 4}^{\infty}  \frac{c_{k, 2}}{k!} \theta^{2 q - k} w^{k - 4} \diff \theta \diff w \\
    ={}& \int_{-W t}^{W t} \int_{\epsilon_q \abs{w}}^{\delta} \abs{\theta}^{2 q - 4} P_1\left(\frac{w}{\theta}\right) \diff \theta \diff w > 0,
\end{align*}
since we have seen that $P_1$ is positive almost everywhere in the proof of \cref{prop:jacobianons0}.

The leading term will be the one in $t$ if $2 q - 2 > 1$, that is $p < 3$, and the one in $t^{2 q - 2}$ if $2 q - 2 < 1$, that is to say $p > 3$. The case $p = 3$ is treated analogously, where a leading term in $t \log(t)$ appears by integration.
\end{proof}

\begin{lemma}[Integral asymptotic over the domain $A_t^2$]
\label{geoddimcase2}
    \[
    \int_{-W t}^{W t} \int_{\abs{w}}^{\epsilon_q \abs{w}} \frac{1}{w^4} \J_R(\theta, w) \diff \theta \diff w \underset{\ \ t \to 0^+}{\sim}
    \begin{cases*}
    t^{2 q - 2} & if  $p \neq 3$  \\
  -  t \log(t) & if $p = 3$.
    \end{cases*}
    \]
\end{lemma}

\begin{proof}

Inside the domain $A_t^2$, we can use \cref{asymptoticFat00} and write
\begin{equation}
    \label{geoddimexpcase2}
    \frac{1}{w^4} \J_R(\theta, w) = |\theta|^{2q-4}P_2\left(\frac{w}{\theta}\right)+\frac{1}{w^4}E(\theta,w),
\end{equation}
as was done with the series representation \cref{series:case2} in the proof of \cref{prop:jacobianons0} (Case 2). Recall that there it was shown that $P_2$ is an analytic function that is positive almost everywhere, and that $E(\theta, w)$ is given in \cref{asymptoticFat00remainder}.

Integrating the principal term gives
\begin{align*}
    0 < \int_{-W t}^{W t} \int_{\abs{w}}^{\epsilon_q \abs{w}} |\theta|^{2q-4} & P_2\left(\frac{w}{\theta}\right) \diff \theta \diff w = \int_{-W t}^{W t} \int_{\abs{w}}^{\epsilon_q \abs{w}} \sum_{k = 4}^{\infty} c_k \theta^{2 q - k} w^{k - 4} \diff \theta \diff w \\
    ={}& \sum_{k = 4}^{\infty} \int_{-W t}^{W t} \frac{c_k}{2 q - k + 1}(\epsilon_q^{2 q - k + 1} - 1) \abs{w}^{2 q - k + 1} w^{k - 4}  \diff w \\
    ={}& \left(\sum_{k = 4}^{\infty} \frac{c_k d_k}{(2 q - k + 1)(2 q - 2)}(\epsilon_q^{2 q - k + 1} - 1) W^{2 q - 2}\right) t^{2 q - 2} \underset{\ \ t \to 0^+}{\sim} t^{2 q - 2}.
\end{align*}

It remains to shows that the second term of \cref{geoddimexpcase2}, that is to say the error term, has order greater than $t^{2 q - 2}$. Inside the domain $A_t^2$, it holds $\abs{w}\leq \abs{\theta} \leq \epsilon_q \abs{w}$, and which helps, together with the binomial series theorem, in obtaining an upper bound on the term $E(\theta,w)$ as
\[
\abs{E(\theta,w)}\leq\sum_{l=3}^{\infty}\tilde{c}_l|w|^{lq},
\]
where $\tilde{c}_l$ are positive constants and the series is uniformly convergent.

Therefore, integrating the error term yields
\begin{align*}
    \abs{\int_{-W t}^{W t} \int_{\abs{w}}^{\epsilon_q \abs{w}} \frac{1}{w^4}E(\theta,w) \diff \theta \diff w} \leq{}& \int_{-W t}^{W t} \int_{\abs{w}}^{\epsilon_q \abs{w}} \sum_{l=3}^{\infty}\tilde{c}_l|w|^{lq - 4} \diff \theta \diff w \\
    ={}& \sum_{l=3}^{\infty} \int_{-W t}^{W t} (\epsilon_q - 1) \tilde{c}_l|w|^{lq - 3} \diff w \\
    ={}& \sum_{l=3}^{\infty} \frac{2 (\epsilon_q - 1) \tilde{c}_l W^{l q - 2}}{l q - 2}t^{lq - 2} \underset{\ \ t \to 0^+}{\sim} t^{3 q - 2}.
\end{align*}
\end{proof}

\begin{lemma}[Integral asymptotic over the domain $A_t^3$]
\label{geoddimcase3}
    \[
    \int_{-W t}^{W t} \int_{\abs{\theta}}^{W t} \frac{1}{w^4} \J_R(\theta, w) \diff w \diff \theta \underset{\ \ t \to 0^+}{\sim} \begin{cases*}
    t^{2 q - 2} & if $p \neq 3$, \\
    - t \log(t) & if $p = 3$. \\
    \end{cases*}
    \]
\end{lemma}

\begin{proof}

We assume again $p \neq 3$. Inside the domain $A_t^3$, we can use \cref{asymptoticFat00} and write
    \begin{equation}
    \label{geoddimexpcase3}
    \frac{1}{w^4} \J_R(\theta, w) = |w|^{2q - 3}P_3\left(\frac{\theta}{w}\right)+ \frac{1}{w^4}E(\theta,w),
\end{equation}
as was done with the series representation \cref{series:case3} in the proof of \cref{prop:jacobianons0} (Case 3). Recall again that there it was shown that $P_3$ is an analytic function in the fractional sense that is positive almost everywhere, and that $E(\theta, w)$ is given in \cref{asymptoticFat00remainder}.

Integrating the principal term gives
\begin{align*}
    0 < \int_{-W t}^{W t}& \int_{\abs{\theta}}^{W t} |w|^{2q - 4} P_3\left(\frac{\theta}{w}\right) \diff w \diff \theta \\
     ={}& \int_{-W t}^{W t} \int_{\abs{\theta}}^{W t}  \sum_{\alpha  \in \{2q, q, 0\}}\sum_{k=0}^\infty c_{k,\alpha}^3 w^{2q-\alpha-k-4}|\theta|^{\alpha}\theta^k \diff w \diff \theta \\
    & \quad + \int_{-W t}^{W t} \int_{\abs{\theta}}^{W t}  \sum_{\beta \in \{q + 1, q - 1\}}\sum_{k=0}^\infty c_{k,\beta}^3 w^{2q-\beta-k-4}|\theta|^{\beta-1}\theta^{k+1} \diff w \diff \theta.
\end{align*}

This last term simplifies to
\begin{align*}
    \sum_{\alpha  \in \{2q, q, 0\}}&\sum_{k=0}^\infty \frac{c_{k,\alpha}^3}{2q-\alpha-k-3} \int_{-W t}^{W t} |\theta|^{\alpha}\theta^k \left( (W t)^{2q-\alpha-k-3} - \abs{\theta}^{2q-\alpha-k-3}\right) \diff \theta \\
    & + \sum_{\beta \in \{q + 1, q - 1\}}\sum_{k=0}^\infty \frac{c_{k,\beta}^3}{2q-\beta-k-3} \int_{-W t}^{W t} |\theta|^{\beta-1}\theta^{k+1} \left( (W t)^{2q-\beta-k-3} - \abs{\theta}^{2q-\beta-k-3}\right) \diff \theta \\
    ={}& \Big[\sum_{\alpha  \in \{2q, q, 0\}}\sum_{k=0}^\infty \frac{c_{k,\alpha}^3 d_k}{2q-\alpha-k-3} \left(\frac{W^{2q - 2}}{k + \alpha + 1} - \frac{W^{2q - 2}}{2 q - 2}\right) \\
    & \qquad + \sum_{\beta \in \{q + 1, q - 1\}}\sum_{k=0}^\infty \frac{c_{k,\beta}^3d_{k+1}}{2q-\beta-k-3} \left(\frac{W^{2q - 2}}{k + \beta + 1} - \frac{W^{2 q - 2}}{2 q - 2}\right) \Big] t^{2 q -2} \underset{\ \ t \to 0^+}{\sim} t^{2 q - 2}.
\end{align*}

It remains to shows that the second term of \cref{geoddimexpcase2}, that is to say the error term, has order greater than $t^{2 q - 2}$. This is done analogously to the proof of \cref{geoddimcase2}, by estimating $\abs{E(\theta,w)}$ with the binomial series theorem and using the condition $\abs{\theta} \leq \abs{w}$ that holds within $A_t^3$.

If $p = 3$, the computations are similar except that a leading term in $- t \log(t)$ appears in the above integration.
\end{proof}

With these lemmas, the asymptotics of $\mathfrak{m}(\Omega_t)$ can be obtained.

\begin{proposition}
\label{omegatasymptotics}
Let $\Omega = \exp_e(A) \subseteq \mathds{H}$, where 
\begin{equation*}
        A = \ointerval{0}{R} \times \interval{-\Theta}{\Theta} \times \interval{-W}{W},
    \end{equation*}
for some $R > 0$, $\Theta \in \interval{0}{\frac{\pi_2}{2}}$ and $W \in \interval{0}{2\pi_q}$. The asymptotic of $\m(\Omega_t)$, as $t \to 0^+$ is given as follows.
\[\m(\Omega_t)\sim
\begin{cases*}
    t^5 & if  $p < 3$  \\
    -t^5 \log(t) & if $p = 3$ \\
    t^{2 q + 2} & if $p > 3$.
    \end{cases*}
\]
\end{proposition}

\begin{proof}
    From \cref{measureOmegat}, we can write
    \begin{align*}
        \m(\Omega_t)= \frac{t^4 R^4}{4} &\int_{-W t}^{W t}\int_{-\Theta}^{\Theta}\frac{1}{w^4}\J_R(\theta,w)\diff \theta\diff w \\
        = \frac{t^4 R^4}{4}& \Big[ \underbrace{\int_{-W t}^{W t} \int_\delta^{\Theta} \frac{1}{w^4} \J_R(\theta, w) \diff \theta \diff w}_{\text{domain } A_t^0} + \underbrace{\int_{-W t}^{W t} \int_{\epsilon_q \abs{w}}^{\delta} \frac{1}{w^4} \J_R(\theta, w) \diff \theta \diff w}_{\text{domain } A_t^1} \\
        &+ \underbrace{\int_{-W t}^{W t} \int_{\abs{w}}^{\epsilon_q \abs{w}} \frac{1}{w^4} \J_R(\theta, w) \diff \theta \diff w}_{\text{domain } A_t^2} + \underbrace{\int_{-W t}^{W t} \int_{\abs{\theta}}^{W t} \frac{1}{w^4} \J_R(\theta, w) \diff w \diff \theta}_{\text{domain } A_t^3} \\
        &+ \underbrace{\int_{-W t}^{W t} \int_{-\Theta}^{-\delta} \frac{1}{w^4} \J_R(\theta, w) \diff \theta \diff w}_{\text{domain } \tilde A_t^0} + \underbrace{\int_{-W t}^{W t} \int_{-\delta}^{-\epsilon_q \abs{w}} \frac{1}{w^4} \J_R(\theta, w) \diff \theta \diff w}_{\text{domain } \tilde A_t^1} \\
        &+ \underbrace{\int_{-W t}^{W t} \int_{-\epsilon_q \abs{w}}^{-\abs{w}} \frac{1}{w^4} \J_R(\theta, w) \diff \theta \diff w}_{\text{domain } \tilde A_t^2} + \underbrace{\int_{-W t}^{W t} \int_{-W t}^{-\abs{\theta}} \frac{1}{w^4} \J_R(\theta, w) \diff w \diff \theta}_{\text{domain } \tilde A_t^3} \Big].
    \end{align*}
    The asymptotics of the integrals over the domains $A_t^0$, $A_t^1$, $A_t^2$, and $A_t^3$ are given by \cref{geoddimcase0}, \cref{geoddimcase1}, \cref{geoddimcase2}, and \cref{geoddimcase3}. Moreover, they are respectively the same as those over $\tilde A_t^0$, $\tilde A_t^1$, $\tilde A_t^2$, and $\tilde A_t^3$ by the symmetries \cref{jacsymmetry} of the Jacobian.
\end{proof}

The geodesic dimension of the $\ell^p$-Heisenberg group for $p\in(1,\infty)$ can be derived directly from \cref{omegatasymptotics}.

\begin{theorem}\label{thm:geodim}
    The geodesic dimension of the $\ell^p$-Heisenberg group is given by
    \[
    N_{\mathrm{geo}}=
    \begin{cases*}
    5 & if  $p \in \linterval{1}{3}$  \\
    2 q + 2 & if $p \in \rinterval{3}{\infty}$.
    \end{cases*}
    \]
\end{theorem}

\begin{proof}

We only check that the geodesic dimension of the $\ell^3$-Heisenberg group is $5$ since the other cases follow in the same way. By \cref{Cscube}, it is also enough to consider the $t$-intermediate set $\Omega_t$ from the identity of sets $\Omega$ which are the image of a cube \cref{cube} (with $\Theta < \frac{\pi_q}{2}$ by symmetry). Since $\mathfrak{m}(\Omega_t)\sim -t^5\log t$ by \cref{omegatasymptotics},
we have that
\[
\limsup_{t\to 0^+}\frac{1}{t^5}\mathfrak{m}(\Omega_t)=+\infty,~~\text{and}~~\limsup_{t\to 0^+}\frac{1}{t^{5-\epsilon}}\mathfrak{m}(\Omega_t)=0
\]
for all $\epsilon>0$.
Therefore, the geodesic dimension is $5$.
\end{proof}

\section{The \texorpdfstring{$\ell^1$}{lp}- and \texorpdfstring{$\ell^\infty$}{linfty}-Heisenberg groups}
\label{sec:l1linfty}

We will finally consider the $\ell^p$-Heisenberg group, where $p = 1$ or $p = \infty$. As the $\ell^1$-Heisenberg group and the $\ell^\infty$-Heisenberg group are isomorphic as metric measure spaces, it is sufficient to focus on the $\ell^1$-Heisenberg group.
In this case,
the Hamiltonian formalism pursued in \cref{sec:geometry} is more difficult to implement due to the fact that the $\ell^1$-sub-Finsler norm is not strictly convex. As we have seen, the dual angle $\theta^\circ$ is not always uniquely determined by $\theta$ if $p = 1$. We will adopt the approach proposed in \cite{duc} and \cite{breuillard--ledonne2013}, based on Busemann's isoperimetric inequality in the plane with an $\ell^1$-norm. In those works, from which we borrow the following result, the authors also studied the Heisenberg group equipped with polygonal sub-Finsler metrics. 

We say that a continuous path in $\R^2$ is an $\ell^\infty$-arc if it parametrises a subset of an $\ell^\infty$-sphere. There are three kinds of geodesics in the $\ell^1$-Heisenberg group.

\begin{proposition}[{\cite[Section 6.]{breuillard--ledonne2013}}]
\label{prop:geodesicsl1H}
    A horizontal curve $\gamma(t) = (x(t), y(t), z(t))$ joining the identity to $(x, y, z)$ is a minimising constant speed geodesic of the $\ell^1$-Heisenberg group if and only if, 
    \begin{enumerate}[label=\normalfont(\roman*)]
    \item when $0 \leq |z| \leq \frac{1}{2} |xy|$, the projection $(x(t), y(t))$ is a minimising constant speed geodesic for the $\ell^1$-norm on the plane, joining $(0, 0)$ to $(x, y)$ and covering a signed area $z$. These geodesics are never unique unless $\abs{z} = \frac{1}{2}\abs{x y}$;
    
    \item when $\frac{1}{2}|xy|\leq |z|\leq \max(|x|,|y|)^2-\frac{1}{2}|xy|$, the projection $(x(t), y(t))$ is an $\ell^\infty$-arc with $3$ edges, parametrised by constant speed (with respect to the $\ell^1$-length of $\R^2$), and covering a signed area $z$. These geodesics are unique;
    
    \item when $\max(|x|,|y|)^2-\frac{|xy|}{2} < |z|$, the projection $(x(t), y(t))$ is an $\ell^\infty$-arc with $4$ edges, parametrised by constant speed (with respect to the $\ell^1$-length of $\R^2$), and covering a signed area $z$. These geodesics are unique unless $x$ or $y$ vanish.
    \end{enumerate}
\end{proposition}

Plane curves of ``staircase type'' are examples of projection of geodesics of case (i) in \cref{prop:geodesicsl1H}. On the other hand, the geodesics of case (ii) and (iii) are projected onto arcs of squares, with the sides parallel to the $x$-axis and $y$-axis. We shall denote by $H_1$, $H_2$ and $H_3$ the regions of $\mathds{H}$ associated with cases (i), (ii), and (iii), respectively.

The $\ell^1$-Heisenberg group does not have negligible cut locus in the sense of \cref{def:negligiblecutlocus}. Indeed, we can write down explicitly the set $\mathcal{C}(e)$ of points joined from the identity by multiple minimising geodesics, and observe that it clearly has a non-zero measure:
\[
\mathcal{C}(e)=\left\{(x,y,z)\in\mathds{H}\mid |z|\leq \frac{1}{2} |xy|, \text{ or else }\abs{x y} = 0 \text{ and } \abs{z} > 0 \right\}.
\]

Moreover, the $\ell^1$-Heisenberg group is branching. It suffices, for example, to consider a small open neighbourhood $\Omega$ of $(x, y, z) \in \mathrm{int}(H_2)$, where $ | x | < y$ and $z > 0$. Then, a geodesic from the identity to a point in $\Omega$ is of type (ii). More specifically, it is an $\ell^\infty$-arc with three edges: the first one going in the direction $\partial_x$, the second in the direction $\partial_y$ and the third in the direction $-\partial_x$. They all coincide at least on a portion of their first edge, before branching (see \cref{def:branching}).

Building upon this last observation, it is possible to prove that the $\ell^1$-Heisenberg group does not satisfies the $\mathsf{MCP}(K, N)$, for any $K \in \mathds{R}$ and any $N \geq 1$.

\begin{theorem}\label{thm:ell1mcp}
    The $\ell^1$-Heisenberg group equipped with the Lebesgue measure does not satisfies the $\mathsf{MCP}(K,N)$ for all $K\in\R$ and $N\geq 1$.
\end{theorem}

\begin{proof}
As in the beginning of the proof of \cref{thm:MCPp<2}, it is enough to show that $\mathsf{MCP}(0, N)$ fails. We are going to find a Borel subset $\Omega$ of $\mathds{H}$ with $0 < \mathfrak{m}(\Omega) < +\infty$ such that $\mathfrak{m}(\Omega_t) 
= 0$ for some $t \in \ointerval{0}{1}$. By \cref{MCPimplies1sideBM}, this will disprove the measure contraction property $\mathsf{MCP}(0,N)$ for all $N\geq 1$.

As we observed previously, a point $h = (x, y, z) \in H_2$ such that $ |x| < y$ and $z > 0$ is connected to the identity by a unique geodesic $\gamma_h$. This geodesic is an $\ell^\infty$-arc consisting of three segments with directions $\partial_x$, $\partial_y$, and $-\partial_x$. We denote the lengths of these segments as $l_1(h)$, $l_2(h)$, and $l_3(h)$, respectively.

Let us fix a point $h_0 \in H_2$ with 
\[
\frac{l_1(h_0)}{l_2(h_0)}=\frac{l_3(h_0)}{l_2(h_0)}=\frac{1}{2}.
\]
Then there is a neighborhood $\Omega$ of $h_0$ in $H_2$ such that for all $h\in \Omega$,
\[\frac{l_1(h)}{l_2(h)},\frac{l_3(h)}{l_2(h)}\in\left(\frac{1}{4},\frac{3}{4}\right).\]
Since geodesics are constant speed,
we can see that for all $h\in \Omega$ and $t\in[0,\frac{1}{8}]$,
$\gamma_h(t)$ is in the $x$-axis of $\mathds{H}$.
In particular,
$\mathfrak{m}(\Omega_t)=0$ for $t\in[0,\frac{1}{8}]$.

\end{proof}

Lastly, we calculate the geodesic dimension of the $\ell^1$-Heisenberg group, which will be derived from the following observations.

Let $h = (x,y,z)\in H_1 \cap \{x,y>0\}$, and $\alpha:=x+y$. For $t\in(0,1)$, we define
\[
F_{t,\alpha} := \left\{(ut,(\alpha-u)t,v)\in \mathds{H} \mid u\in\interval{0}{\alpha}{},~|v|\leq \frac{u(\alpha-u)t^2}{2}\right\}\subseteq H_1.
\]

Note that $\alpha=\|(x,y)\|_1=\diff_1(0,h)$,
and therefore all minimising constant speed geodesic defined on $\interval{0}{1}$ joining $0$ to $h$ will pass through $F_{t,\alpha}$ at time $t$.
In other words,
\[
\left\{\gamma(t) \mid \gamma : \interval{0}{1} \to \mathds{H} \text{ minimising constant speed geodesic from $0$ to $h$}\right\} \subseteq F_{t,\alpha}.
\]
In the following lemma, we show that for $t$ sufficiently small, the inclusion above is, in fact, an equality.

\begin{figure}
    \centering
    \includegraphics[width=12cm]{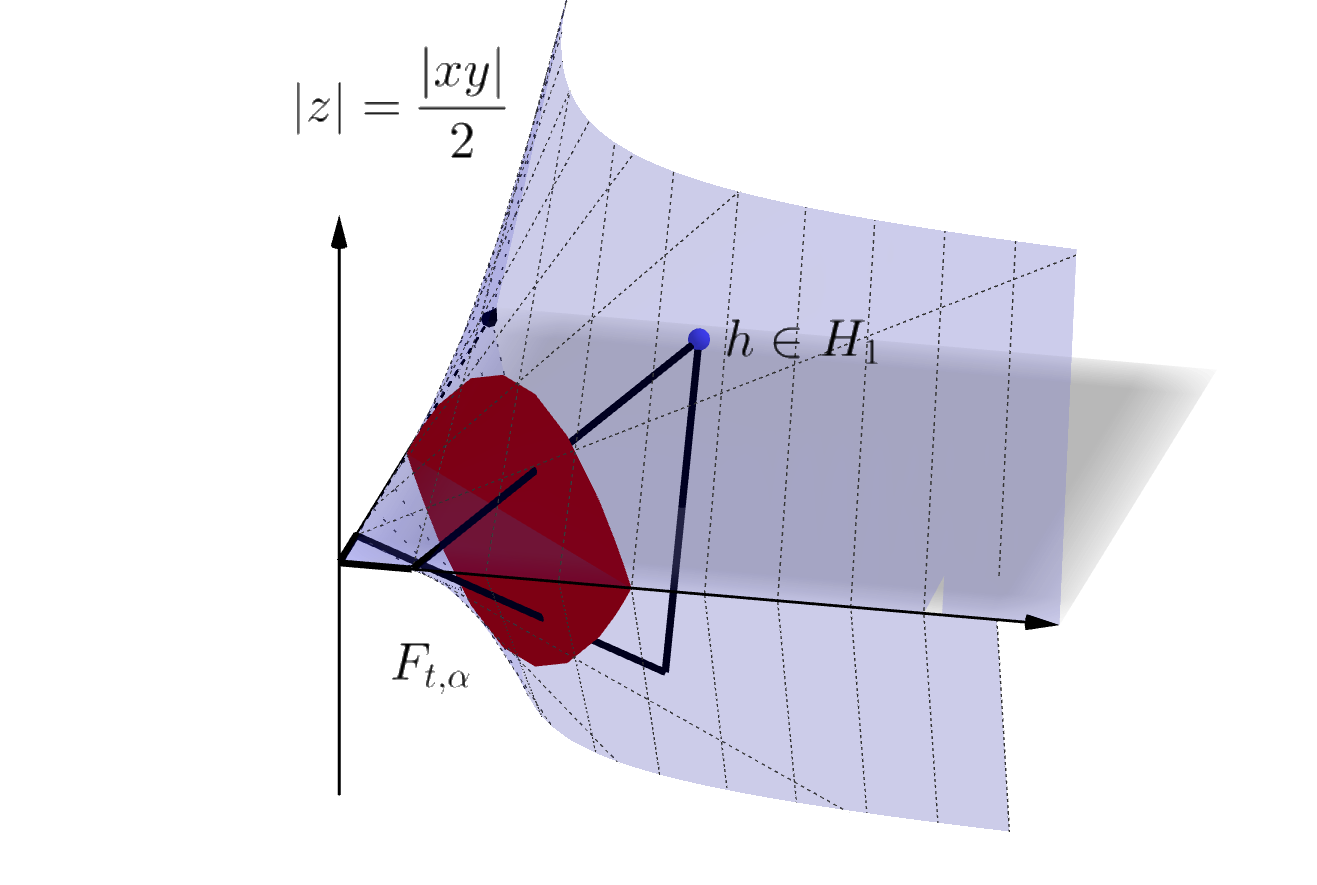}
    \caption{Geodesics pass through $F_{t,\alpha}$}
    \label{fig:Ftalpha}
\end{figure}

\begin{lemma}\label{lemma:ell1}
    For $h=(x,y,z)\in \mathrm{int}(H_1)\cap\{x,y>0\}$,
    define the number  $T=T(h)\in(0,1)$ by
    \[
    T(h) := \min\left(\frac{\min(x,y)}{\alpha},\frac{\frac{1}{2}xy-z}{x\alpha},\frac{\frac{1}{2}xy+z}{y\alpha}\right),
    \]
    where $\alpha:=x+y$. Then, for all $t\leq T$,
    \begin{equation}\label{equalityFth}
        \left\{\gamma(t) \mid \gamma : \interval{0}{1} \to \mathds{H} \text{ minimising constant speed geodesic from $0$ to $h$}\right\} = F_{t,\alpha}.
    \end{equation}
\end{lemma}

\begin{proof}
    We are going to establish the following claim, which implies the lemma:
    \begin{equation}
        \label{sublemma}
        \text{For all $t\leq T$ and all $g=(ut,(\alpha-u)t,v)\in F_{t,\alpha}$,
    $g^{-1}h$ is in $H_1$}.
    \end{equation}
    Indeed, by \cref{prop:geodesicsl1H} and left-invariance, \cref{sublemma}
    yields that there is a minimising geodesic $\gamma_2$ from $g$ to $h$ such that the projection to $xy$-plane is an $\ell^1$-geodesic.
    Moreover,
    by the definition of $F_{t,\alpha}$
    there is a minimising geodesic $\gamma_1$ from $0$ to $g$ with the same property.
    The concatenation of $\gamma_1$ and $\gamma_2$ is, after reparametrisation, a minimising constant speed geodesic defined on $\interval{0}{1}$ from $0$ to $h$.
    Since the choice of $g$ is arbitrary,
    this implies (\ref{equalityFth}).
    
    We now proceed to prove \cref{sublemma}. A direct computation shows that
        \[
        g^{-1}h=\left(x-ut,y-(\alpha-u)t,z-v+t\frac{x(\alpha-u)-yu}{2}\right).
        \]
        Therefore, we have that $g^{-1}h\in H_1$ if and only if
        \[
        \abs{z-v+t\frac{x(\alpha-u)-yu}{2}}\leq \frac{|x-ut||y-(\alpha-u)t|}{2}.
        \]
        Since $t\leq T\leq \frac{1}{\alpha}\min(x,y)$ and $u \in \interval{0}{\alpha}$,
        both $x-ut$ and $y-(\alpha-u)t$ are positive.
        Hence the desired inequality is equivalent to
        \begin{equation}\label{inequalityclaimF}
            -\frac{(x-ut)(y-(\alpha-u)t)}{2}\leq 
            z-v+t\frac{x(\alpha-u)-yu}{2}\leq 
            \frac{(x-ut)(y-(\alpha-u)t)}{2},
        \end{equation}
        for all $t\leq T$,
        $|v|\leq \frac{u(\alpha-u)t^2}{2}$ and $u\in[0,\alpha]$.
        Since $|v|\leq \frac{u(\alpha-u)t^2}{2}$, we have
        \[
        z-v+t\frac{x(\alpha-u)-yu}{2} \leq z+\frac{u(\alpha-u)t^2}{2}+t\frac{x(\alpha-u)-yu}{2},
        \]
        as well as
        \[
        z-v+t\frac{x(\alpha-u)-yu}{2} \geq z-\frac{u(\alpha-u)t^2}{2}+t\frac{x(\alpha-u)-yu}{2}.
        \]
        We set the functions $G_1(u)$ and $G_2(u)$ by
        \[
        G_1(u) := z-\frac{u(\alpha-u)t^2}{2}+t\frac{x(\alpha-u)-yu}{2}+\frac{(x-ut)(y-(\alpha-u)t)}{2}
        \]
        and
        \[
        G_2(u) := z+\frac{u(\alpha-u)t^2}{2}+t\frac{x(\alpha-u)-yu}{2}-\frac{(x-ut)(y-(\alpha-u)t)}{2}.
        \]
        Thus, the inequality \cref{inequalityclaimF} holds if $G_1(u) \geq 0$ and $G_2(u) \leq 0$, for all $u \in \interval{0}{\alpha}$. The function $G_1(u)$ (resp. $G_2(u)$) is linear in $u$, and attains its minimum at $u = \alpha$ (resp. its maximum $u = 0$). Therefore, we obtain
        \[
        G_1(u) \geq G_1(\alpha)=z+\frac{xy}{2}-ty\alpha\geq 0, \text{ and } G_2(u)\leq G_2(0)=z-\frac{xy}{2}+tx\alpha\leq 0,
        \]
        since, in the first case, $t\leq T\leq \frac{\frac{xy}{2}+z}{y\alpha}$, and since, in the second case, $t\leq T\leq \frac{\frac{xy}{2}-z}{x\alpha}$.
\end{proof}

We can now conclude. The geometric idea behind the following proof is illustrated in \cref{geoddimp1proof}.

\begin{figure}
  \centering
  \includegraphics[width=0.5\linewidth]{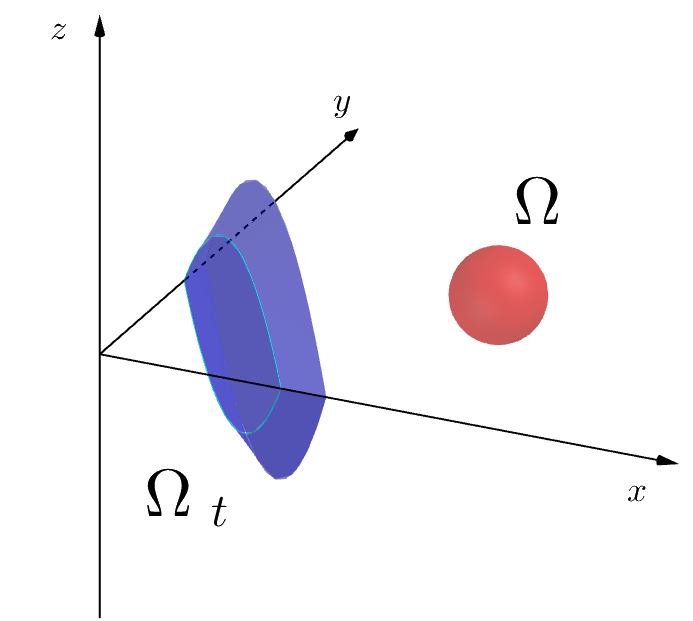}
  \caption{Any open set $\Omega$ contained in the region $H_1$ has a $t$-homothety $\Omega_t$ ``filling'' a section of $H_1$, for all $t$ small enough.}
  \label{geoddimp1proof}
\end{figure}

\begin{theorem}\label{thm:ell1geodim}
The geodesic dimension of the $\ell^1$-Heisenberg group equipped with the Lebesgue measure is $4$.
\end{theorem}

\begin{proof}
    Let $\Omega \subseteq \mathrm{int}(H_1)$ be a compact subset with $0 < \mathfrak{m}(\Omega) < +\infty$. We introduce
    \[
    T(\Omega):=\min\{T(h)\mid h\in \Omega\}, 
    \]
    and
    \[
    m=:\min\{x+y\mid (x,y,z)\in\Omega\}, \ M:=\max\{x+y\mid (x,y,z)\in \Omega\}.
    \]
    By the definition of $T(h)$ in \cref{lemma:ell1}, and since $\Omega$ is chosen to be compact, the constants $T(\Omega)$, $m$, and $M$ are both positive and finite. By \cref{lemma:ell1}, we find that
    for all $t\leq T(\Omega)$,
    \[
    \Omega_{t} = \bigcup_{\alpha\in[m,M]}F_{t,\alpha}=\left\{(ut,(\alpha-u)t,v)~\Big{|}~\alpha\in[m,M], u\in[0,\alpha],~|v|\leq \frac{u(\alpha-u)t^2}{2}\right\}.
    \] 
Hence, the volume of $\Omega_t$ is given by
\begin{align*}
    \mathfrak{m}(\Omega_t) ={}& \int_{\mathrm{Proj}(\Omega_t)} x y\diff x\diff y && \text{by Fubini's theorem} \\
    ={}& \int_{m}^{M} \int_0^\alpha u(\alpha-u)t^2 \abs{\det \diff \Psi(u, \alpha)} \diff u\diff \alpha && \text{by the change of variables}\\
    & && \Psi(u, \alpha) := (u t, (\alpha - u)t)\\
    ={}&t^4\int_{m}^{M} \int_0^\alpha u(\alpha-u)\diff u\diff \alpha =t^4\frac{M^4-m^4}{24}.
\end{align*}

In particular, we have that
\[
\limsup_{t\to 0^+}\frac{1}{t^4}\frac{\mathfrak{m}(\Omega_t)}{\mathfrak{m}(\Omega)} = \frac{M^4-m^4}{24 \mathfrak{m}(\Omega)} > 0,
\]
and thus from \cref{def:geod_dim}, we know that $C_4(e) = +\infty$, and $N_{geod}\leq 4$.

On the other hand, the geodesic dimension is no less than the Hausdorff dimension (see \cref{theo:NcurvNgeoNH}):
$N_{\mathrm{geo}}\geq \dim_{\mathcal{H}}\mathds{H}=4$. Therefore we have $N_{\mathrm{geo}}=4$.
\end{proof}



\begin{spacing}{0.9}
\printbibliography	
\end{spacing}

\end{document}

\bigskip

\begin{itemize}
    \item \href{https://www.springernature.com/gp/authors/campaigns/latex-author-support}{Springer Latex Template}
    \item \href{https://tex.stackexchange.com/questions/53513/hyperref-token-not-allowed}{Hyperref warning - Token not allowed in a PDF string}
\begin{verbatim}
\subsection{The classes \texorpdfstring{$\mathcal{L}(\gamma)$}{Lg}}
\end{verbatim}
    \item \href{https://www.austms.org.au/Rankings/AustMS_final_ranked.html}{Journal rankings}
    \item \href{https://dmargalit7.math.gatech.edu/tsr/Journals.pdf}{Journal rankings}
    \item \href{https://www.elsevier.com/authors/policies-and-guidelines/latex-instructions}{Elsevier Latex Instruction}
    \item \href{http://www.ams.org/publications/authors/tex/tex}{AMS Tex ressources}
    \item \href{https://www.overleaf.com/latex/templates/overleaf-keyboard-shortcuts/qykqfvmxdnjf}{Overleaf Keyboard Shortcuts}
    \item \href{https://tex.stackexchange.com/questions/365456/is-there-a-cheat-sheet-for-texstudio-keyboard-shortcuts}{TexStudio default shortcuts}
\end{itemize}